\documentclass[13pt,twoside,a4paper]{article}
\usepackage{amssymb,amsmath,amsfonts,amsthm}
\usepackage{mathtools} 
\usepackage{tkz-euclide}
\usepackage{graphicx}
\usepackage[T1]{fontenc}
\usepackage{wesa}
\usetikzlibrary{patterns}

\usepackage{mathrsfs}
\usepackage{bbm}
\usepackage[colorlinks, citecolor=blue,pagebackref,hypertexnames=false]{hyperref}

\allowdisplaybreaks
\theoremstyle{plain}
\newtheorem{thm}{Theorem}[section]
\newtheorem{prop}[thm]{Proposition}
\newtheorem{lem}[thm]{Lemma}

\theoremstyle{definition}
\newtheorem{defn}[thm]{Definition}
\newtheorem{remark}[thm]{Remark}

\begin{document}

\title{\bf\Large Maximal functions with twisted structures, distribution  inequality and applications}
\author{Ji Li, Chong-Wei Liang, Chaojie Wen and Qingyan Wu}
\date{}
\maketitle
\let\thefootnote\relax

\begin{abstract}
Motivated by the geometric reduction of Cauchy--Szeg\H{o} projections on quadratic surfaces of higher codimension \cite{NRS} and recent developments on the real-variable theory adapted to twisted multiparameter structures \cite{FLLWW}, we establish the Fefferman--Stein type distribution inequality relating the twisted area function and the twisted  non-tangential maximal function over $\mathbb{R}^{2m}$. By deploying a recursive integration-by-parts argument involving the twisted gradient and Laplacian, and constructing smooth, compactly supported weight functions to absorb cross-derivative errors, we obtain the required estimate. As an application, we prove the uniform $L^1$ boundedness of the twisted maximal function on the twisted atoms and complete the maximal function characterization of the twisted Hardy space.
\end{abstract}

\section{Introduction and statement of main results}
\setcounter{equation}{0}

A classical theme in harmonic analysis is the comparison between the non-tangential maximal function and the Lusin area integral. This equivalence is central to the real-variable theory of Hardy spaces and to the study of boundary value problems for partial differential equations. In the one-parameter setting, Fefferman and Stein \cite{FS72} established the corresponding good-$\lambda$ inequality. More recent developments include the work of Hofmann, Kenig, Mayboroda and Pipher \cite{HKMP} on square function and non-tangential maximal function estimates for the Dirichlet problem associated with non-symmetric elliptic operators.

The extension of this theory to multiparameter settings has been an important direction in harmonic analysis. As Stein emphasized in his 1998 survey \cite{St98}, multiparameter methods are closely connected with geometric problems arising in several complex variables. This point of view led to the development of the classical product theory; see, for example, Gundy--Stein \cite{GuS}, Chang--Fefferman \cite{CF80, CF85}, and Fefferman--Stein \cite{Feff-S}. It also motivated the later theory of flag singular integrals, initiated by M\"uller, Ricci, and Stein \cite{MRS} and further developed by Nagel, Ricci, Stein, and Wainger \cite{NRS, NRSW12, NRSW18}.

However, as observed by Nagel, Ricci, and Stein \cite{NRS}, the class of product operators is not stable under passage to quotient subgroups. Such quotient constructions arise naturally in the study of the Cauchy--Szeg\H{o} projection and the solution operators for $\bar{\partial}_b$ on higher-codimensional quadratic surfaces in $\mathbb{C}^n$. After passing to a quotient, the singularities of a product kernel become intertwined, producing operators that are singular along both coordinate hyperplanes and diagonals. This falls outside the scope of the classical multiparameter theories described above.

To treat this situation, the theory of twisted multiparameter singular integrals was initiated in \cite{FLLWW}. In that paper, a real-variable framework on $\mathbb{R}^{2m}$ was developed through the introduction of twisted tubes $T(\mathbf{x},\mathbf{r})$ and the associated tube maximal function. By classifying dyadic tubes into five geometric types and proving a type-preserving twisted covering lemma, \cite{FLLWW} introduced the twisted atomic Hardy space $H^1_{Tw,atom}(\mathbb{R}^{2m})$ and proved that it coincides with the area-function Hardy space $H^1_{area,\varphi}(\mathbb{R}^{2m})$.

In the present paper, we establish a Fefferman--Stein type good-$\lambda$ inequality in the twisted setting. Our argument is motivated by Merryfield's work on the polydisc \cite{M} and by recent work on the Shilov boundaries of finite-type domains \cite{Li_Pisa}, where one avoids the use of Fourier and group structures. The key point in our setting is that the twisted Poisson extension $U_f$ is harmonic in each block, that is, $\Delta_j U_f=0$ for $j=1,2,3$.

Using this property, we construct smooth tensor-valued auxiliary weights to absorb the cross-derivative terms arising in the integration-by-parts argument. This allows us to transfer the space-time derivatives $\nabla_1$, $\nabla_2$, and $\nabla_3$ from the main expression to the indicator of a proxy good set. Combining with the $L^2$ boundedness from the twisted Littlewood--Paley theory, this yields the desired good-$\lambda$ inequality and, as a consequence, the maximal-function characterization of the twisted Hardy space.

We first recall the underlying geometric structure induced by the quotient map established in \cite{FLLWW}. Let $\pi: \mathbb{R}^m \times \mathbb{R}^m \times \mathbb{R}^m \to \mathbb{R}^m \times \mathbb{R}^m$ be the projection defined by \begin{align}\label{proj}
    \pi(x_1, x_2, x_3) = (x_1+x_3, x_2+x_3).
\end{align}
A natural ball is the Cartesian product of three balls in Euclidean spaces $\mathbb{R}^m$. Thus a natural ball in the projecting space $\mathbb{R}^{2m}$ is the image of such a
  product under the projection   $\pi$. Motivated by this, we introduce the notion of a tube $ T(\mathbf{x},\mathbf{r})$
for $\mathbf{x}\in {\mathbb R}^{2m}$ and $  \mathbf{r} :=\left({r_1} ,
          {r_2} ,r_3 \right)\in \mathbb{  R}^3_+$ (see Section \ref{sec:2} for the definition). 
          
For any $\mathbf{x} \in \mathbb{R}^{2m}$, the translated tube is denoted by $T(\mathbf{x}, \mathbf{r}) := \mathbf{x} + T(\mathbf{0}, \mathbf{r})$, and its volume is denoted by $V_{\mathbf{r}} := |T(\mathbf{x}, \mathbf{r})|$. Based on this twisted tube system, for any aperture $\beta > 0$, we define the \textit{twisted non-tangential region} (or cone) of aperture $\beta$ at $\mathbf{x}$ by
\begin{align}
\Gamma^\beta(\mathbf{x}) := \left\{ (\mathbf{y}, \mathbf{r}) \in \mathbb{R}^{2m} \times \mathbb{R}^3_+ : \mathbf{y} \in T(\mathbf{x}, \beta\mathbf{r}) \right\},
\end{align}
where $\beta\mathbf{r} = (\beta r_1, \beta r_2, \beta r_3)$. When $\beta = 1$, we simply write $\Gamma(\mathbf{x})$.

We now introduce the test functions and maximal operators adapted to this setting.

\begin{defn}
Let $\mathbf{x}=(x_1,x_2)\in\mathbb R^{2m}$ and define
$$\phi(\mathbf{x}) := \left(\left(\phi^{(1)}\otimes\phi^{(2)}\right)*_3\phi^{(3)}\right)(\mathbf{x}) = \int_{\mathbb{R}^m} \phi^{(1)}(x_1 - u) \phi^{(2)}(x_2 - u) \phi^{(3)}(u) \, du,$$ where $\phi^{(j)} \in \mathcal{S}(\mathbb{R}^m)$ satisfies
\(
\int_{\mathbb{R}^m}\phi^{(j)}(x_j)\,dx_j=1,\, \text{for all}\,\,j=1,2,3.
\)
We denote by $\mathfrak{D}_{tw}(\mathbb{R}^{2m})$ the collection of all functions $\phi$ that satisfy the above condition.
\end{defn}

Let $\phi\in\mathfrak{D}_{tw}(\mathbb{R}^{2m})$ be given and let $f\in L^1(\mathbb{R}^{2m})$. The \textit{twisted non-tangential maximal function} of $f$ is defined by sweeping the twisted non-tangential region:
\begin{align*}
M^*_{\phi}(f)(\mathbf{x}) := \sup_{(\mathbf{y},\,\mathbf{r})\in\Gamma(\mathbf{x})}\left|\phi_{\mathbf{r}}*f\right|(\mathbf{y}),
\end{align*}
where $\phi_{\mathbf{r}}(\mathbf{x}) := \left(\left(\phi^{(1)}_{r_1}\otimes\phi^{(2)}_{r_2}\right)*_3\phi^{(3)}_{r_3}\right)(\mathbf{x})$, and for each $j=1,2,3$, the dilation is defined as $\phi^{(j)}_{r_j}(x) := r_j^{-m}\phi^{(j)}(x/r_j)$ for $x \in \mathbb{R}^m$.
\medskip

\begin{defn}
For $j=1,2,3$, let $Poi^{(j)}$ be the standard Poisson kernel on $\mathbb{R}^m$, that is
\begin{align*}
Poi^{(j)}(x_j) = \frac{c_m}{(1+|x_j|^2)^{\frac{1}{2}(m+1)}}.  
\end{align*}
We define the \textit{twisted Poisson kernel} via the push-forward along the fibers of $\pi$ as
\begin{align}\label{twistedPoisson r}
Poi_{twist}(\mathbf{x}) &:= \pi_{*}\left( Poi^{(1)}\otimes Poi^{(2)}\otimes Poi^{(3)}\right)(\mathbf{x})
= (Poi^{(1)}\otimes Poi^{(2)})*_3 Poi^{(3)}(\mathbf{x}).
\end{align}
For each $L^1(\mathbb{R}^{2m})$ function $f$, the \textit{twisted Poisson integral} (or extension) of $f$ is given by
\begin{align}
U_f(\mathbf{x},\,\mathbf{r}) := (Poi_{twist})_{\mathbf{r}}*f(\mathbf{x}).
\end{align}
\end{defn}

With the suitable twisted Poisson structure, we introduce the \textit{twisted non-tangential maximal function}.
\begin{defn}
Let $f\in L^1(\mathbb{R}^{2m})$ and let $\beta > 0$. We define the \textit{twisted non-tangential maximal function} of $f$ by
\begin{align*}
U_f^*(\mathbf{x}) := \sup_{(\mathbf{y},\,\mathbf{r})\in\Gamma^\beta(\mathbf{x})}\left| U_f(\mathbf{y},\,\mathbf{r})\right|,
\end{align*}
where $\Gamma^\beta(\mathbf{x})$ is the twisted non-tangential cone defined above. We define the corresponding \textit{twisted area function} by
\begin{align*}
S_{twist}(f)(\mathbf{x}) := \left( \iint_{\Gamma(\mathbf{x})} |r_1\nabla_1 r_2\nabla_2 r_3\nabla_3 U_f(\mathbf{y}, \mathbf{r})|^2 \frac{d\mathbf{y} d\mathbf{r}}{r_1 r_2 r_3 V_{\mathbf{r}}} \right)^{1/2},
\end{align*}
where $\nabla_j = (\nabla_{x_j},\partial_{r_j})$ for $j=1,2$, $\nabla_3 := (\nabla_{x_1}+\nabla_{x_2}, \partial_{r_3})$ and $d\mathbf{r} = dr_1 dr_2 dr_3$.
\end{defn}

The main result of this paper is the following Fefferman--Stein type good-$\lambda$ inequality, establishing the $L^1$ equivalence between the maximal function and the area function in the twisted setting.
\begin{thm}\label{thm:good_lambda}
There exist constants $C>0$ and $\beta>1$ such that for all $f \in C_0^\infty(\mathbb{R}^{2m})$ and all $\lambda>0$,
\begin{align*}
&\left|\{\mathbf{x} \in \mathbb{R}^{2m}: S_{twist}(f)(\mathbf{x})>\lambda\}\right| \\
&\leq C\left|\{\mathbf{x} \in \mathbb{R}^{2m}: U_f^*(\mathbf{x})>\lambda\}\right|+\frac{C}{\lambda^{2}}\int_{\{U_f^*(\mathbf{x}) \le \lambda\}} U_f^*(\mathbf{x})^{2} d\mathbf{x}.
\end{align*}
\end{thm}

As an application, we obtain the following characterization.
\begin{thm}\label{secondmainthm}
The area function Hardy space $H^1_{area, \varphi}(\mathbb{R}^{2m})$ introduced in \cite{FLLWW} can be characterized equivalently by the non-tangential maximal function $U_f^*$. That is, $H^1_{area, \varphi}(\mathbb{R}^{2m})$
coincides with $H^1_{max}(\mathbb{R}^{2m})
=\{f\in L^1(\mathbb{R}^{2m}): U_f^*\in L^1(\mathbb{R}^{2m})\}$, equipped with the norm  $\|f\|_{H^1_{max}(\mathbb{R}^{2m})}$ $=\|U_f^*\|_{L^1(\mathbb{R}^{2m})} $.
\end{thm}
We note that Theorem \ref{thm:good_lambda} implies directly that 
$\|S_{twist}(f)\|_{L^1(\mathbb{R}^{2m})} \lesssim \|U_f^*\|_{L^1(\mathbb{R}^{2m})}$. The reverse direction $\|U_f^*\|_{L^1(\mathbb{R}^{2m})} \lesssim \|S_{twist}(f)\|_{L^1(\mathbb{R}^{2m})} $ relies on the atomic decomposition of the function $f$ under the condition that $\|S_{twist}(f)\|_{L^1(\mathbb{R}^{2m})}<\infty$ and on the argument that the $L^1$ norm of twisted non-tangential maximal function acting on every twisted atom is uniformly bounded.

As a further application, Theorem \ref{thm:good_lambda} provides the tool to transfer endpoint estimates from the maximal function to the area function. Since the non-tangential maximal function $U^*$ is dominated by the tube maximal function (see Section \ref{sec:2} for its property), $U^*$ is bounded from $L\log L(\mathbb{R}^{2m})$ to the hyper-weak $L^1$ space, the good-$\lambda$ inequality  implies that the twisted area function $S_{twist}$ inherits this   same endpoint boundedness (we refer to \cite{CLLP} and \cite{HLLW2} for this argument in the tensor product and the flag setting, respectively). To be more explicit, we have
\begin{prop}
    For every $\lambda>0$, and for $f$ with $\int_{\mathbb{R}^{2m}} |f(\mathbf{x})| \log(e+ |f(\mathbf{x})|) d\mathbf{x}<\infty$, 
    \begin{align*}
        \left|\{\mathbf{x}\in\mathbb{R}^{2m}: S_{twist}(f)(\mathbf{x})>\lambda\}\right|
        \lesssim \int_{\mathbb{R}^{2m}} {|f(\mathbf{x})|\over \lambda} \log\bigg(e+ {|f(\mathbf{x})|\over \lambda}\bigg) d\mathbf{x}.
    \end{align*}
\end{prop}

Building upon these estimates, one can develop an atomic decomposition specifically tailored for functions in $L\log L(\mathbb{R}^{2m})$. This structural decomposition will, in turn, lead directly to weak endpoint estimates for twisted Lusin area function $S_{area,  \varphi}(f)(\mathbf{x})$ defined via Schwartz functions $\varphi$ (not just Poisson kernel), twisted Calder\'on--Zygmund operators, extending the work of \cite{Feff-S}. Since this requires a substantial separate technical framework, these endpoint estimates will be developed in a subsequent paper.

This paper is organized as follows. In Section \ref{sec:2}, we collect the necessary preliminaries and prove several auxiliary lemmas, including the geometry of the projected space, the multi-harmonicity of the twisted Poisson extension and the relevant product rule identities. In Section \ref{sec:3}, we prove the twisted Fefferman--Stein good-$\lambda$ inequality. The proof is based on a recursive integration-by-parts argument together with the construction of smooth tensor-valued auxiliary weights. In Section \ref{sec:4}, we use the twisted atomic decomposition to prove the uniform $L^1(\mathbb{R}^{2m})$ boundedness of the non-tangential maximal function on twisted atoms. This gives the reverse inequality and completes the proof of Theorem \ref{secondmainthm}.


\section{Preliminaries and auxiliary lemmas}\label{sec:2}
 \subsection{Tube structures and tube maximal function}  We recall the geometric structure induced by the projection (\ref{proj}) and the mapping property of the maximal function adapted to this geometry in this section.
 
 A natural ball in $
       \mathbb{R}^{ m} \times   \mathbb{R}^{ m}\times   \mathbb{R}^{ m}
$
  is the translation of 
  \begin{equation}\label{eq:product-balls}
  \tilde{{B}} (\mathbf{0} ,\mathbf{r}):={B}_1(\mathbf{0}_1,r_1)\times  {B}_2(\mathbf{0}_2,r_2)\times  {B}_3(\mathbf{0}_3, {r}_3),
  \end{equation}  for $\mathbf{r} :=\left({r_1} ,
          {r_2} ,r_3 \right)\in \mathbb{  R}^3_+ $,    the product of three balls, respectively. A natural ball in
          $ 
       \mathbb{R}^{2 m}  $ is its image  
    under the projection   $\pi$. 
 It is direct  to see that the image of $\tilde{{B}} (\mathbf{0} ,\mathbf{r})$ under $  {\pi}$ is a
 hexagon,
   which is basically the {\it rectangle} $B_a(\mathbf{0}_1)\times B_b(\mathbf{0}_2)\subset\mathbb{R}^{2m}$ or the {\it parallelogram with base on the first direction}:
       \begin{equation*}
          P^{(first)}_{a,b }:=\{(x_1,x_2)\in \mathbb{R}^{2m}: \, |x_1-x_2|<a, |x_2|<b\},
       \end{equation*} 
or the {\it parallelogram with base on the second direction}:       
\begin{align*}
     P^{(second)}_{a,b }:=\{(x_1,x_2)\in \mathbb{R}^{2m}: \,|x_2-x_1|<a,\,|x_1|<b  \}
\end{align*}
for some $a>0  , b>0 $.
 So, if $r_1,r_2\geq r_3$, we define
$
     T(\mathbf{0},\mathbf{r}):=B_{r_1}(\mathbf{0}_1)\times B_{r_2}(\mathbf{0}_2)
    $ which is the standard rectangle;\,and if  $r_1,r_3\geq r_2$, we simply define
     $T(\mathbf{0},\mathbf{r}): =   P^{(first)}_{r_1,\,r_3 }$ in Figure 1(a);
  \,and if  $r_2,r_3\geq r_1$, we define
     $
T(\mathbf{0},\mathbf{r}): =   P^{(second)}_{r_2,\,r_3 }
  $ in Figure 1(b).

\begin{tikzpicture}[
    thick,
    >=stealth,
    axis/.style={->, thin, gray},
    label/.style={font=\small}
]

    
    
    

\begin{scope}[shift={(7,6)}]
    \draw[axis] (-3.0,0) -- (3.0,0);
    \draw[axis] (0,-2.5) -- (0,2.5);
    
    \draw (2.4, 0.4) -- (-1.6, 0.4) -- (-2.4, -0.4) -- (1.6, -0.4) -- cycle;
    
    \draw (2.4, 1.6) -- (1.2, 1.6) -- (-2.4, -1.6) -- (-1.2, -1.6) -- cycle;

    \node[label] at (0, -3) {Figure 1(a): $r_1, r_3 \geqslant r_2$};
\end{scope}

\begin{scope}[shift={(15,6)}]
    \draw[axis] (-3.0,0) -- (3.0,0);
    \draw[axis] (0,-2.5) -- (0,2.5);
    
    \draw (-0.4, 1.6) -- (0.4, 2.4) -- (0.4, -1.6) -- (-0.4, -2.4) -- cycle;
    
    \draw (-1.5, -1.0) -- (1.5, 2.2) -- (1.5, 1.0) -- (-1.5, -2.2) -- cycle;

    \node[label] at (0, -3) {Figure 1(b): $r_2, r_3 \geqslant r_1$};
\end{scope}
\end{tikzpicture}

 For $\mathbf{x}\in \mathbb{R}^{2m}$, set $ T(\mathbf{x},\mathbf{r}):= \mathbf{x}+T(\mathbf{0},\mathbf{r})$.
    We call $ T(\mathbf{x},\mathbf{r})$ a {\it tube}. Its volume is
  \begin{equation}\label{eq:tube-volume}
    | T(\mathbf{x},\mathbf{r})|\simeq  \left\{   \begin{array}{ll} r_1^m\cdot r_2^m, \qquad & {\rm if}\quad r_1,\,r_2\geq r_3, \\
    r^m_1\cdot r^m_3, \qquad & {\rm if}\quad r_1,\,r_3\geq r_2, \\
  r_2^m \cdot r_3^m,\qquad & {\rm if}\quad r_2,\,r_3\geq r_1.   \end{array} \right.
  \end{equation}   

   With these tube structures, we have the following geometric proposition, which is straightforward.   
\begin{prop}\label{prop:tube-pi}
For all $\mathbf{r}  \in \mathbb{
R}^3_+$, $ T(\mathbf{x},\mathbf{r}/2)\subset    \pi  (\tilde{{B}} (\mathbf{x} ,\mathbf{r}) )\subset T(\mathbf{x},2\mathbf{r}).$  
  \end{prop}
  Recall that the {\it tube maximal function} is given by
\begin{equation*}
   M_{tube}  (f)(\mathbf{x})=\sup_{\mathbf{r}\in \mathbb{R}^3_+} \frac 1{| T(\mathbf{x},\mathbf{r})|}\int_{  T(\mathbf{x},\mathbf{r})}|f (\mathbf{y})|d\mathbf{y},
\end{equation*}
which is an $L^p$-bounded operator for all $1<p<\infty$.
\begin{thm}[\cite{FLLWW}]\label{thm:maximal} For $1<p<\infty$,    tube  maximal function  $M_{tube}  $ is bounded from $L^p({\mathbb R}^{2m})$ to $L^p({\mathbb R}^{2m})$.
 \end{thm}

\subsection{Miscellaneous types of dyadic rectangles/ tubes}
    There are five types of  dyadic rectangles/tubes.
    
   $\bullet$ A
       {\it type ${\rm I}$  dyadic  rectangle} is the   translation   of the standard  dyadic  rectangle  $[0,2^{j_1})^m\times[0,2^{j_2})^m $ under an element of the lattice
 $
   2^{  {j}_1  } \mathbb{Z}^{m } \times  2^{ j_2   }   \mathbb{Z}^m
$; 

$\bullet$ A   {\it type ${\rm I\!I}$  dyadic  rectangle}  is the  translation    of  the  slant  dyadic  rectangle (based on the first direction)
              \begin{equation}\label{eq:slant1}
     I\times_t J:=    \left\{ \mathbf{x} \in  \mathbb{R}^{2m}:\, \mathbf{x}_1-\mathbf{x}_2\in I,\mathbf{x}_2\in J\right\} ,\quad  
     \end{equation} 
     with $$I=[0,2^{j_1})^m, J=[0,2^{j_2})^m {\rm\ \ \  and\ \ \ }
     j_1\leq j_2
     $$
       under an element of the lattice
  \begin{equation*}
   2^{  {j}_1  } \mathbb{Z}^{m } \times_t  2^{ j_2   }   \mathbb{Z}^m:=\{(\mathbf{n}_1+\mathbf{n}_2, \mathbf{n}_2):\mathbf{n}_1 \in 2^{  {j}_1  } \mathbb{Z}^{m }, \,  \mathbf{n}_2\in  2^{ j_2   }   \mathbb{Z}^m \};    
 \end{equation*}

$\bullet$  A {\it type ${\rm I\!I\!I}$  dyadic  rectangle} is defined to be the same structure as the {\it type ${\rm I\!I}$  dyadic  rectangle} but with the constraint that $j_1>j_2.$

$\bullet$  A {\it type ${\rm I\!V}$  dyadic  rectangle} is the  translation    of  the  slant  dyadic  rectangle (based on the second direction)
              \begin{equation}\label{eq:slant11}
     I\,{}_t\times J:=    \left\{ \mathbf{x} \in  \mathbb{R}^{2m}:\,  \mathbf{x}_1 \in I, \mathbf{x}_2-\mathbf{x}_1 \in J\right\}   
     \end{equation} 
     with $$I=[0,2^{j_1})^m, J=[0,2^{j_2})^m {\rm \ \ \ and \ \ \ }
     j_1\leq j_2
     $$
       under an element of the lattice
  \begin{equation*}
   2^{  {j}_1  } \mathbb{Z}^{m } \,{}_t\times  2^{ j_2   }   \mathbb{Z}^m:=\{(\mathbf{n}_1, \mathbf{n}_1+\mathbf{n}_2):\,\mathbf{n}_1 \in 2^{  {j}_1  } \mathbb{Z}^{m },\,  \mathbf{n}_2\in  2^{ j_2   }   \mathbb{Z}^m \}.    
 \end{equation*} 

 $\bullet$  A {\it type ${\rm V}$  dyadic  rectangle} is defined to be the same structure as the {\it type ${\rm I\!V}$  dyadic  rectangle} but with the constraint that $j_1>j_2.$

  \begin{remark} A {\it  dyadic  rectangle   of scale $\mathbf{j}\in \mathbb{Z}^3$} is a type ${\rm I}$ rectangle as a translation of $[0,2^{j_1})^m\times[0,2^{j_2})^m $ if $j_1, j_2\geq  j_3  $; it is a type ${\rm I\!I}$ / ${\rm I\!I\!I}$ rectangle as a translation of
    $[0,2^{j_1})^m\times_t [0,2^{j_3})^m $ if    $ \ j_1 , j_3 >  j_2$;  it is a type ${\rm I\!V}$ / ${\rm V}$ rectangle as a translation of $[0,2^{j_2})^m_t\times[0,2^{j_3})^m $ if 
         $ \ j_2 , j_3 >  j_1$. Now denote by
  $ \mathfrak{R }_{\mathbf{j}}  
$ the set of  rectangles of scale $\mathbf{j} $
 and by $
   \mathfrak{R }:=\bigcup_{\mathbf{j}\in \mathbb{Z} ^3}\mathfrak{R }_{\mathbf{j}}.
$
  the set of all rectangles.  They constitute a partition $  \mathbb{R}^{2m} $ for each $\mathbf{j}$.
  \end{remark}

Regarding the five different types of rectangles/tubes, we consider the maximal ones in $\Omega$ and the related maximal functions.  
\begin{itemize}
\item \textbf{Standard maximal rectangles}:
For an open set $\Omega$ in $ \mathbb{R}^{2m}$ with finite measure, let $m^{\rm I}_1(\Omega)$ be the class of all dyadic tubes of type ${\rm I}$ (the standard rectangles), $I\times J\subset\Omega$, which are maximal in the direction along the $x$-axis. Dyadic tubes in $m^{\rm I}_2(\Omega)$ are maximal along the $y$-axis.

    \item 
\textbf{Maximal tubes of type ${\rm I\!I}$}:
For an open set $\Omega$ in $ \mathbb{R}^{2m}$ with finite measure, let $m^{\rm I\!I}_1(\Omega)$ be the class of all dyadic tubes of type ${\rm I\!I}$, $I\times_t J\subset\Omega$, which are maximal in the direction along the $x$-axis. Dyadic tubes of type ${\rm I\!I}$ in $m^{\rm I\!I}_2(\Omega)$ are maximal in the direction of $(1,1)$. Consider the tube maximal function based on the first direction, which is defined by
\begin{align*}
    M^{\rm I\!I}_{tube}(f)(\mathbf{x}):=\sup_{\mathbf{x}\in I\times_t J}\frac{1}{|I\times_t J|}\int_{I\times_t J}|f(\mathbf{y})|\,d\mathbf{y},
\end{align*}
for suitable function $f$. Note that the tube maximal function based on the first direction is dominated by the tube maximal function $M_{tube}$, therefore $ M^{\rm I\!I}_{tube}$ is an $L^p(\mathbb{R}^{2m})$ bounded operator as well.

    \item 
\textbf{Maximal tubes of type ${\rm I\!I\!I}$, ${\rm I\!V}$ and ${\rm V}$}:
Similarly, we can define $m^{j}_1(\Omega)$, $m^{j}_2(\Omega)$ and
$M^{j}_{tube}(f)(\mathbf{x})$ for $j= {\rm I\!I\!I}, {\rm I\!V}, {\rm V}$. 
\end{itemize}

\subsection{Covering Lemma adapted to the miscellaneous dyadic rectangles}
We now present the covering lemmas. The first one is the standard one due to Journ\'e and Pipher \cite{J,JP}.
We also recall the covering lemmas in the twisted setting \cite{FLLWW}.
\begin{lem}[\cite{FLLWW}]\label{lem:Journe standard}
    Let $\Omega$
be an open subset of $ \mathbb{R}^{2m}$ with finite measure and $\kappa>0$. Then for each $I\times J\in m^{\rm I}_2(\Omega)$, there is an $\widehat{I}$ containing $I$ such that
\begin{align}
 \sum_{R= I \times J\in m^{\rm I}_2(\Omega)} |R|  \left(\frac {\ell (I)}{\ell (\widehat{I} )}\right)^\kappa \leq    C     |\Omega| 
 \end{align}
 for some constant $C$ independent of $\Omega$. Likewise, for each $I\times J\in m^{I}_1(\Omega)$, there is an $\widehat{J}$ containing $J$ such that
 \begin{align}
\sum_{R= I \times J\in m^{\rm I}_1(\Omega)} |R| \left(\frac {\ell (J)}{\ell (\widehat{J} )}\right)^\kappa \leq    C     |\Omega|.
 \end{align}
\end{lem}

\begin{lem}[\cite{FLLWW}]~\label{lem:Journe 1}
    Let $\Omega$
be an open subset of $ \mathbb{R}^{2m}$ with finite measure and $\kappa>0$. Then for each $I\times_t J\in m^{\rm I\!I}_2(\Omega)$, there is an $\widehat{I}$ containing $I$ such that
\begin{align}\label{Journe 1 I hat}
 \sum_{R= I \times_t J\in m^{\rm I\!I}_2(\Omega)} |R|  \left(\frac {\ell (I)}{\ell (\widehat{I} )}\right)^\kappa \leq    C     |\Omega|  
 \end{align}
 for some constant $C$ independent of $\Omega$. Likewise, for each $I\times_t J\in m^{\rm I\!I}_1(\Omega)$, there is an $\widehat{J}$ containing $J$ such that
 \begin{align}\label{Journe 1 J hat}
\sum_{R= I \times_t J\in m^{\rm I\!I}_1(\Omega)} |R| \left(\frac {\ell (J)}{\ell (\widehat{J} )}\right)^\kappa \leq    C     |\Omega|.
 \end{align}

 We also have the similar version for tubes in type ${\rm I\!I\!I}$.
\end{lem}

\begin{lem}[\cite{FLLWW}]\label{lem:Journe 2}
    Let $\Omega$
be an open subset of $ \mathbb{R}^{2m}$ with finite measure and $\kappa>0$. Then for each $I \,_t\times J\in m^{\rm I\!V}_2(\Omega)$, there is an $\widehat{I}$ containing $I$ such that
\begin{align}
 \sum_{R= I\,_t \times J\in m^{\rm I\!V}_2(\Omega)} |R| \cdot \left(\frac {\ell (I)}{\ell (\widehat{I} )}\right)^\kappa \leq    C     |\Omega|  
 \end{align}
 for some constant $C$ independent of $\Omega$. Likewise, for each $I\,_t\times J\in m^{\rm I\!V}_1(\Omega)$, there is an $\widehat{J}$ containing $J$ such that
 \begin{align}
\sum_{R= I\,_t \times J\in m^{\rm I\!V}_1(\Omega)} |R|\cdot\left(\frac {\ell (J)}{\ell (\widehat{J} )}\right)^\kappa \leq    C     |\Omega|.
 \end{align}
  We also have the similar version for tubes in type ${\rm V}$.
\end{lem}
\subsection{Twisted differential operators}
In this section, we establish the fundamental differential operators and prove the multi-harmonicity of the twisted Poisson extension, which will be used in Section \ref{sec:3}.

Write $ \mathbf{x}_j=({x}^1_j,\ldots,{x}^m_j)\in\mathbb{R}^{ m}$. For each $j \in \{1, 2\}$, let $\Delta_{x_j}$ denote the standard spatial Laplacian acting on the $j$-th component of the variables, and let $\nabla_{x_j}$ be the corresponding spatial gradient. That is,
\[
\Delta_{x_j} := \sum_{k=1}^{m} \frac{ \partial^2}{\partial (x_{j}^{k})^2}, \quad\text{and}\quad \nabla_{x_j} := (\partial_{x_j^1},\ldots, \partial_{x_j^m}).
\]
Define the full space-time Laplacian and gradient for $j=1,2$ by
\[
\Delta_j := \Delta_{x_j} + \partial_{r_j}^2, \quad\text{and}\quad  \nabla_j := (\nabla_{x_j}, \partial_{r_j});
\]
which, for $j=3$, we define
\[
\Delta_3 := \Delta_{twist} + \partial_{r_3}^2, \quad\text{and}\quad  \nabla_3 := (\nabla_{twist}, \partial_{r_3}),
\]
where
\[
\Delta_{twist} := \sum_{k=1}^{m} \left(\frac \partial {\partial x^k_{1} }+\frac \partial {\partial x^k_{2} }\right)^2 , \quad\text{and}\quad  \nabla_{twist} :=\nabla_{x_1}+\nabla_{x_2} .
\]

\begin{lem}
\label{lem:multi_harmonic}
Let $U_f(\mathbf{y}, \mathbf{r}) = (Poi_{twist})_{\mathbf{r}} * f(\mathbf{y})$ be the twisted Poisson integral of $f \in L^1(\mathbb{R}^{2m})$, where $\mathbf{r} = (r_1, r_2, r_3) \in \mathbb{R}_+^3$. Then $U_f$ is multi-harmonic, meaning $\Delta_j U_f = 0$ for each $j=1,2,3$. Consequently, we have the identity:
\begin{equation}\label{eq:harmonic_twist}
\Delta_j \left( U_f(\mathbf{y}, \mathbf{r})^2 \right) = 2|\nabla_j U_f(\mathbf{y}, \mathbf{r})|^2.
\end{equation}
\end{lem}

\begin{proof}
By definition, the twisted Poisson kernel $Poi_{twist}(\mathbf{x})$ is constructed via the push-forward (or twisted convolution) of the individual Euclidean Poisson kernels $Poi^{(j)}_{r_j}$. Since each $Poi^{(j)}_{r_j}$ is the integral kernel of the Poisson semigroup $e^{-r_j \sqrt{-\Delta_{x_j}}}$, it acts as the fundamental solution to the Laplace equation in the upper half-space $\mathbb{R}^m \times \mathbb{R}_+$. 

Specifically, the spectral theorem dictates that the Poisson semigroup satisfies:
\[
\partial_{r_j}^2 \left( e^{-r_j \sqrt{-\Delta_{x_j}}} f \right) = - \Delta_{x_j} \left( e^{-r_j \sqrt{-\Delta_{x_j}}} f \right).
\]
Rearranging this yields
\[
\left(\Delta_{x_j} + \partial_{r_j}^2\right) \left( e^{-r_j \sqrt{-\Delta_{x_j}}} f \right) = 0.
\]
Since the full extension $U_f(\mathbf{y}, \mathbf{r})$ is generated component-wise, the differential operators in the $j$-th variables commute with the convolution in the other variables. Thus, for $j=1,2$, $U_f$ satisfies the Laplace equation, that is
\[
\Big(\Delta_{x_j}+\partial_{r_j}^2\Big) U_f(\mathbf{y}, \mathbf{r}) = 0;
\]
For $j=3$, the operator is defined as $\Delta_3 := \Delta_{\text{twist}} + \partial_{r_3}^2 $, where $\Delta_{twist} = (\nabla_{x_1} + \nabla_{x_2})^2$. To see that $\Delta_3 U_f = 0$, we examine the Fourier representation. Recall that the twisted Poisson kernel is defined as 
\begin{align*}
    (Poi_{twist})_\mathbf{r}(\mathbf{y})=(Poi^{(1)}_{r_1}\otimes Poi^{(2)}_{r_2})*_3 Poi^{(3)}_{r_3}(\mathbf{y}).
\end{align*}

For $U_f(\mathbf{y},\,\mathbf{r}) = (Poi_{twist})_{\mathbf{r}}*f(\mathbf{y})$, we take the Fourier transform in the spatial variables $\mathbf{y} = (y_1, y_2)$ and obtain
\begin{align}\label{FourierofPoi}
\mathcal{F}_y U_f(\xi_1,\xi_2,\mathbf{r})=e^{-r_1|\xi_1|-r_2|\xi_2|-r_3|\xi_1+\xi_2|}\hat f(\xi_1,\xi_2).
\end{align}
Now $\Delta_{twist}$ corresponds to multiplication by $-|\xi_1 + \xi_2|^2$ in the Fourier domain. Hence,
\[\mathcal{F}_y[\Delta_3 U_f](\xi_1, \xi_2, \mathbf{r}) = (-|\xi_1 + \xi_2|^2 + \partial_{r_3}^2) \mathcal{F}_y[U_f](\xi_1, \xi_2, \mathbf{r}).\]
Since $\partial_{r_3}^2(e^{-r_3 |\xi_1 + \xi_2|}) = |\xi_1 + \xi_2|^2 e^{-r_3 |\xi_1 + \xi_2|}$, the right-hand side vanishes identically. Taking the inverse Fourier transform yields $\Delta_3 U_f =0$. Consequently, $U_f$ is multi-harmonic, i.e., 
$\Delta_j U_f=0$ for each $j=1,2,3$.
\smallskip

To derive the gradient identity, we evaluate the action of the full Laplacian $\Delta_j$ on the square of the harmonic function $U_f(\mathbf{y}, \mathbf{r})^2$.\\
For $j=1,2$, apply the standard Leibniz product rule for the second derivative to both spatial variables and the scale parameter $r_j$:
\begin{align*}
\Delta_j \left( U_f^2 \right) &= \left( \Delta_{x_j} + \partial_{r_j}^2 \right) (U_f^2) 
= \sum_{k} \partial_{y_{j,k}}^2 (U_f^2) + \partial_{r_j}^2 (U_f^2) \\
&= \sum_{k} \left[ 2 U_f (\partial_{y_{j,k}}^2 U_f) + 2 (\partial_{y_{j,k}} U_f)^2 \right] + \left[ 2 U_f (\partial_{r_j}^2 U_f) + 2 (\partial_{r_j} U_f)^2 \right].
\end{align*}
We regroup the terms by factoring out $2U_f$ from the second derivatives and collecting the squared first derivatives:
\begin{align*}
\Delta_j \left( U_f^2 \right) &= 2 U_f \left( \sum_{k} \partial_{y_{j,k}}^2 U_f + \partial_{r_j}^2 U_f \right) + 2 \left( \sum_{k} (\partial_{y_{j,k}} U_f)^2 + (\partial_{r_j} U_f)^2 \right) \\
&= 2 U_f \big( \Delta_{x_j} U_f + \partial_{r_j}^2 U_f \big) + 2 \big( |\nabla_{x_j} U_f|^2 + |\partial_{r_j} U_f|^2 \big) \\
&= 2 U_f \cdot \Delta_j U_f  + 2 |\nabla_j U_f|^2.
\end{align*}
Since $U_f$ is harmonic with respect to the $j$-th block ($\Delta_j U_f = 0$),  the first term vanishes entirely, leaving:
\[
\Delta_j \left( U_f^2 \right) = 2 |\nabla_j U_f|^2, \qquad j=1,2.
\]

For $j=3$, a similar computation gives
\begin{align*}
    \Delta_3 (U_f^2)
    &= 2U_f\Bigg(\sum_k \Big(\partial_{y_1,k} + \partial_{y_2,k}\Big)^2 U_f + \partial_{r_3}^2 U_f\Bigg) + 2\Bigg(\sum_k \big(\partial_{y_1,k} U_f + \partial_{y_2,k} U_f\big)^2 + (\partial_{r_3} U_f)^2\Bigg)\\
    &= 2U_f \cdot \Delta_3 U_f + 2|\nabla_3 U_f|^2\\
    &= 2|\nabla_3 U_f|^2.
\end{align*}
The proof is complete.
\end{proof}

\section{The recursive argument and the twisted good-$\lambda$ inequality}\label{sec:3}
\setcounter{equation}{0}

We now begin the proof of Theorem \ref{thm:good_lambda}.

\subsection{Setup and the separation by the Poisson extension}

We first establish the following auxiliary result.
\begin{lem}\label{lem:Poisson-tube-max}
Fix $\beta>0$. Then there exists a constant $C_0=C_0(\beta,m)>0$ such that for every $h\in L^1_{\mathrm{loc}}(\mathbb{R}^{2m})$, every $\mathbf{x},\mathbf{y}\in \mathbb{R}^{2m}$, and every $\mathbf{r}=(r_1,r_2,r_3)\in \mathbb{R}_+^3$ with
\(
\mathbf{y}\in T(\mathbf{x},\beta\mathbf{r}),
\)
one has
\begin{align}\label{P by M}
\left|(Poi_{twist})_{\mathbf r}*h(\mathbf y)\right|
\le C_0\, M_{tube}(h)(\mathbf x).
\end{align}
\end{lem}

\begin{proof}
Since $(Poi_{twist})_{\mathbf r}$ is nonnegative, it is enough to prove \eqref{P by M} for $|h|$.
We write
\begin{align}
(Poi_{twist})_{\mathbf r}*|h|(\mathbf y)
&=
\int_{\mathbb R^{2m}} |h(\mathbf z)|
\int_{\mathbb R^m}
Poi^{(1)}_{r_1}(y_1-z_1-u)\,
Poi^{(2)}_{r_2}(y_2-z_2-u)\,
Poi^{(3)}_{r_3}(u)\,du\,d\mathbf z.
\label{eq:twisted-conv-start}
\end{align}

We first record a dyadic domination for the Euclidean Poisson kernel.

\medskip
\noindent
\textit{Claim:} There exists a constant $C=C(m)>0$ such that for every $a>0$ and every $v\in \mathbb R^m$,
\begin{equation}\label{eq:1d-poisson-dyadic}
Poi_a(v)\le
C\sum_{\nu=0}^\infty
2^{-\nu}\,
\frac{\chi_{B(0,2^\nu a)}(v)}{|B(0,2^\nu a)|}.
\end{equation}

\noindent
\textit{Proof of the claim.}
Recall that
\[
Poi_a(v)=\frac{c_m a}{(a^2+|v|^2)^{\frac{m+1}{2}}}.
\]
On one hand, if $|v|<a$, then
\[
Poi_a(v)\le C a^{-m}\simeq \frac{C}{|B(0,a)|},
\]
which is the $\nu=0$ term on the right-hand side of \eqref{eq:1d-poisson-dyadic}.

Now let $\nu\ge 1$ and assume
\(
2^{\nu-1}a\le |v|<2^\nu a.
\)
Then
\[
Poi_a(v)
\le C \frac{a}{|v|^{m+1}}
\le C \frac{a}{(2^{\nu-1}a)^{m+1}}
\le C\,2^{-\nu(m+1)}a^{-m}.
\]
On the other hand,
\[
2^{-\nu}\frac{1}{|B(0,2^\nu a)|}
\simeq 2^{-\nu}\frac{1}{(2^\nu a)^m}
=2^{-\nu(m+1)}a^{-m}.
\]
Since $\chi_{B(0,2^\nu a)}(v)=1$, this gives \eqref{eq:1d-poisson-dyadic}. The claim is proved.

\medskip

We apply \eqref{eq:1d-poisson-dyadic} to each of the three Poisson factors in \eqref{eq:twisted-conv-start}. This yields
\[
(Poi_{twist})_{\mathbf r}*|h|(\mathbf y)
\le
C\sum_{\nu_1,\nu_2,\nu_3\ge 0}
2^{-(\nu_1+\nu_2+\nu_3)}\, I_{\nu_1,\nu_2,\nu_3},
\]
where, with
\(
s_j:=2^{\nu_j}r_j
\ \text{and}\ 
B_j:=B(0,s_j)\subset \mathbb R^m,
\)
we set
\begin{align*}
I_{\nu_1,\nu_2,\nu_3}
:=
\frac{1}{|B_1||B_2||B_3|}
\int_{\mathbb R^{2m}}
\int_{\mathbb R^m}
|h(\mathbf z)|
\chi_{B_1}(y_1-z_1-u)\,
\chi_{B_2}(y_2-z_2-u)\,
\chi_{B_3}(u)\,
du\,d\mathbf z.
\end{align*}

We now estimate the term $I_{\nu_1,\nu_2,\nu_3}$.
Fix $\mathbf z\in \mathbb R^{2m}$. If
\(
\chi_{B_1}(y_1-z_1-u)\,
\chi_{B_2}(y_2-z_2-u)\,
\chi_{B_3}(u)\neq 0
\)
for some $u\in \mathbb R^m$, then
\(
y_1-z_1-u\in B_1,\ 
y_2-z_2-u\in B_2,\)
and \(
u\in B_3.
\)
Define
\[
\zeta_1:=y_1-z_1-u,\quad
\zeta_2:=y_2-z_2-u,\quad\text{and}\quad
\zeta_3:=u.
\]
Then $(\zeta_1,\zeta_2,\zeta_3)\in \widetilde B(\mathbf 0,\mathbf s)$, where
\(
\mathbf s:=(s_1,s_2,s_3),
\)
and
\(
\mathbf y-\mathbf z
=
(\zeta_1+\zeta_3,\zeta_2+\zeta_3)
=
\pi(\zeta_1,\zeta_2,\zeta_3).
\)
Therefore
\[
\mathbf z\in \mathbf y-\pi\big(\widetilde B(\mathbf 0,\mathbf s)\big).
\]
Moreover, for fixed $\mathbf z$, the set of all such $u$ has measure at most
\(
\min\{|B_1|,|B_2|,|B_3|\}.
\)
Hence
\begin{align}
I_{\nu_1,\nu_2,\nu_3}
&\le
\frac{\min\{|B_1|,|B_2|,|B_3|\}}{|B_1||B_2||B_3|}
\int_{\mathbf y-\pi(\widetilde B(\mathbf 0,\mathbf s))}
|h(\mathbf z)|\,d\mathbf z.
\label{eq:Ik-first}
\end{align}

Next, since $|B_j|\simeq s_j^m$, we have
\[
\frac{\min\{|B_1|,|B_2|,|B_3|\}}{|B_1||B_2||B_3|}
=
\frac{1}{\max\{|B_1||B_2|,\ |B_1||B_3|,\ |B_2||B_3|\}}
\lesssim \frac{1}{V_{\mathbf s}}.
\]
Therefore \eqref{eq:Ik-first} implies
\begin{align}
I_{\nu_1,\nu_2,\nu_3}
\lesssim
\frac{1}{V_{\mathbf s}}
\int_{\mathbf y-\pi(\widetilde B(\mathbf 0,\mathbf s))}
|h(\mathbf z)|\,d\mathbf z.
\label{eq:Ik-second}
\end{align}

Since $\mathbf y\in T(\mathbf x,\beta\mathbf r)$, Proposition \ref{prop:tube-pi} with $2\beta\mathbf r$ in place of $\mathbf r$ gives
\(
T(\mathbf x,\beta\mathbf r)\subset \mathbf x+\pi\big(\widetilde B(\mathbf 0,2\beta\mathbf r)\big).
\)
Hence there exists
\(
\eta=(\eta_1,\eta_2,\eta_3)\in \widetilde B(\mathbf 0,2\beta\mathbf r)
\)
such that
\(
\mathbf y-\mathbf x=\pi(\eta).
\)\\
Now let
\(
\mathbf z\in \mathbf y-\pi(\widetilde B(\mathbf 0,\mathbf s)).
\)
Then there exists
\(
\zeta=(\zeta_1,\zeta_2,\zeta_3)\in \widetilde B(\mathbf 0,\mathbf s)
\)
such that
\(
\mathbf y-\mathbf z=\pi(\zeta).
\)
Subtracting the two identities, we obtain
\[
\mathbf z-\mathbf x
=
(\mathbf y-\mathbf x)-(\mathbf y-\mathbf z)
=
\pi(\eta)-\pi(\zeta)
=
\pi(\eta-\zeta).
\]
Since $s_j=2^{\nu_j}r_j\ge r_j$, we have
\(
|\eta_j-\zeta_j|
\le |\eta_j|+|\zeta_j|
\le 2\beta r_j+s_j
\le (2\beta+1)s_j
\ (j=1,2,3).
\)
Thus
\[
\eta-\zeta\in \widetilde B(\mathbf 0,(2\beta+1)\mathbf s),
\]
and therefore
\[
\mathbf z\in \mathbf x+\pi\big(\widetilde B(\mathbf 0,(2\beta+1)\mathbf s)\big).
\]
Applying Proposition \ref{prop:tube-pi} once more,
\(
\mathbf x+\pi\big(\widetilde B(\mathbf 0,(2\beta+1)\mathbf s)\big)
\subset
T\big(\mathbf x,\,2(2\beta+1)\mathbf s\big).
\)
Hence
\(
\mathbf y-\pi(\widetilde B(\mathbf 0,\mathbf s))
\subset
T\big(\mathbf x,\,2(2\beta+1)\mathbf s\big).
\)
Substituting this into \eqref{eq:Ik-second}, we get
\[
I_{\nu_1,\nu_2,\nu_3}
\lesssim
\frac{1}{V_{\mathbf s}}
\int_{T(\mathbf x,\,2(2\beta+1)\mathbf s)}
|h(\mathbf z)|\,d\mathbf z.
\]

Since the tube volume is homogeneous of degree $2m$ under scalar dilations,
\[
V_{\,2(2\beta+1)\mathbf s}
=
\big(2(2\beta+1)\big)^{2m}V_{\mathbf s}.
\]
Therefore
\begin{align*}
I_{\nu_1,\nu_2,\nu_3}
&\lesssim
\frac{V_{\,2(2\beta+1)\mathbf s}}{V_{\mathbf s}}
\cdot
\frac{1}{V_{\,2(2\beta+1)\mathbf s}}
\int_{T(\mathbf x,\,2(2\beta+1)\mathbf s)}
|h(\mathbf z)|\,d\mathbf z \lesssim_\beta
M_{tube}(h)(\mathbf x).
\end{align*}

Returning to the dyadic sum, we conclude that
\begin{align*}
(Poi_{twist})_{\mathbf r}*|h|(\mathbf y)
&\le
C\sum_{\nu_1,\nu_2,\nu_3\ge 0}
2^{-(\nu_1+\nu_2+\nu_3)}
I_{\nu_1,\nu_2,\nu_3} \\
&\lesssim_\beta
\left(\sum_{\nu_1,\nu_2,\nu_3\ge 0}2^{-(\nu_1+\nu_2+\nu_3)}\right)
M_{tube}(h)(\mathbf x) \\
&\lesssim_\beta M_{tube}(h)(\mathbf x).
\end{align*}
This proves the lemma.
\end{proof}

Let $E_\beta(\lambda) = \{\mathbf{x} \in \mathbb{R}^{2m}: U_f^*(\mathbf{x}) \le \lambda\}$ and define the indicator function $g(\mathbf{x}) = \chi_{E_\beta(\lambda)}(\mathbf{x})$. We define the proxy good set
\[ A_\beta(\lambda) := \left\{ \mathbf{x} \in \mathbb{R}^{2m} : M_{tube}(\chi_{E_\beta(\lambda)^c})(\mathbf{x}) \le \frac{1}{10C_0} \right\}. \]
By the $L^2$ boundedness of $M_{tube}$, $|A_\beta(\lambda)^c| \le C |E_\beta(\lambda)^c|$. By Chebyshev's inequality, it suffices to estimate the integral of the area function squared over $A_\beta(\lambda)$.

\begin{lem}
\label{lem:strict_separation}
Let $E_\beta(\lambda) = \{\mathbf{x} \in \mathbb{R}^{2m}: U_f^*(\mathbf{x}) \le \lambda\}$ be the good set, and let $g(\mathbf{x}) = \chi_{E_\beta(\lambda)}(\mathbf{x})$. 

Define the proxy set $A_\beta(\lambda) := \left\{ \mathbf{x} \in \mathbb{R}^{2m} : M_{tube}(\chi_{E_\beta(\lambda)^c})(\mathbf{x}) \le \frac{1}{10C_0} \right\}$. Let $W_\beta$ and $\widetilde{W}_\beta$ be the Carleson tents given by
\[ W_\beta = \bigcup_{\mathbf{x} \in A_\beta(\lambda)} \Gamma^\beta(\mathbf{x}), \quad\text{and}\quad \widetilde{W}_\beta = \bigcup_{\mathbf{x} \in E_\beta(\lambda)} \Gamma^\beta(\mathbf{x}). \]
Then the twisted Poisson extension $U_g(\mathbf{y}, \mathbf{r}) = (Poi_{twist})_{\mathbf{r}} * g(\mathbf{y})$ satisfies
\begin{align}\label{inner}
    U_g(\mathbf{y}, \mathbf{r}) > \frac{9}{10},\quad (\mathbf{y}, \mathbf{r}) \in W_\beta
\end{align}
and
\begin{align}\label{outer}
   \exists\ C_1=C_1(\beta) \in (0, 9/10)\ {\rm such\ that\ } \ U_g(\mathbf{y}, \mathbf{r}) \le C_1,\quad \forall(\mathbf{y}, \mathbf{r}) \notin \widetilde{W}_\beta,
\end{align}
provided that $\beta$ is chosen large enough.
\end{lem}

\begin{proof}
We first show (\ref{inner}).
Remarking that $g(\mathbf{y}) = 1 - \chi_{E_\beta(\lambda)^c}(\mathbf{y})$ and the twisted Poisson kernel integrates to 1, we can write the extension as:
\begin{align*}
U_g(\mathbf{y}, \mathbf{r}) &= (Poi_{twist})_{\mathbf{r}} * \left( 1 - \chi_{E_\beta(\lambda)^c} \right)(\mathbf{y}) 
= 1 - (Poi_{twist})_{\mathbf{r}} * \chi_{E_\beta(\lambda)^c}(\mathbf{y}).
\end{align*}
Let $(\mathbf{y}, \mathbf{r}) \in W_\beta$. By definition of $W_\beta$, there exists a base point $\mathbf{x} \in A_\beta(\lambda)$ such that $(\mathbf{y}, \mathbf{r}) \in \Gamma^\beta(\mathbf{x})$. 
Thus, we further have $\mathbf y\in T(\mathbf x,\beta\mathbf r)$. Applying Lemma \ref{lem:Poisson-tube-max} with $h=\chi_{E_\beta(\lambda)^c}$, we obtain
\[
(Poi_{twist})_{\mathbf r} * \chi_{E_\beta(\lambda)^c}(\mathbf y)
\le C_0\, M_{tube}\!\left(\chi_{E_\beta(\lambda)^c}\right)(\mathbf x).
\]

Since $\mathbf{x}$ belongs to $A_\beta(\lambda)$, we know that $M_{tube}\left(\chi_{E_\beta(\lambda)^c}\right)(\mathbf{x}) \le \frac{1}{10C_0}$. Substituting this bound yields 
\[
(Poi_{twist})_{\mathbf{r}} * \chi_{E_\beta(\lambda)^c}(\mathbf{y}) \le C_0 \left( \frac{1}{10C_0} \right) = \frac{1}{10},
\]
which gives (\ref{inner}).
\smallskip

Next, we verify (\ref{outer}). We will bound $U_g(\mathbf{y}, \mathbf{r})$ for all  $(\mathbf{y}, \mathbf{r})\in \big(\widetilde{W}_\beta\big)^c$.

We evaluate the Poisson integral of the indicator $g = \chi_{E_\beta(\lambda)}$. By explicitly writing the twisted Poisson kernel via its push-forward fiber integral, we lift the integration to $\mathbb{R}^{3m}$ and obtain
\begin{align*}
U_g(\mathbf{y}, \mathbf{r}) &= \int_{\mathbb{R}^{2m}} g(\mathbf{z}) (Poi_{twist})_{\mathbf{r}}(\mathbf{y} - \mathbf{z}) \, d\mathbf{z} \\
&= \int_{E_\beta(\lambda)} \left[ \int_{\mathbb{R}^m} Poi^{(1)}_{r_1}(y_1 - z_1 - u) Poi^{(2)}_{r_2}(y_2 - z_2 - u) Poi^{(3)}_{r_3}(u) \, du \right] d\mathbf{z}.
\end{align*}

Since $g(\mathbf{z}) \le 1$, we can bound the integral by extending the spatial integration over $d\mathbf{z}$ to the region where the lifted coordinates exceed the aperture. We perform a global change of variables in $\mathbb{R}^{3m}$:
$$v_1 = y_1 - z_1 - u, \quad v_2 = y_2 - z_2 - u, \quad\text{and}\quad v_3 = u.$$
For a fixed $\mathbf{y}$, the map $(z_1, z_2, u) \mapsto (v_1, v_2, v_3)$ is an affine transformation with a Jacobian determinant $1$. Thus, $dz_1 dz_2 du = dv_1 dv_2 dv_3$.

We proceed to translate the geometric condition $(\mathbf{y}, \mathbf{r}) \notin \widetilde{W}_\beta$ into boundaries for our new variables $v_1, v_2, v_3$. If $(\mathbf{y}, \mathbf{r}) \notin \widetilde{W}_\beta$, then for every $\mathbf{z} \in E_\beta(\lambda)$, the point $\mathbf{y}$ lies outside the twisted tube $T(\mathbf{z}, \beta\mathbf{r})$. Recall that this twisted tube is the projection of the product ball $\tilde{B}(\mathbf{z}_{lifted}, \beta\mathbf{r})$ along the fiber variable $u$. For $\mathbf{y}$ to fall outside this projection, the lifted coordinates cannot simultaneously belong to the product ball for any choice of $u \in \mathbb{R}^m$. In other words, the following three conditions cannot hold simultaneously:
$$|y_1 - z_1 - u| < \beta r_1, \quad |y_2 - z_2 - u| < \beta r_2 \quad \text{and} \quad |u| < \beta r_3.$$
That is, for every $u$, at least one of the reverse inequalities holds:
$$|y_1 - z_1 - u| \ge \beta r_1, \quad |y_2 - z_2 - u| \ge \beta r_2 \quad \text{or} \quad |u| \ge \beta r_3.$$
Substituting our new variables $v_1 = y_1 - z_1 - u$, $v_2 = y_2 - z_2 - u$, and $v_3 = u$ simplifies this logical statement to:
$$|v_1| \ge \beta r_1, \quad |v_2| \ge \beta r_2, \quad \text{or} \quad |v_3| \ge \beta r_3.$$
Expressed in terms of sets, every point $(v_1, v_2, v_3)$ in our integration domain belongs to the union of three regions $\Omega_1 \cup \Omega_2 \cup \Omega_3$, where:
$$\Omega_1 = \{|v_1| \ge \beta r_1\}, \quad \Omega_2 = \{|v_2| \ge \beta r_2\} \quad \text{and} \quad \Omega_3 = \{|v_3| \ge \beta r_3\}.$$
Although these sets technically overlap rather than being disjoint, the Poisson kernels are positive. This ensures that bounding the integrals over their union by the sum of the integrals over the individual sets is  valid via subadditivity, which justifies the decomposition in the subsequent step.

This implies
$$\begin{aligned}
U_g(\mathbf{y}, \mathbf{r}) &\le \iiint_{\{|v_1| \ge \beta r_1\} \cup \{|v_2| \ge \beta r_2\} \cup \{|v_3| \ge \beta r_3\}} Poi^{(1)}_{r_1}(v_1) Poi^{(2)}_{r_2}(v_2) Poi^{(3)}_{r_3}(v_3) \, dv_1 dv_2 dv_3 \\
&\le \sum_{j=1}^3 \int_{|v_j| \ge \beta r_j} Poi^{(j)}_{r_j}(v_j) \left( \iint_{\mathbb{R}^{2m}} Poi^{(k)}_{r_k}(v_k) Poi^{(\ell)}_{r_\ell}(v_\ell) \, dv_k dv_\ell \right) dv_j.
\end{aligned}$$

To establish the upper bound, we observe that the indicator function satisfies $g(\mathbf{z}) \le 1$. Since all Poisson kernels are positive, we can drop $g(\mathbf{z})$ and expand the domain of integration from the specific image of $E_\beta(\lambda)$ to the entire union $\Omega_1 \cup \Omega_2 \cup \Omega_3$. Utilizing the subadditivity of the integral for positive functions, the integral over this union is majorized by the sum of the integrals over the individual sets $\Omega_j$. 

Evaluating the first term over $\Omega_1 = \{|v_1| \ge \beta r_1\}$, we note that the domain restricts only the variable $v_1$, placing absolutely no conditions on $v_2$ or $v_3$. Consequently, $v_2$ and $v_3$ are integrated over their respective full spaces $\mathbb{R}^m$. Since the variables $v_1, v_2, v_3$ are completely separable, applying Fubini's Theorem separates the triple integral into a product of three independent integrals:
\begin{align*}
&\iiint_{\Omega_1} Poi^{(1)}_{r_1}(v_1) Poi^{(2)}_{r_2}(v_2) Poi^{(3)}_{r_3}(v_3) \, dv_1 dv_2 dv_3 \\
&= \left( \int_{|v_1| \ge \beta r_1} Poi^{(1)}_{r_1}(v_1) \, dv_1 \right) \left( \int_{\mathbb{R}^m} Poi^{(2)}_{r_2}(v_2) \, dv_2 \right) \left( \int_{\mathbb{R}^m} Poi^{(3)}_{r_3}(v_3) \, dv_3 \right).
\end{align*}
This reasoning applies identically to the remaining integrals over $\Omega_2$ and $\Omega_3$. Condensing these results, the overall bound reduces to a sum $\sum_{j=1}^3$ where, for each $j$, the integration of $v_j$ is restricted to its tail region $\{|v_j| \ge \beta r_j\}$, while the remaining variables $v_k$ and $v_\ell$ (for $k, \ell \neq j$) are integrated over the entire space $\mathbb{R}^{2m}$.

Since $\int_{\mathbb{R}^m} Poi(v) dv = 1$,  we have
\begin{align*}
U_g(\mathbf{y}, \mathbf{r}) &\le \sum_{j=1}^3 \int_{|v_j| \ge \beta r_j} Poi^{(j)}_{r_j}(v_j) \, dv_j.
\end{align*}

Recall the explicit formula for the Euclidean Poisson kernel: $Poi^{(j)}_{r_j}(v) = \frac{c_m r_j}{(r_j^2 + |v|^2)^{(m+1)/2}}$. We evaluate the tail integral for each $j$ by performing the scaling substitution $v = r_j w$, which gives $dv = r_j^m dw$:
\begin{align*}
\int_{|v| \ge \beta r_j} \frac{c_m r_j}{(r_j^2 + |v|^2)^{\frac{m+1}{2}}} \, dv &= \int_{|w| \ge \beta} \frac{c_m r_j}{(r_j^2 + r_j^2 |w|^2)^{\frac{m+1}{2}}} r_j^m \, dw 
= \int_{|w| \ge \beta} \frac{c_m}{(1 + |w|^2)^{\frac{m+1}{2}}} \, dw.
\end{align*}
For large $\beta$, integrating in spherical coordinates (where the surface area grows as $|w|^{m-1}$) yields a decay rate of:
\[
\int_{|w| \ge \beta} |w|^{-(m+1)} |w|^{m-1} \, d|w| = \int_\beta^\infty w^{-2} \, dw = O(\beta^{-1}).
\]
Thus, for each $j=1,2,3$, the tail integral is bounded by $C/\beta$ for some geometric constant $C$. Summing over the three regions gives that
\[
U_g(\mathbf{y}, \mathbf{r}) \le \frac{3C}{\beta}.
\]
By choosing $\beta$ sufficiently large such that $\frac{3C}{\beta} =: C_1 < \frac{9}{10}$, we guarantee that $U_g(\mathbf{y}, \mathbf{r}) \le C_1$ everywhere outside the enlarged tent $\widetilde{W}_\beta$.
\end{proof}

Choose a smooth cut-off $\varphi(t)$ such that $\varphi(t)=1$ for $t \ge 9/10$ and $\varphi(t)=0$ for $t \le C_1$. 
To shorten the notation, let $v_{2,3}(\mathbf{y}, \mathbf{r}) =  \nabla_2 \nabla_3 U_f(\mathbf{y}, \mathbf{r})$. 

We define the primary integral of the area function over the proxy good set $A_\beta(\lambda)$. Recalling the definition of the tent $W_\beta = \bigcup_{\mathbf{x} \in A_\beta(\lambda)} \Gamma^\beta(\mathbf{x})$, we can bound the integral over $A_\beta(\lambda)$ by the integral over the full tent $W_\beta$:
\begin{align*}
\int_{A_\beta(\lambda)} S_{twist}(f)^2(\mathbf{x}) d\mathbf{x} &\le \int_{W_\beta} |\nabla_1 v_{2,3}(\mathbf{y}, \mathbf{r})|^2 r_1 r_2 r_3 \, d\mathbf{y} d\mathbf{r}.
\end{align*}

Since the Poisson extension $U_g(\mathbf{y}, \mathbf{r}) > 9/10$ on the tent $W_\beta$ (from Lemma \ref{lem:strict_separation}), our chosen cut-off $\varphi(U_g)$ is identically 1 on $W_\beta$. Thus, we can introduce the auxiliary term $\varphi(U_g)^2$ into the integrand without decreasing its value over the tent. Let $\mathcal{I}_1$ be defined as the integral over the entire multi-parameter upper half-space:
\begin{align}\label{eq:integral_A}
\mathcal{I}_1 := \int_{\mathbb{R}^{2m} \times \mathbb{R}_+^3} |\nabla_1 v_{2,3}|^2 |\varphi(U_g)|^2 r_1 r_2 r_3 \, d\mathbf{y} d\mathbf{r}.
\end{align}
Since $\varphi(U_g)^2 = 1$ on $W_\beta$ and is non-negative everywhere else, it holds that $\int_{W_\beta} |\nabla_1 v_{2,3}|^2 \le \mathcal{I}_1$, which justifies extending the integration domain to $\mathbb{R}^{2m} \times \mathbb{R}_+^3$.

\subsection{First iteration: passing $\nabla_1$ to $g$.}

We evaluate the integral $\mathcal{I}_1$.
Since both $v_{2,3}$ and $U_g$ are harmonic with respect to the standard Euclidean Laplacian $\Delta_1 = \Delta_{y_1} + \partial_{r_1}^2$, we apply the product rule $\Delta_1(AB) = A \Delta_1(B) + B \Delta_1(A) + 2\nabla_1(A)\cdot\nabla_1(B)$ where $A = |v_{2,3}|^2$ and $B = \varphi(U_g)^2$.

Substituting the harmonic property $\Delta_1(u^2) = 2|\nabla_1 u|^2$, we evaluate the individual pieces:
\begin{enumerate}
\item $\Delta_1(|v_{2,3}|^2) = 2|\nabla_1 v_{2,3}|^2$ (since $\Delta_1 v_{2,3} = 0$).
\item By the chain rule, $\nabla_1(\varphi(U_g)^2) = 2\varphi(U_g)\varphi'(U_g)\nabla_1 U_g$. Applying the divergence operator yields
\begin{align*}
    \Delta_1(\varphi(U_g)^2) &= 2\big(\varphi'(U_g)^2 + \varphi(U_g)\varphi''(U_g)\big)|\nabla_1 U_g|^2 + 2\varphi(U_g)\varphi'(U_g)\Delta_1 U_g\\
    &= 2\big(\varphi'(U_g)^2 + \varphi(U_g)\varphi''(U_g)\big)|\nabla_1 U_g|^2,
\end{align*} 
where the last equality follows from $\Delta_1 U_g = 0$.
\item The cross-term evaluates to 
\begin{align*}
2\nabla_1(|v_{2,3}|^2) \cdot \nabla_1(\varphi(U_g)^2) &= 2(2v_{2,3}\nabla_1 v_{2,3}) \cdot (2\varphi(U_g)\varphi'(U_g)\nabla_1 U_g) \\
&= 8 v_{2,3} \nabla_1 v_{2,3} \cdot \varphi(U_g) \varphi'(U_g) \nabla_1 U_g.
\end{align*}
\end{enumerate}
Substituting these back into $\Delta_1(|v_{2,3}|^2 \varphi(U_g)^2)$ and then algebraically rearranging to isolate the term $|\nabla_1 v_{2,3}|^2 |\varphi(U_g)|^2$, we obtain:
\begin{align}\label{I1 f1234}
|\nabla_1 v_{2,3}|^2 |\varphi(U_g)|^2 &= \frac{1}{2}\Delta_1 \left( |v_{2,3}|^2 |\varphi(U_g)|^2 \right) 
 - 4 v_{2,3} \nabla_1 v_{2,3} \cdot \varphi(U_g) \varphi'(U_g) \nabla_1 U_g \\
&\quad - |v_{2,3}|^2 |\varphi'(U_g)|^2 |\nabla_1 U_g|^2 - |v_{2,3}|^2 \varphi(U_g) \varphi''(U_g) |\nabla_1 U_g|^2 \notag\\
&=: f_1 + f_2 + f_3 + f_4.\notag
\end{align}
We integrate these four terms over $d\mathbf{y}_1$ and $r_1 dr_1$.

For the term $f_1$, we split the Laplacian into its spatial and scale components: $\Delta_1 = \Delta_{y_1} + \partial_{r_1}^2$.
\begin{align*}
&\iint_{\mathbb{R}^{m} \times \mathbb{R}_+} f_1 \, r_1 dr_1 d\mathbf{y}_1 \\
&= \frac{1}{2} \int_{\mathbb{R}_+} \left( \int_{\mathbb{R}^{m}} \Delta_{y_1} \big( |v_{2,3}|^2 \varphi^2 \big) d\mathbf{y}_1 \right) r_1 dr_1 + \frac{1}{2} \int_{\mathbb{R}^{m}} \left( \int_0^\infty \partial_{r_1}^2 \big( |v_{2,3}|^2 \varphi^2 \big) r_1 dr_1 \right) d\mathbf{y}_1.
\end{align*}
By the divergence theorem, the spatial integral completely vanishes since the Poisson extension and its derivatives exhibit rapid polynomial decay at spatial infinity (as $|\mathbf{y}_1| \to \infty$), leaving no boundary contribution.

For the scale integral, let $F(r_1) = |v_{2,3}|^2 \varphi(U_g)^2$. Integrating by parts once (setting $u = r_1$ and $dv = F''(r_1)dr_1$) yields:
\[
\int_0^\infty F''(r_1) r_1 dr_1 = \Big[ r_1 F'(r_1) \Big]_0^\infty - \int_0^\infty F'(r_1) dr_1 = \Big[ r_1 F'(r_1) \Big]_0^\infty - \Big[ F(r_1) \Big]_0^\infty.
\]
We evaluate the boundaries:

$\bullet$ As $r_1 \to \infty$, the Poisson extension spreads out, meaning $U_g(\mathbf{y}, \mathbf{r}) \to 0$. Since the cut-off $\varphi(t) = 0$ for $t \le C_1$ (where $C_1 > 0$), $\varphi(U_g)$ is identically zero for sufficiently large $r_1$. Thus, we have $F(\infty) = 0$ and $F'(\infty) = 0$.

$\bullet$  As $r_1 \to 0$, the standard gradient properties of the Poisson kernel ensure that $r_1 \partial_{r_1} U_f \to 0$ almost everywhere for integrable functions. Thus, the scale-weighted derivative $r_1 F'(r_1)$ vanishes at the origin.

Since the upper limits vanish and the lower limit of the first term vanishes, the only remaining term is the lower boundary evaluation of $F(r_1)$ evaluated at zero. The integration by parts is evaluated    to $F(0)$, that is,
\begin{align}\label{int f1}
\iint_{\mathbb{R}^{m} \times \mathbb{R}_+} f_1 \, r_1 dr_1 d\mathbf{y}_1 = \frac{1}{2} F(0) = \frac{1}{2} \int_{\mathbb{R}^{m}} |v_{2,3}(\mathbf{y}, 0, r_2, r_3)|^2 \varphi(U_g(\mathbf{y}, 0, r_2, r_3))^2 d\mathbf{y}_1.
\end{align}
This boundary term drops the integration of $r_1$, then becomes the starting point for the next iteration of the remaining parameters.

We now consider the cross terms $f_2, f_3, f_4$.

We first bound the absolute value of the first-order cross term $f_2$ using Young's inequality ($2ab \le \epsilon a^2 + \frac{1}{\epsilon} b^2$) with $\epsilon = \frac{1}{10}$. By explicitly choosing $a = |\nabla_1 v_{2,3}| \varphi(U_g)$ and $b = |2 v_{2,3} \varphi'(U_g) \nabla_1 U_g|$, we have:
\begin{align*}
|f_2| &= \left| 2 \big(\nabla_1 v_{2,3} \varphi(U_g)\big) \cdot \big(2 v_{2,3} \varphi'(U_g) \nabla_1 U_g\big) \right| \\
&\le \frac{1}{10} \big| \nabla_1 v_{2,3} \varphi(U_g) \big|^2 + 10 \big| 2 v_{2,3} \varphi'(U_g) \nabla_1 U_g \big|^2 \\
&= \frac{1}{10} |\nabla_1 v_{2,3}|^2 \varphi(U_g)^2 + 10 \left( 4 |v_{2,3}|^2 |\varphi'(U_g)|^2 |\nabla_1 U_g|^2 \right) \\
&= \frac{1}{10} |\nabla_1 v_{2,3}|^2 \varphi(U_g)^2 + 40 |v_{2,3}|^2 |\varphi'(U_g)|^2 |\nabla_1 U_g|^2.
\end{align*}

Next, we simply take the absolute values of the remaining terms $f_3$ and $f_4$. Since they are already products of squares (or simple scalars), they do not require Young's inequality. We only need to ensure that the potentially negative second derivative $\varphi''(U_g)$ is enveloped by the absolute value:
\begin{align*}
    |f_3| = \left| - |v_{2,3}|^2 |\varphi'(U_g)|^2 |\nabla_1 U_g|^2 \right| = |v_{2,3}|^2 |\varphi'(U_g)|^2 |\nabla_1 U_g|^2, 
\end{align*}
and
\begin{align*}
|f_4| = \left| - |v_{2,3}|^2 \varphi(U_g) \varphi''(U_g) |\nabla_1 U_g|^2 \right| = |v_{2,3}|^2 \big|\varphi(U_g) \varphi''(U_g)\big| |\nabla_1 U_g|^2.
\end{align*}
By the triangle inequality, the total error from the cross terms is bounded by the sum of their absolute values. Notice that $|f_2|$, $|f_3|$, and $|f_4|$ all share the common functional multiplier $|v_{2,3}|^2 |\nabla_1 U_g|^2$. We group them explicitly:
\begin{align*}
|f_2 + f_3 + f_4| 
&\le \frac{1}{10} |\nabla_1 v_{2,3}|^2 \varphi(U_g)^2  + 40 |v_{2,3}|^2 |\varphi'(U_g)|^2 |\nabla_1 U_g|^2 \\
&\quad +  |v_{2,3}|^2 |\varphi'(U_g)|^2 |\nabla_1 U_g|^2  +  |v_{2,3}|^2 \big|\varphi(U_g) \varphi''(U_g)\big| |\nabla_1 U_g|^2\\
&= \frac{1}{10} |\nabla_1 v_{2,3}|^2 \varphi(U_g)^2 + |v_{2,3}|^2 \Big( 41 |\varphi'(U_g)|^2 + \big|\varphi(U_g) \varphi''(U_g)\big| \Big) |\nabla_1 U_g|^2.
\end{align*}

To absorb this messy bracketed term into a single harmonic-compatible form, we define a new smooth cut-off function $\Phi_1(t)$ such that the square of  its derivative    matches the internal bracket:
\[
\Phi_1'(t) = \Big( 41 |\varphi'(t)|^2 + \big|\varphi(t) \varphi''(t)\big| \Big)^{1/2}, \quad \text{with} \quad \Phi_1(C_1) = 0.
\]
Since $\varphi(t)$ is constant outside the interval $(C_1, 9/10)$, its derivatives $\varphi'$ and $\varphi''$ vanish identically there. This ensures that $\Phi_1'(t) = 0$ everywhere outside that interval, meaning that $\Phi_1$ inherits the   same compact support properties as $\varphi$.

By the chain rule, $\nabla_1 \Phi_1(U_g) = \Phi_1'(U_g) \nabla_1 U_g$. Squaring this gives 
\[
|\nabla_1 \Phi_1(U_g)|^2 = |\Phi_1'(U_g)|^2 |\nabla_1 U_g|^2.
\] 
Substituting this directly back into our grouped bound, we can rewrite it elegantly and compactly as:
\begin{align}\label{bound f234}
|f_2 + f_3 + f_4| \le \frac{1}{10} |\nabla_1 v_{2,3}|^2 \varphi(U_g)^2 + |v_{2,3}|^2 |\nabla_1 \Phi_1(U_g)|^2.
\end{align}

Thus, combining \eqref{eq:integral_A}, \eqref{I1 f1234}, \eqref{int f1} and \eqref{bound f234}, we have
\begin{align*}
\mathcal{I}_1 &\le \int_{\mathbb{R}^{2m} \times \mathbb{R}_+^2} \frac{1}{2} |v_{2,3}|_{r_1=0}^2 \varphi(U_g|_{r_1=0})^2 \, d\mathbf{y} \, r_2 dr_2 \, r_3 dr_3 \\
&\quad + \frac{1}{10} \underbrace{\int_{\mathbb{R}^{2m} \times \mathbb{R}_+^3} |\nabla_1 v_{2,3}|^2 \varphi(U_g)^2 r_1 r_2 r_3 \, d\mathbf{y} d\mathbf{r}}_{= \mathcal{I}_1} \\
&\quad + \int_{\mathbb{R}^{2m} \times \mathbb{R}_+^3} |v_{2,3}|^2 |\nabla_1 \Phi_1(U_g)|^2 r_1 r_2 r_3 \, d\mathbf{y} d\mathbf{r}.
\end{align*}
We now absorb the $\frac{1}{10} \mathcal{I}_1$ term into the left-hand side by subtracting it from $\mathcal{I}_1$:
\begin{align*}
 \frac{9}{10} \mathcal{I}_1 
&\le \int_{\mathbb{R}^{2m} \times \mathbb{R}_+^2} \frac{1}{2} |v_{2,3}|_{r_1=0}^2 \varphi(U_g|_{r_1=0})^2 \, d\mathbf{y} \, r_2 dr_2 \, r_3 dr_3 \\
&\quad + \int_{\mathbb{R}^{2m} \times \mathbb{R}_+^3} |v_{2,3}|^2 |\nabla_1 \Phi_1(U_g)|^2 r_1 r_2 r_3 \, d\mathbf{y} d\mathbf{r}.
\end{align*}

Since we established earlier that $\int_{A_\beta(\lambda)} S_{twist}(f)^2 d\mathbf{x} \le \mathcal{I}_1$, substituting this algebraic result provides the complete bound for the first iteration:
\begin{align}\label{eq:iter1}
\int_{A_\beta(\lambda)} S_{twist}(f)^2 d\mathbf{x} &\le \frac{5}{9} \int_{\mathbb{R}^{2m} \times \mathbb{R}_+^2} |v_{2,3}|_{r_1=0}^2 \varphi(U_g|_{r_1=0})^2 \, d\mathbf{y} \, r_2 dr_2 \, r_3 dr_3 \nonumber \\
&\quad + \frac{10}{9} \int_{\mathbb{R}^{2m} \times \mathbb{R}_+^3} |v_{2,3}|^2 |\nabla_1 \Phi_1(U_g)|^2 r_1 r_2 r_3 \, d\mathbf{y} d\mathbf{r}.
\end{align}

Thus, the estimate is reduced to a two-parameter term and an error term. The first term depends only on $r_2$ and $r_3$, and will be treated by repeating the argument in the second block. In the second term, the derivative $\nabla_1$ has been transferred from $v_{2,3}$ to the cut-off $\Phi_1(U_g)$, and this term will be handled in the same way.

\subsection{Second iteration: passing $\nabla_2$ to $g$.} 

Recall the function $v_{2,3} = \nabla_2 v_3$, where $v_3 = \nabla_3 U_f$. (Note: The scale weights $r_1, r_2, r_3$ have already been extracted into the volume measure $r_1 dr_1 \, r_2 dr_2 \, r_3 dr_3$ by the definition of the area function, simplifying the gradients).

From the first iteration, we have two integrals containing the term $|\nabla_2 v_3|^2 B$, where the weight function $B$ is either $\varphi(U_g|_{r_1=0})^2$ (from the boundary term) or $|\nabla_1 \Phi_1(U_g)|^2$ (from the cross-derivative term). To formalize the bootstrapping process, we explicitly define these two integrals:
\begin{align*}
\mathcal{I}_{2, cross} := \int_{\mathbb{R}^{2m} \times \mathbb{R}_+^3} |\nabla_2 v_3|^2 |\nabla_1 \Phi_1(U_g)|^2 r_1 r_2 r_3 \, d\mathbf{y} d\mathbf{r}, 
\end{align*}
and
\begin{align*}
\mathcal{I}_{2, bdry}:= \int_{\mathbb{R}^{2m} \times \mathbb{R}_+^2} |\nabla_2 v_3|_{r_1=0}^2 \varphi(U_g|_{r_1=0})^2 \, d\mathbf{y} \, r_2 dr_2 \, r_3 dr_3.
\end{align*}

We demonstrate the process on the cross-derivative integral $\mathcal{I}_{2, cross}$; the boundary term $\mathcal{I}_{2, bdry}$ follows an identical procedure.

We apply the harmonic product rule for the standard Euclidean Laplacian $\Delta_2 = \Delta_{y_2} + \partial_{r_2}^2$ to the product $A \cdot B$, where $A = |v_3|^2$ and $B = |\nabla_1 \Phi_1(U_g)|^2$. 
Recall the identity $\Delta_2(A B) = (\Delta_2 A) B + A (\Delta_2 B) + 2 \nabla_2 A \cdot \nabla_2 B$. 
Since $v_3$ is harmonic with respect to the second block ($\Delta_2 v_3 = 0$), we have $\Delta_2(|v_3|^2) = 2|\nabla_2 v_3|^2$. Furthermore, the gradient expands as $\nabla_2(|v_3|^2) = 2 v_3 \nabla_2 v_3$. Substituting these into the product rule gives:
\begin{align*}
\Delta_2 \Big( |v_3|^2 |\nabla_1 \Phi_1|^2 \Big) &= \Big( 2|\nabla_2 v_3|^2 \Big) |\nabla_1 \Phi_1|^2 + |v_3|^2 \Delta_2 \big( |\nabla_1 \Phi_1|^2 \big) + 2 \Big( 2 v_3 \nabla_2 v_3 \Big) \cdot \nabla_2 \big( |\nabla_1 \Phi_1|^2 \big) \\
&= 2 |\nabla_2 v_3|^2 |\nabla_1 \Phi_1|^2 + |v_3|^2 \Delta_2 \big( |\nabla_1 \Phi_1|^2 \big) + 4 v_3 \nabla_2 v_3 \cdot \nabla_2 \big( |\nabla_1 \Phi_1|^2 \big).
\end{align*}

Rearranging algebraically to isolate the target term $|\nabla_2 v_3|^2 |\nabla_1 \Phi_1|^2$, and dividing the entire equation by 2, we obtain the fundamental identity for the second iteration:
\begin{align}\label{nabla2 v3 nabla1 phi1}
|\nabla_2 v_3|^2 |\nabla_1 \Phi_1|^2 &= \frac{1}{2} \Delta_2 \Big( |v_3|^2 |\nabla_1 \Phi_1|^2 \Big)  - 2 v_3 \nabla_2 v_3 \cdot \nabla_2 \big( |\nabla_1 \Phi_1|^2 \big)  - \frac{1}{2} |v_3|^2 \Delta_2 \big( |\nabla_1 \Phi_1|^2 \big).
\end{align}

For the term $\frac{1}{2} \Delta_2 \Big( |v_3|^2 |\nabla_1 \Phi_1|^2 \Big)$, we separate the spatial and scale components:
\begin{align*}
\frac{1}{2} \iint_{\mathbb{R}^{2m} \times \mathbb{R}_+} \Delta_2 \Big( |v_3|^2 |\nabla_1 \Phi_1|^2 \Big) r_2 dr_2 d\mathbf{y}_2
& = \frac{1}{2} \int_{\mathbb{R}_+} \left( \int_{\mathbb{R}^{m}} \Delta_{y_2} \Big( |v_3|^2 |\nabla_1 \Phi_1|^2 \Big) d\mathbf{y}_2 \right) r_2 dr_2 \\
&\quad+ \frac{1}{2} \int_{\mathbb{R}^{m}} \left( \int_0^\infty \partial_{r_2}^2 \Big( |v_3|^2 |\nabla_1 \Phi_1|^2 \Big) r_2 dr_2 \right) d\mathbf{y}_2.
\end{align*}
By the divergence theorem, the spatial integral completely vanishes due to the rapid polynomial decay of the Poisson extension and its derivatives at spatial infinity.

For the scale integral, let $F(r_2) = |v_3|^2 |\nabla_1 \Phi_1|^2$. Integrating by parts once yields:
\[
\int_0^\infty F''(r_2) r_2 dr_2 = \Big[ r_2 F'(r_2) \Big]_0^\infty - \int_0^\infty F'(r_2) dr_2 = \Big[ r_2 F'(r_2) \Big]_0^\infty - \Big[ F(r_2) \Big]_0^\infty.
\]
We evaluate the boundaries:

$\bullet$ As $r_2 \to \infty$, the Poisson extension vanishes. Since the compact support of $\Phi_1$ forces it to zero for small inputs, the upper boundary evaluates   to zero ($F(\infty) = 0$ and $F'(\infty) = 0$).

$\bullet$  As $r_2 \to 0$, the $r_2$-weighted gradient of the harmonic extension goes to zero, ensuring the term $r_2 F'(r_2)$ vanishes at the origin.

Thus, 
\begin{align*}
&\frac{1}{2} \iint \Delta_2 \Big( |v_3|^2 |\nabla_1 \Phi_1|^2 \Big) r_2 dr_2 d\mathbf{y}_2 
={1\over2}F(0)= \frac{1}{2} \int_{\mathbb{R}^{m}} \Big( |v_3(\mathbf{y}, r_1, 0, r_3)|^2 |\nabla_1 \Phi_1(U_g(\mathbf{y}, r_1, 0, r_3))|^2 \Big) d\mathbf{y}_2.
\end{align*}

We now consider the cross term $- 2 v_3 \nabla_2 v_3 \cdot \nabla_2 \big( |\nabla_1 \Phi_1|^2 \big)$ in \eqref{nabla2 v3 nabla1 phi1}.
 Let $\mathbf{W} = \nabla_1 \Phi_1(U_g)$, which is a smooth vector-valued function. Then
\[
2 v_3 \nabla_2 v_3 \cdot \nabla_2 \big( |\mathbf{W}|^2 \big) = 2 v_3 \nabla_2 v_3 \cdot \big( 2 \mathbf{W} \cdot \nabla_2 \mathbf{W} \big) = 4 ( \mathbf{W} \nabla_2 v_3 ) \cdot ( v_3 \nabla_2 \mathbf{W} ).
\]

We apply Young's inequality in the form $2 \mathbf{a} \cdot \mathbf{b} \le \epsilon |\mathbf{a}|^2 + \frac{1}{\epsilon} |\mathbf{b}|^2$ with $\epsilon = \frac{1}{10}$. By identifying $\mathbf{a} = \mathbf{W} \nabla_2 v_3$ and $\mathbf{b} = 2 v_3 \nabla_2 \mathbf{W}$, we obtain:
\begin{align*}
\left| 4 ( \mathbf{W} \nabla_2 v_3 ) \cdot ( v_3 \nabla_2 \mathbf{W} ) \right| &\le \frac{1}{10} |\mathbf{W} \nabla_2 v_3|^2 + 10 |2 v_3 \nabla_2 \mathbf{W}|^2 \\
&= \frac{1}{10} |\nabla_2 v_3|^2 |\nabla_1 \Phi_1|^2 + 40 |v_3|^2 |\nabla_2 \nabla_1 \Phi_1|^2.
\end{align*}

Next, we evaluate the remaining term $\frac{1}{2} |v_3|^2 \big| \Delta_2 \big( |\mathbf{W}|^2 \big) \big|$ in \eqref{nabla2 v3 nabla1 phi1}. Using the identity $\Delta_2( |\mathbf{W}|^2 ) = 2|\nabla_2 \mathbf{W}|^2 + 2 \mathbf{W} \cdot \Delta_2 \mathbf{W}$, we have:
\[
\frac{1}{2} |v_3|^2 \left| \Delta_2 \big( |\nabla_1 \Phi_1|^2 \big) \right| \le |v_3|^2 |\nabla_2 \nabla_1 \Phi_1|^2 + |v_3|^2 \left| \nabla_1 \Phi_1 \cdot \Delta_2 \nabla_1 \Phi_1 \right|.
\]
By the chain rule, the mixed derivatives of $\Phi_1(U_g)$ expand completely into polynomial combinations of the gradients of $U_g$. Specifically, note that $\Delta_2 U_g = 0$:
\begin{align*}
\nabla_1 \Phi_1(U_g) = \Phi_1'(U_g) \nabla_1 U_g, 
\end{align*}
\begin{align*}
\nabla_2 \nabla_1 \Phi_1(U_g) = \Phi_1''(U_g) \nabla_2 U_g \otimes \nabla_1 U_g + \Phi_1'(U_g) \nabla_2 \nabla_1 U_g, 
\end{align*}
and
\begin{align*}
\Delta_2 \nabla_1 \Phi_1(U_g) = \nabla_1 \big( \Delta_2 \Phi_1(U_g) \big) = \nabla_1 \big( \Phi_1''(U_g) |\nabla_2 U_g|^2 \big).
\end{align*}
Expanding the last term, we see that it is composed of terms like $\Phi_1'''(U_g) \nabla_1 U_g |\nabla_2 U_g|^2$ and $2\Phi_1''(U_g) \nabla_2 U_g \cdot \nabla_1 \nabla_2 U_g$. These terms are exclusively composed of the components $|\nabla_1 U_g| |\nabla_2 U_g|^2$ and $|\nabla_2 U_g| |\nabla_1 \nabla_2 U_g|$, all of which are enveloped by the smooth, compactly supported derivatives of $\Phi_1$.

Combining the above estimates, we see that the second and third term in the right-hand side of \eqref{nabla2 v3 nabla1 phi1} is bounded by
\begin{align*}
&\left| - 2 v_3 \nabla_2 v_3 \cdot \nabla_2 \big( |\nabla_1 \Phi_1|^2 \big) \right| + \left| \frac{1}{2} |v_3|^2 \Delta_2 \big( |\nabla_1 \Phi_1|^2 \big) \right| \\
&\le \frac{1}{10} |\nabla_2 v_3|^2 |\nabla_1 \Phi_1|^2 + |v_3|^2 \left( 41 |\nabla_2 \nabla_1 \Phi_1(U_g)|^2 + \left| \nabla_1 \Phi_1(U_g) \cdot \Delta_2 \nabla_1 \Phi_1 (U_g)\right| \right).
\end{align*}

To absorb this bracketed term, we substitute the expansions of $\nabla_1 \Phi_1(U_g)$, $\nabla_2 \nabla_1 \Phi_1(U_g)$, and $\Delta_2 \nabla_1 \Phi_1(U_g)$. For the first term in the bracket, we obtain:
\begin{align*}
    41|\nabla_2 \nabla_1 \Phi_1(U_g)|^2
    &= 41|\Phi_1''(U_g) \nabla_1 U_g \otimes\nabla_2 U_g+\Phi_1'(U_g)\nabla_1\nabla_2U_g|^2\\
    &\le 82|\Phi_1''(U_g)|^2|\nabla_1 U_g \otimes\nabla_2 U_g|^2+82|\Phi_1'(U_g)|^2|\nabla_1\nabla_2U_g|^2.
\end{align*}

For the second term in the bracket, we use Young's inequality and obtain:
\begin{align*}
    &\left| \nabla_1 \Phi_1(U_g) \cdot \Delta_2 \nabla_1 \Phi_1 (U_g)\right|\\
    &=\Big|\Phi_1'(U_g)\nabla_1U_g \cdot \big(\Phi_1'''(U_g)\nabla_1U_g|\nabla_2U_g|^2+2\Phi_1''(U_g)\nabla_2U_g \cdot \nabla_2\nabla_1U_g\big)\Big|\\
    &\le |\Phi_1'(U_g)\Phi_1'''(U_g)||\nabla_1U_g|^2|\nabla_2U_g|^2+2\big|\Phi_1'(U_g)\Phi_1''(U_g)(\nabla_1U_g\otimes\nabla_2U_g) \cdot \nabla_2\nabla_1U_g\big|\\
    &\le |\Phi_1'(U_g)\Phi_1'''(U_g)||\nabla_1U_g|^2|\nabla_2U_g|^2+|\Phi_1''(U_g)|^2|\nabla_1U_g\otimes\nabla_2U_g|^2+|\Phi_1'(U_g)|^2|\nabla_1\nabla_2U_g|^2.
\end{align*}

Combine them and we obtain the total estimate of the bracketed term:
\begin{align*}
    &41|\nabla_2 \nabla_1 \Phi_1(U_g)|^2 +  \left| \nabla_1 \Phi_1(U_g) \cdot \Delta_2 \nabla_1 \Phi_1 (U_g)\right|\\
    &\qquad\le \big(83|\Phi_1''(U_g)|^2+|\Phi_1'(U_g)\Phi_1'''(U_g)|\big)|\nabla_1U_g|^2|\nabla_2U_g|^2+ 83|\Phi_1'(U_g)|^2|\nabla_1\nabla_2U_g|^2.
\end{align*}

Now, we define two smooth cut-off functions $\Psi_A$ and $\Psi_B$ such that their squares    match the corresponding internal brackets:
\[
\Psi_A(t)=\Big(83|\Phi_1''(t)|^2+|\Phi_1'(t)\Phi_1'''(t)|\Big)^{1/2}
\quad\text{and}\quad
\Psi_B(t)=\Big(83|\Phi_1'(t)|^2\Big)^{1/2}.
\]
Then we further define the enveloping function $\mathcal{G}_2(\mathbf{y}, \mathbf{r})$:
\[
\mathcal{G}_2(\mathbf{y}, \mathbf{r}) := 
|W_A|^2+|W_B|^2,
\]
where
$W_A = \Psi_A(U_g) (\nabla_1 U_g \otimes \nabla_2 U_g)$ and $W_B = \Psi_B(U_g) (\nabla_1 \nabla_2 U_g)$.

Since $\Phi_1(t)$ is constant outside $(C_1, 9/10)$, its derivatives vanish there, guaranteeing that $\Psi_A$, $\Psi_B$ and $\mathcal{G}_2$ inherits the   same compact support as our original cut-off $\varphi$. Substituting the   boundary integration of the Laplacian term at $r_2=0$ and this new unified pointwise upper bound back into the full integral defining $\mathcal{I}_{2, cross}$, we obtain:
\begin{align*}
\mathcal{I}_{2, cross} &\le \int_{\mathbb{R}^{2m} \times \mathbb{R}_+^2} \frac{1}{2} \Big( |v_3|^2 |\nabla_1 \Phi_1(U_g)|^2 \Big)\Big|_{r_2=0} \, d\mathbf{y} \, r_1 dr_1 \, r_3 dr_3 \\
&\quad + \frac{1}{10} \underbrace{ \int_{\mathbb{R}^{2m} \times \mathbb{R}_+^3} |\nabla_2 v_3|^2 |\nabla_1 \Phi_1(U_g)|^2 r_1 r_2 r_3 \, d\mathbf{y} d\mathbf{r} }_{= \mathcal{I}_{2, cross}} \\
&\quad + \int_{\mathbb{R}^{2m} \times \mathbb{R}_+^3} |v_3|^2 \mathcal{G}_2(\mathbf{y}, \mathbf{r}) r_1 r_2 r_3 \, d\mathbf{y} d\mathbf{r}.
\end{align*}
Thus, 
\begin{align}\label{eq:I2_cross_bound}
\mathcal{I}_{2, cross} &\le \frac{5}{9} \int_{\mathbb{R}^{2m} \times \mathbb{R}_+^2} \Big( |v_3|^2 |\nabla_1 \Phi_1(U_g)|^2 \Big)\Big|_{r_2=0} \, d\mathbf{y} \, r_1 dr_1 \, r_3 dr_3 \nonumber \\
&\quad + \frac{10}{9} \int_{\mathbb{R}^{2m} \times \mathbb{R}_+^3} |v_3|^2 \mathcal{G}_2(\mathbf{y}, \mathbf{r}) r_1 r_2 r_3 \, d\mathbf{y} d\mathbf{r}.
\end{align}
We apply the   same harmonic product rule process to $\mathcal{I}_{2, bdry}$. The only difference is that the weight function is $\varphi(U_g|_{r_1=0})^2$ instead of $|\nabla_1 \Phi_1(U_g)|^2$. The $r_2$-integration by parts gives an evaluation at $r_2=0$, and Young's inequality yields a cross-term bound involving $\nabla_2 \Phi_1(U_g|_{r_1=0})$. Expanding this bound, separating the $\frac{1}{10} \mathcal{I}_{2, bdry}$ term and then absorbing it by the left-hand side gives:
\begin{align}\label{eq:I2_bdry_bound}
\mathcal{I}_{2, bdry} &\le \frac{5}{9} \int_{\mathbb{R}^{2m} \times \mathbb{R}_+} \Big( |v_3|^2 \varphi(U_g)^2 \Big)\Big|_{r_1=0, r_2=0} \, d\mathbf{y} \, r_3 dr_3 \nonumber \\
&\quad + \frac{10}{9} \int_{\mathbb{R}^{2m} \times \mathbb{R}_+^2} |v_3|_{r_1=0}^2 |\nabla_2 \Phi_1(U_g|_{r_1=0})|^2 \, d\mathbf{y} \, r_2 dr_2 \, r_3 dr_3.
\end{align}

We substitute the expanded, absorbed bounds \eqref{eq:I2_cross_bound} and \eqref{eq:I2_bdry_bound} back into our main inequality resulting from Iteration 1. Recall from equation \eqref{eq:iter1} that $\int_{A_\beta(\lambda)} S_{twist}(f)^2 d\mathbf{x} \le \frac{5}{9}\mathcal{I}_{2, bdry} + \frac{10}{9}\mathcal{I}_{2, cross}$. Substituting the bounded expansions isolates the full tri-parameter integral of $v_3$ and completely defines the explicit boundary terms:
\begin{align}\label{eq:iter2}
\int_{A_\beta(\lambda)} S_{twist}(f)^2 d\mathbf{x} &\le \frac{100}{81} \int_{\mathbb{R}^{2m} \times \mathbb{R}_+^3} |v_3|^2 \mathcal{G}_2(\mathbf{y}, \mathbf{r}) r_1 r_2 r_3 \, d\mathbf{y} d\mathbf{r}  + \mathcal{B}_1 + \mathcal{B}_2 + \mathcal{B}_3,
\end{align}
where the \textit{boundary terms} is defined as
\begin{align*}
\mathcal{B}_1 &:= \frac{50}{81} \int_{\mathbb{R}^{2m} \times \mathbb{R}_+^2} |v_3(\mathbf{y}, r_1, 0, r_3)|^2 |\nabla_1 \Phi_1(U_g(\mathbf{y}, r_1, 0, r_3))|^2 \, d\mathbf{y} \, r_1 dr_1 \, r_3 dr_3,\\
\mathcal{B}_2 &:= \frac{50}{81} \int_{\mathbb{R}^{2m} \times \mathbb{R}_+^2} |v_3(\mathbf{y}, 0, r_2, r_3)|^2 |\nabla_2 \Phi_1(U_g(\mathbf{y}, 0, r_2, r_3))|^2 \, d\mathbf{y} \, r_2 dr_2 \, r_3 dr_3,\\
\mathcal{B}_3&:= \frac{25}{81} \int_{\mathbb{R}^{2m} \times \mathbb{R}_+} |v_3(\mathbf{y}, 0, 0, r_3)|^2 \varphi(U_g(\mathbf{y}, 0, 0, r_3))^2 \, d\mathbf{y} \, r_3 dr_3.
\end{align*}

Since the Poisson kernels behave as approximate identities, evaluations at $r_1=0$ or $r_2=0$ correspond    to the pointwise limits in those parameters. These boundary integrals $\mathcal{B}_1, \mathcal{B}_2, \mathcal{B}_3$ now depend exclusively on the final undifferentiated target function $v_3 = \nabla_3 U_f$ and are  positioned to undergo the {third iteration} (passing $\nabla_3$ off $v_3$).

\subsection{Third iteration: passing $\nabla_3$ to $g$.}\label{sec:3.4}

Finally, we unpack the last target function $v_3 = \nabla_3 U_f$. (As established in Iteration 2, the scale weight $r_3$ is naturally incorporated into the volume measure). We are now evaluating integrals containing the term $|\nabla_3 U_f|^2 \mathcal{G}_2$, where the remaining undifferentiated factor is    $|U_f|^2$, and $\mathcal{G}_2$ is the smooth enveloping structure formed in the previous iterations.

We define the primary integral generated by Iteration 2 over the bulk space:
\begin{align*}
\mathcal{I}_{main} &:= \int_{\mathbb{R}^{2m} \times \mathbb{R}_+^3} |\nabla_3 U_f|^2 \mathcal{G}_2(\mathbf{y}, \mathbf{r}) r_1 r_2 r_3 \, d\mathbf{y} d\mathbf{r}.
\end{align*}
Since the third block is generated by the fiber integration along the diagonal in the twisted setting, its spatial gradient is the directional derivative $\nabla_{twist} = \nabla_{y_1} + \nabla_{y_2}$. We apply the harmonic product rule for the full twisted space-time Laplacian $\Delta_3 = \Delta_{twist} + \partial_{r_3}^2$ to the product $|U_f|^2 \mathcal{G}_2$.

The full extension $U_f$ is multi-harmonic, we then have $\Delta_3 U_f = 0$, which yields $\Delta_3(|U_f|^2) = 2|\nabla_3 U_f|^2$, where $\nabla_3 = (\nabla_{twist}, \partial_{r_3})$. Furthermore, the gradient expands as $\nabla_3(|U_f|^2) = 2 U_f \nabla_3 U_f$.\\
Expanding the twisted Laplacian of the product gives:
\begin{align*}
\Delta_3 \Big( |U_f|^2 \mathcal{G}_2 \Big) &= \Big( \Delta_3 |U_f|^2 \Big) \mathcal{G}_2 + |U_f|^2 \Delta_3 \mathcal{G}_2 + 2 \nabla_3 \big( |U_f|^2 \big) \cdot \nabla_3 \mathcal{G}_2 \\
&= 2 |\nabla_3 U_f|^2 \mathcal{G}_2 + |U_f|^2 \Delta_3 \mathcal{G}_2 + 4 U_f \nabla_3 U_f \cdot \nabla_3 \mathcal{G}_2.
\end{align*}

Rearranging algebraically to isolate the target term $|\nabla_3 U_f|^2 \mathcal{G}_2$, and dividing the entire equation by 2, we obtain the fundamental identity for the third iteration:
\begin{align}\label{nabla3 U}
|\nabla_3 U_f|^2 \mathcal{G}_2 &= \frac{1}{2} \Delta_3 \Big( |U_f|^2 \mathcal{G}_2 \Big)  - 2 U_f \nabla_3 U_f \cdot \nabla_3 \mathcal{G}_2  - \frac{1}{2} |U_f|^2 \Delta_3 \mathcal{G}_2.
\end{align}

For the leading term $ \frac{1}{2} \Delta_3 \Big( |U_f|^2 \mathcal{G}_2 \Big) $, we separate the spatial (twisted) and scale components explicitly:
\begin{align*}
\frac{1}{2} \iint_{\mathbb{R}^{2m} \times \mathbb{R}_+} \Delta_3 \Big( |U_f|^2 \mathcal{G}_2 \Big) r_3 dr_3 d\mathbf{y} 
&= \frac{1}{2} \int_{\mathbb{R}_+} \left( \int_{\mathbb{R}^{2m}} \Delta_{twist} \Big( |U_f|^2 \mathcal{G}_2 \Big) d\mathbf{y} \right) r_3 dr_3 \\
&+ \frac{1}{2} \int_{\mathbb{R}^{2m}} \left( \int_0^\infty \partial_{r_3}^2 \Big( |U_f|^2 \mathcal{G}_2 \Big) r_3 dr_3 \right) d\mathbf{y}.
\end{align*}
Since the twisted Laplacian $\Delta_{twist}$ is the sum of second-order directional derivatives along the diagonal, we can apply the standard divergence theorem over $\mathbb{R}^{2m}$. The spatial integral completely vanishes due to the rapid polynomial decay of the Poisson extension $U_f$ and its derivatives at spatial infinity.

For the scale integral, let $F(r_3) = |U_f|^2 \mathcal{G}_2$. A third integration by parts once over $r_3$ yields:
\begin{align*}
\int_0^\infty \partial_{r_3}^2 \Big( |U_f|^2 \mathcal{G}_2 \Big) r_3 dr_3 = \Big[ r_3 \partial_{r_3} \Big( |U_f|^2 \mathcal{G}_2 \Big) \Big]_0^\infty - \Big[ |U_f|^2 \mathcal{G}_2 \Big]_0^\infty.
\end{align*}
We evaluate the boundaries:
\begin{itemize}
\item As $r_3 \to \infty$, the Poisson extension $U_f \to 0$, forcing the upper boundary evaluation to be   zero.
\item As $r_3 \to 0$, the $r_3$-weighted gradient of the harmonic extension vanishes, ensuring the term $r_3 \partial_{r_3} \big( |U_f|^2 \mathcal{G}_2 \big)$ vanishes at the origin.
\end{itemize}

Thus, 
\begin{align*}
\frac{1}{2} \iint_{\mathbb{R}^{2m} \times \mathbb{R}_+} \Delta_3 \Big( |U_f|^2 \mathcal{G}_2 \Big) r_3 dr_3 d\mathbf{y} = {1\over2} F(0)= \frac{1}{2} \int_{\mathbb{R}^{2m}} \Big( |U_f(\mathbf{y}, r_1, r_2, 0)|^2 \mathcal{G}_2(\mathbf{y}, r_1, r_2, 0) \Big) d\mathbf{y}.
\end{align*}

We now consider the second term $ - 2 U_f \nabla_3 U_f \cdot \nabla_3 \mathcal{G}_2$ in the right-hand side of \eqref{nabla3 U}. From iteration 2, our enveloping geometric weight is explicitly structured as the sum of the squared norms of two smooth tensor-valued functions: $\mathcal{G}_2 = |W_A|^2 + |W_B|^2$, where $W_A = \Psi_A(U_g) (\nabla_1 U_g \otimes \nabla_2 U_g)$ and $W_B = \Psi_B(U_g) (\nabla_1 \nabla_2 U_g)$. 

Remarking that $\mathcal{G}_2$ splits linearly, we then apply the twisted Laplacian $\Delta_3$ and the product rule distributively to the terms $|U_f|^2 |W_A|^2$ and $|U_f|^2 |W_B|^2$. We demonstrate the bound for the generic component $W \in \{W_A, W_B\}$. 

Recall that $\nabla_3 = (\nabla_{twist}, \partial_{r_3})$ is the full space-time twisted gradient. By the chain rule, $\nabla_3(|W|^2) = 2 W \cdot \nabla_3 W$ (where $\cdot$ denotes the appropriate tensor inner product contracting over the twisted directions). The cross-term for this component expands as:
\[
 \nabla_3(|U_f|^2) \cdot \nabla_3(|W|^2) = (2 U_f \nabla_3 U_f) \cdot (2 W \cdot \nabla_3 W) = 4 (W \nabla_3 U_f) \cdot (U_f \nabla_3 W).
\]
To avoid dividing by zero where the compactly supported weights vanish, we apply Young's inequality ($2ab \le \epsilon a^2 + \frac{1}{\epsilon} b^2$) with $\epsilon = \frac{1}{10}$ directly to this factored form, bypassing fractional division entirely:
\begin{align*}
\big| 4 (W \nabla_3 U_f) \cdot (U_f \nabla_3 W) \big| &\le \frac{1}{10} \big| W \nabla_3 U_f \big|^2 + 40 \big| U_f \nabla_3 W \big|^2 
= \frac{1}{10} |\nabla_3 U_f|^2 |W|^2 + 40 |U_f|^2 |\nabla_3 W|^2.
\end{align*}

Next we consider the last term $ - \frac{1}{2} |U_f|^2 \Delta_3 \mathcal{G}_2$ in the right-hand side of \eqref{nabla3 U}.
 Using the identity $\Delta_3(|W|^2) = 2|\nabla_3 W|^2 + 2 W \cdot \Delta_3 W$, we take the absolute value and get
\[
\left| \frac{1}{2} |U_f|^2 \Delta_3(|W|^2) \right| \le |U_f|^2 |\nabla_3 W|^2 + |U_f|^2 \big| W \cdot \Delta_3 W \big|.
\]
Summing these bounds over both components $W_A$ and $W_B$, the total pointwise error from the cross terms and twisted Laplacian terms acting on $|U_f|^2 \mathcal{G}_2$ is   bounded by
\[
\frac{1}{10} |\nabla_3 U_f|^2 \Big( |W_A|^2 + |W_B|^2 \Big) + |U_f|^2 \sum_{W \in \{W_A, W_B\}} \Big( 41 |\nabla_3 W|^2 + \big| W \cdot \Delta_3 W \big| \Big).
\]

Notice that the $\frac{1}{10}$ coefficient isolates our target twisted gradient $|\nabla_3 U_f|^2 \mathcal{G}_2$. The remaining terms are purely products of $|U_f|^2$ and smooth derivatives of the geometric weights. We group this sum into our final globally well-defined enveloping factor:
\[
\mathcal{G}_3(\mathbf{y}, \mathbf{r}) := \sum_{W \in \{W_A, W_B\}} \Big( 41 |\nabla_3 W|^2 + \big| W \cdot \Delta_3 W \big| \Big).
\]
Since $W_A$ and $W_B$ inherit the compact support of the original cut-off $\varphi$, their spatial and scale derivatives remain smoothly compactly supported. This guarantees that $\mathcal{G}_3$ is uniformly bounded.

\medskip
Recall from the end of Iteration 2 \eqref{eq:iter2} that our total bound for $\int_{A_\beta(\lambda)} S_{twist}(f)^2 d\mathbf{x}$ consists of the main tri-parameter integral and three inherited boundary integrals. To formalize the final integration by parts and bootstrapping on all of them simultaneously, we define the following four target integrals:
\begin{align*}
\mathcal{I}_{main} &:= \int_{\mathbb{R}^{2m} \times \mathbb{R}_+^3} |\nabla_3 U_f|^2 \mathcal{G}_2(\mathbf{y}, \mathbf{r}) r_1 r_2 r_3 \, d\mathbf{y} d\mathbf{r}, \\
\mathcal{I}_{\mathcal{B}_1} &:= \int_{\mathbb{R}^{2m} \times \mathbb{R}_+^2} |\nabla_3 U_f|_{r_2=0}^2 |\nabla_1 \Phi_1(U_g|_{r_2=0})|^2 \, d\mathbf{y} \, r_1 dr_1 \, r_3 dr_3, \\
\mathcal{I}_{\mathcal{B}_2} &:= \int_{\mathbb{R}^{2m} \times \mathbb{R}_+^2} |\nabla_3 U_f|_{r_1=0}^2 |\nabla_2 \Phi_1(U_g|_{r_1=0})|^2 \, d\mathbf{y} \, r_2 dr_2 \, r_3 dr_3, \\
\mathcal{I}_{\mathcal{B}_3} &:= \int_{\mathbb{R}^{2m} \times \mathbb{R}_+} |\nabla_3 U_f|_{r_1=0, r_2=0}^2 \varphi(U_g|_{r_1=0, r_2=0})^2 \, d\mathbf{y} \, r_3 dr_3.
\end{align*}

The total bound from Iteration 2 is    the linear combination of these integrals with their inherited coefficients:
\begin{equation}\label{eq:iter2_linear_combo}
\int_{A_\beta(\lambda)} S_{twist}(f)^2 d\mathbf{x} \le \frac{100}{81} \mathcal{I}_{main} + \frac{50}{81} \mathcal{I}_{\mathcal{B}_1} + \frac{50}{81} \mathcal{I}_{\mathcal{B}_2} + \frac{25}{81} \mathcal{I}_{\mathcal{B}_3}.
\end{equation}
Since all four of these integrals contain the target factor $|v_3|^2 = |\nabla_3 U_f|^2$ and are integrated over $r_3$, we apply the Iteration 3 product rule (using the full twisted Laplacian $\Delta_3$), $r_3$-integration by parts, and Young's inequality to all of them. 

For any of these four target integrals, let $\Omega$ denote the domain of the spatial and remaining scale variables, and let $d\omega$ denote its corresponding volume measure (e.g., $\Omega = \mathbb{R}^{2m} \times \mathbb{R}_+^2$ with $d\omega = d\mathbf{y} \, r_1 dr_1 \, r_2 dr_2$ for $\mathcal{I}_{main}$, whereas $\Omega = \mathbb{R}^{2m}$ with $d\omega = d\mathbf{y}$ for $\mathcal{I}_{\mathcal{B}_3}$). 

Thus, the generic integral takes the form $\mathcal{I} = \int_{\Omega \times \mathbb{R}_+} |\nabla_3 U_f|^2 W \, d\omega \, r_3 dr_3$ with a weight function $W$. This process expands the integral into a boundary evaluation over $\Omega$, a $\frac{1}{10}$ bootstrapping term over the full domain $\Omega \times \mathbb{R}_+$, and a bulk geometric error term:
\begin{align*}
\mathcal{I} &\le \int_{\Omega} \frac{1}{2} \Big( |U_f|^2 W \Big)\Big|_{r_3=0} \, d\omega+ \frac{1}{10} \underbrace{ \int_{\Omega \times \mathbb{R}_+} |\nabla_3 U_f|^2 W \, d\omega \, r_3 dr_3 }_{= \mathcal{I}} + \int_{\Omega \times \mathbb{R}_+} |U_f|^2 \mathcal{G}_W \, d\omega \, r_3 dr_3,
\end{align*}
where $\mathcal{G}_W$ is the enveloping smooth function containing the twisted derivatives $\nabla_3 W$ and $\Delta_3 W$. 
Thus,
\begin{equation}\label{eq:local_bootstrap}
\mathcal{I} \le \frac{5}{9} \int_{\Omega} \Big( |U_f|^2 W \Big)\Big|_{r_3=0} \, d\omega + \frac{10}{9} \int_{\Omega \times \mathbb{R}_+} |U_f|^2 \mathcal{G}_W \, d\omega \, r_3 dr_3.
\end{equation}
To ensure that the applications of Young's inequality to the inherited boundary integrals $\mathcal{I}_{\mathcal{B}_1}, \mathcal{I}_{\mathcal{B}_2}$, and $\mathcal{I}_{\mathcal{B}_3}$ do not produce zero denominators, we treat their base weight functions as squared norms of smooth vector fields. We define $W^{(1)}(\mathbf{y}, r_1, r_3) := \nabla_1 \Phi_1(U_g(\mathbf{y}, r_1, 0, r_3))$, $W^{(2)}(\mathbf{y}, r_2, r_3) := \nabla_2 \Phi_1(U_g(\mathbf{y}, 0, r_2, r_3))$ and $W^{(3)}(\mathbf{y}, r_3) := \varphi(U_g(\mathbf{y}, 0, 0, r_3)).$

Applying the twisted Laplacian $\Delta_3$ and the product rule to $|U_f|^2 |W^{(j)}|^2$, we construct the enveloping weights for the boundary terms: 
\begin{align*}
\mathcal{G}^{(1,3)}(\mathbf{y}, r_1, 0, r_3) &:= 41 \big| \nabla_3 W^{(1)} \big|^2 + \big| W^{(1)} \cdot \Delta_3 W^{(1)} \big|, \\
\mathcal{G}^{(2,3)}(\mathbf{y}, 0, r_2, r_3) &:= 41 \big| \nabla_3 W^{(2)} \big|^2 + \big| W^{(2)} \cdot \Delta_3 W^{(2)} \big|, \\
\mathcal{G}^{(3)}(\mathbf{y}, 0, 0, r_3) &:= 41 \big| \nabla_3 W^{(3)} \big|^2 + \big| W^{(3)} \Delta_3 W^{(3)} \big|.
\end{align*}
Since $\Phi_1$ and $\varphi$ are constant outside the interval $(C_1, 9/10)$, all mixed spatial and scale derivatives (including the twisted derivatives) contained within these explicitly defined envelopes identically vanish outside this region. This guarantees that $\mathcal{G}^{(1,3)}$, $\mathcal{G}^{(2,3)}$, and $\mathcal{G}^{(3)}$ remain globally smooth, compactly supported, and structured to define valid bounds.

Now, we substitute the locally bootstrapped bound \eqref{eq:local_bootstrap} for each of the four components ($\mathcal{I}_{main}$, $\mathcal{I}_{\mathcal{B}_1}$, $\mathcal{I}_{\mathcal{B}_2}$, $\mathcal{I}_{\mathcal{B}_3}$) directly back into top-level linear combination \eqref{eq:iter2_linear_combo}:
\begin{align*}
\int_{A_\beta(\lambda)} S_{twist}(f)^2 d\mathbf{x} 
&\le \frac{100}{81} \left( \frac{5}{9}\text{bdry}_{main} + \frac{10}{9}\text{bulk}_{main} \right) + \frac{50}{81} \left( \frac{5}{9}\text{bdry}_{\mathcal{B}_1} + \frac{10}{9}\text{bulk}_{\mathcal{B}_1} \right) \\
&\quad + \frac{50}{81} \left( \frac{5}{9}\text{bdry}_{\mathcal{B}_2} + \frac{10}{9}\text{bulk}_{\mathcal{B}_2} \right)  + \frac{25}{81} \left( \frac{5}{9}\text{bdry}_{\mathcal{B}_3} + \frac{10}{9}\text{bulk}_{\mathcal{B}_3} \right).
\end{align*}
Substituting these bounds and simplifying the coefficients, we obtain:
\begin{align}\label{eq:iter3}
\int_{A_\beta(\lambda)} S_{twist}(f)^2 d\mathbf{x} \le \mathcal{I}_{bulk} + \mathcal{B}_{final},
\end{align}
where we explicitly define the main tri-parameter bulk integral as:
\begin{align*}
\mathcal{I}_{bulk} := \frac{1000}{729} \int_{\mathbb{R}^{2m} \times \mathbb{R}_+^3} |U_f(\mathbf{y}, \mathbf{r})|^2 \mathcal{G}_3(\mathbf{y}, \mathbf{r}) r_1 r_2 r_3 \, d\mathbf{y} d\mathbf{r},
\end{align*}
and $\mathcal{B}_{final}$ is the explicit sum of the seven lower-dimensional boundary and mixed-bulk integrals generated during this final pass:
\begin{align*}
\mathcal{B}_{final} := \mathcal{T}_{main, bdry} + \mathcal{T}_{1, bulk} + \mathcal{T}_{1, bdry} + \mathcal{T}_{2, bulk} + \mathcal{T}_{2, bdry} + \mathcal{T}_{3, bulk} + \mathcal{T}_{3, bdry}.
\end{align*}
These seven terms are explicitly defined with their   multiplied coefficients as follows:

\medskip
\noindent{1. The $r_3=0$ boundary generated from the main tri-parameter integral:}
\begin{align*}
\mathcal{T}_{main, bdry} := \frac{500}{729} \int_{\mathbb{R}^{2m} \times \mathbb{R}_+^2} |U_f(\mathbf{y}, r_1, r_2, 0)|^2 \mathcal{G}_2(\mathbf{y}, r_1, r_2, 0) \, d\mathbf{y} \, r_1 dr_1 \, r_2 dr_2.
\end{align*}

\noindent{2. The processing of $\mathcal{I}_{\mathcal{B}_1}$ (originally over $r_1, r_3$):}
\begin{align*}
\mathcal{T}_{1, bulk} &:= \frac{500}{729} \int_{\mathbb{R}^{2m} \times \mathbb{R}_+^2} |U_f(\mathbf{y}, r_1, 0, r_3)|^2 \mathcal{G}^{(1,3)}(\mathbf{y}, r_1, 0, r_3) \, d\mathbf{y} \, r_1 dr_1 \, r_3 dr_3,\quad\text{and} \\
\mathcal{T}_{1, bdry} &:= \frac{250}{729} \int_{\mathbb{R}^{2m} \times \mathbb{R}_+} |U_f(\mathbf{y}, r_1, 0, 0)|^2 |\nabla_1 \Phi_1(U_g(\mathbf{y}, r_1, 0, 0))|^2 \, d\mathbf{y} \, r_1 dr_1.
\end{align*}

\medskip
\noindent{3. The processing of $\mathcal{I}_{\mathcal{B}_2}$ (originally over $r_2, r_3$):}
\begin{align*}
\mathcal{T}_{2, bulk} &:= \frac{500}{729} \int_{\mathbb{R}^{2m} \times \mathbb{R}_+^2} |U_f(\mathbf{y}, 0, r_2, r_3)|^2 \mathcal{G}^{(2,3)}(\mathbf{y}, 0, r_2, r_3) \, d\mathbf{y} \, r_2 dr_2 \, r_3 dr_3, \quad\text{and}\\
\mathcal{T}_{2, bdry} &:= \frac{250}{729} \int_{\mathbb{R}^{2m} \times \mathbb{R}_+} |U_f(\mathbf{y}, 0, r_2, 0)|^2 |\nabla_2 \Phi_1(U_g(\mathbf{y}, 0, r_2, 0))|^2 \, d\mathbf{y} \, r_2 dr_2.
\end{align*}

\medskip
\noindent{4. The processing of $\mathcal{I}_{\mathcal{B}_3}$ (originally over $r_3$ only):}
\begin{align*}
\mathcal{T}_{3, bulk} &:= \frac{250}{729} \int_{\mathbb{R}^{2m} \times \mathbb{R}_+} |U_f(\mathbf{y}, 0, 0, r_3)|^2 \mathcal{G}^{(3)}(\mathbf{y}, 0, 0, r_3) \, d\mathbf{y} \, r_3 dr_3, \quad\text{and}\\
\mathcal{T}_{3, bdry} &:= \frac{125}{729} \int_{\mathbb{R}^{2m}} |U_f(\mathbf{y}, 0, 0, 0)|^2 \varphi(U_g(\mathbf{y}, 0, 0, 0))^2 \, d\mathbf{y}.
\end{align*}
\subsection{$L^2$ Littlewood--Paley theory and the Chebyshev splitting} \label{sec:3.5}
From iteration 3, we have bounded the $L^2$ norm of the twisted area function over the proxy good set $A_\beta(\lambda)$ by a main volumetric bulk integral and seven lower-dimensional boundary integrals. We will bound these terms using the pointwise control of the Poisson extension and the $L^2$ boundedness of the twisted Area function (the twisted Littlewood-Paley theory), completely bypassing the need for Carleson measure or BMO theory.

\medskip
We first establish a pointwise bound on $|U_f(\mathbf{y}, \mathbf{r})|$ over the support of the integrals. By construction in Steps 1 through 3, the enveloping geometric weights $\mathcal{G}_3, \mathcal{G}^{(1,3)}, \mathcal{G}^{(2,3)}$, and $\mathcal{G}^{(3)}$ are all formed from smooth derivatives of the cut-off $\varphi(U_g)$. Since $\varphi(t)$ is identically constant for $t \le C_1$, its derivatives are supported where $U_g(\mathbf{y}, \mathbf{r}) > C_1$.

By Lemma \ref{lem:strict_separation}, we explicitly know that $U_g(\mathbf{y}, \mathbf{r}) \le C_1$ everywhere outside the enlarged tent $\widetilde{W}_\beta$. Therefore, the support of all the integration weights is contained within $\widetilde{W}_\beta$.

Let $(\mathbf{y}, \mathbf{r}) \in \widetilde{W}_\beta = \bigcup_{\mathbf{z} \in E_\beta(\lambda)} \Gamma^\beta(\mathbf{z})$. By definition, there exists at least one base point $\mathbf{z} \in E_\beta(\lambda)$ such that $(\mathbf{y}, \mathbf{r}) \in \Gamma^\beta(\mathbf{z})$. Since $\mathbf{z}$ belongs to the good set $E_\beta(\lambda)$, its non-tangential maximal function satisfies $U_f^*(\mathbf{z}) \le \lambda$. 

Since the supremum over the cone at $\mathbf{z}$ is bounded by $\lambda$, it follows that the extension evaluated at the specific point inside the cone is also bounded:
\begin{equation}\label{eq:pointwise_lambda}
|U_f(\mathbf{y}, \mathbf{r})| \le U_f^*(\mathbf{z}) \le \lambda.
\end{equation}
Thus, $|U_f(\mathbf{y}, \mathbf{r})| \le \lambda$ uniformly over the entire domain of integration for all bulk and mixed-boundary terms.

\medskip
We evaluate the main bulk integral and the intermediate mixed-boundary terms from equation \eqref{eq:iter3}. Applying the pointwise bound \eqref{eq:pointwise_lambda} over $\widetilde{W}_\beta$, we pull $\lambda^2$ out of all these integrals.

To apply the $L^2$ theory, we again use the indicator of the bad set, $h(\mathbf{y}) = \chi_{E_\beta(\lambda)^c}(\mathbf{y})$. Since $U_g + U_h = 1$, their spatial and scale gradients are   negatives: $\nabla_j U_g = - \nabla_j U_h$. Squaring this removes the sign, allowing us to majorize the geometric weight envelopes entirely by the mixed gradients of $U_h$.

We can group the integrals into three categories based on the number of active scale parameters (i.e., parameters integrated over $\mathbb{R}_+$ rather than evaluated at $0$).

\medskip
\noindent {I. The tri-parameter bulk term:}\\
Recall from Section \ref{sec:3.4} that 
\begin{align*}
\mathcal{I}_{bulk} &:= \frac{1000}{729} \iint_{\mathbb{R}^{2m} \times \mathbb{R}_+^3} |U_f(\mathbf{y}, \mathbf{r})|^2 \mathcal{G}_3(\mathbf{y}, \mathbf{r}) r_1 r_2 r_3 \, d\mathbf{y} d\mathbf{r}\\
&\le \frac{1000}{729} \lambda^2 \iint_{\mathbb{R}^{2m} \times \mathbb{R}_+^3} \mathcal{G}_3(\mathbf{y}, \mathbf{r}) r_1 r_2 r_3 \, d\mathbf{y} d\mathbf{r},
\end{align*}
where we used \eqref{eq:pointwise_lambda}.

To bound this integral, we expand the structure of the enveloping weight $\mathcal{G}_3(\mathbf{y}, \mathbf{r})$. Recall from Iteration 2 that the tensor-valued functions $W_A$ and $W_B$ originate from the derivatives of the smooth cut-off composition $\Phi_1(U_g)$:
\begin{align*}
W_A &:= \Psi_A(U_g) \big( \nabla_1 U_g \otimes \nabla_2 U_g \big), \quad\text{and}\quad
W_B := \Psi_B(U_g) \big( \nabla_1 \nabla_2 U_g \big)
\end{align*}
where $\Psi_A(t)=\Big(83|\Phi_1''(t)|^2+|\Phi_1'(t)\Phi_1'''(t)|\Big)^{1/2}$ and $
\Psi_B(t)=\Big(83|\Phi_1'(t)|^2|\Big)^{1/2}$.

Recall from Section \ref{sec:3.4} that 
\begin{align*}
\mathcal{G}_3(\mathbf{y}, \mathbf{r}) := \sum_{W \in \{W_A, W_B\}} \Big( 41 |\nabla_3 W|^2 + \big| W \cdot \Delta_3 W \big| \Big).
\end{align*}
Since the cut-off function $\Phi_1$ is smooth and constant outside the interval $(C_1, 9/10)$, its derivatives are uniformly bounded. Applying the product and chain rules to expand $\nabla_3$ and $\Delta_3$ acting on $W_A$ and $W_B$, the weight $\mathcal{G}_3$ is pointwise bounded by a polynomial combination of purely geometric mixed gradients.
\begin{itemize}
\item Estimates of $|\nabla_3 W_A|^2$:
\begin{align*}
    \nabla_3 W_A
    &=\Psi_A'(U_g) \nabla_3 U_g \otimes \nabla_1 U_g \otimes \nabla_2 U_g \\
    &\quad + \Psi_A(U_g) (\nabla_3 \nabla_1 U_g \otimes \nabla_2 U_g + \nabla_1 U_g \otimes \nabla_3 \nabla_2 U_g ).
\end{align*}
Hence, we have
\begin{align*}
    |\nabla_3 W_A|^2
    &\lesssim |\nabla_1 U_g|^2 |\nabla_2 U_g|^2 |\nabla_3 U_g|^2+|\nabla_3 \nabla_1 U_g|^2 |\nabla_2 U_g|^2+|\nabla_1 U_g|^2 |\nabla_3 \nabla_2 U_g|^2.
\end{align*}

\item Estimates of $|W_A\cdot \Delta_3W_A|$:
\begin{align*}
    \Delta_3W_A
    &=\Psi_A''(U_g)|\nabla_3 U_g|^2 \nabla_1 U_g \otimes \nabla_2 U_g \\
    &\quad + 2\Psi_A'(U_g) \nabla_3 U_g \otimes (\nabla_3\nabla_1 U_g \otimes \nabla_2 U_g + \nabla_1 U_g \otimes \nabla_3\nabla_2 U_g) \\
    &\quad + 2\Psi_A(U_g) \nabla_3 \nabla_1 U_g \otimes \nabla_3\nabla_2 U_g.
\end{align*}
Thus, we have
\begin{align*}
    |W_A\cdot \Delta_3W_A|
    &\lesssim |\nabla_1 U_g|^2 |\nabla_2 U_g|^2 |\nabla_3 U_g|^2+|\nabla_2 U_g|^2 |\nabla_3 \nabla_1 U_g|^2 +|\nabla_1 U_g|^2 |\nabla_3 \nabla_2 U_g|^2.
\end{align*}

\item Estimates of $|\nabla_3 W_B|^2$:
\[
    \nabla_3W_B=\Psi_B'(U_g) \nabla_3 U_g \otimes \nabla_2 \nabla_1 U_g + \Psi_B (U_g) \nabla_3 \nabla_2 \nabla_1 U_g,
\]
which implies that
\[
    |\nabla_3 W_B|^2\lesssim |\nabla_3 U_g|^2 |\nabla_2 \nabla_1 U_g|^2 + |\nabla_3\nabla_2 \nabla_1 U_g|^2.
\]

\item Estimates of $|W_B\cdot \Delta_3W_B|$: Since 
\[
    \Delta_3 W_B=\Psi_B''(U_g) |\nabla_3 U_g|^2 \nabla_2 \nabla_1 U_g + 2\Psi_B'(U_g) \nabla_3 U_g \otimes \nabla_3 \nabla_2 \nabla_1 U_g,
\]
we thus have
\begin{align*}
    |W_B\cdot \Delta_3 W_B|
    &\le |\Psi_B(U_g) \Psi_B''(U_g)| |\nabla_3 U_g|^2 |\nabla_2 \nabla_1 U_g|^2 \\
    &\quad + 2|\Psi_B(U_g) \Psi_B'(U_g)| |\nabla_2 \nabla_1 U_g| |\nabla_3 U_g| |\nabla_3\nabla_2 \nabla_1 U_g|\\
    &\lesssim |\nabla_3 U_g|^2 |\nabla_2 \nabla_1 U_g|^2 + |\nabla_3\nabla_2 \nabla_1 U_g|^2.
\end{align*}
\end{itemize}

Let $\mathcal{M}(U_g)$ denote this explicit envelope:
\begin{align*}
\mathcal{G}_3(\mathbf{y}, \mathbf{r}) \lesssim \mathcal{M}(U_g) &:= |\nabla_1 \nabla_2 \nabla_3 U_g|^2 \\
&\quad + |\nabla_3 U_g|^2 |\nabla_1 \nabla_2 U_g|^2 + |\nabla_1 \nabla_3 U_g|^2 |\nabla_2 U_g|^2 + |\nabla_1 U_g|^2 |\nabla_2 \nabla_3 U_g|^2 \\
&\quad + |\nabla_3 U_g|^2 |\nabla_1 U_g|^2 |\nabla_2 U_g|^2.
\end{align*}
Although $\mathcal{M}(U_g)$ contains products of mixed lower-order gradients, their integrals are  bounded by the integral of the full mixed top-order gradient via integration by parts. We now first consider $|\nabla_1 U_g|^2 |\nabla_2 \nabla_3 U_g|^2$. Using the multi-harmonic identity $|\nabla_1 U_g|^2 = \frac{1}{2}\Delta_1(U_g^2)$ and integrating by parts over the block variable $r_1 dr_1 d\mathbf{y}_1$, the Laplacian shifts onto the neighboring gradient:
\begin{align*}
\iint_{\mathbb{R}^m \times \mathbb{R}_+} |\nabla_1 U_g|^2 |\nabla_2 \nabla_3 U_g|^2 r_1 dr_1 d\mathbf{y}_1 &= \frac{1}{2} \iint_{\mathbb{R}^m \times \mathbb{R}_+} \Delta_1(U_g^2) |\nabla_2 \nabla_3 U_g|^2 r_1 dr_1 d\mathbf{y}_1 \\
&= \frac{1}{2} \iint_{\mathbb{R}^m \times \mathbb{R}_+} U_g^2 \Delta_1\big(|\nabla_2 \nabla_3 U_g|^2\big) r_1 dr_1 d\mathbf{y}_1.
\end{align*}

We  now claim that we have $\Delta_1(|\nabla_2 \nabla_3 U_g|^2) = 2|\nabla_1 \nabla_2 \nabla_3 U_g|^2$. 

To see this identity, we express the squared norm of the mixed tensor $\nabla_2 \nabla_3 U_g$ in terms of its scalar components. Let $W = \nabla_2 \nabla_3 U_g$, such that $|W|^2 = \sum_{\alpha} (W_{\alpha})^2$, where the index $\alpha$ enumerates all the individual scalar partial derivatives comprising the mixed gradient in the second and third blocks.\\
Applying the linear operator $\Delta_1$ to this sum, we evaluate the action of the Laplacian on each scalar component using the standard product identity $\Delta(u^2) = 2u\Delta u + 2|\nabla u|^2$. Therefore,
\begin{align*}
\Delta_1\big(|\nabla_2 \nabla_3 U_g|^2\big) = \Delta_1 \left( \sum_{\alpha} (W_{\alpha})^2 \right) 
&= \sum_{\alpha} \Delta_1 \big( (W_{\alpha})^2 \big) = \sum_{\alpha} \Big( 2 W_{\alpha} \Delta_1(W_{\alpha}) + 2 |\nabla_1 W_{\alpha}|^2 \Big).
\end{align*}
Since $\Delta_1$, $\nabla_2$, and $\nabla_3$ all commute, by the established multi-harmonicity of the Poisson extension, we know that $\Delta_1 U_g = 0$. Therefore, the Laplacian of each component vanishes:
\begin{align*}
\Delta_1(W_{\alpha}) = \Delta_1 \big( \partial_\alpha U_g \big) = \partial_\alpha \big( \Delta_1 U_g \big) = 0.
\end{align*}
Consequently, the first term in the summation is identically zero for every component $\alpha$. The remaining terms    reconstruct the squared norm of the full tri-parameter mixed gradient:
\begin{align*}
\Delta_1\big(|\nabla_2 \nabla_3 U_g|^2\big) = \sum_{\alpha} 2 |\nabla_1 W_{\alpha}|^2 
&= 2 \sum_{\alpha} \big|\nabla_1 (\text{component}_\alpha \text{ of } \nabla_2 \nabla_3 U_g)\big|^2 
= 2 |\nabla_1 \nabla_2 \nabla_3 U_g|^2.
\end{align*}

Furthermore, since $U_g$ is the extension of a characteristic function, we have $|U_g| \le 1$. Substituting these yields    the top-order bound:
\begin{align*}
\iint_{\mathbb{R}^m \times \mathbb{R}_+} U_g^2 |\nabla_1 \nabla_2 \nabla_3 U_g|^2 r_1 dr_1 d\mathbf{y}_1 \le \iint_{\mathbb{R}^m \times \mathbb{R}_+} |\nabla_1 \nabla_2 \nabla_3 U_g|^2 r_1 dr_1 d\mathbf{y}_1.
\end{align*}

By applying this   integration by parts recursively to all lower-order combinations in $\mathcal{M}(U_g)$, the entire volumetric integral of $\mathcal{G}_3$ is majorized by the integral of the fully mixed gradient:
\begin{align*}
\iint_{\mathbb{R}^{2m} \times \mathbb{R}_+^3} \mathcal{G}_3(\mathbf{y}, \mathbf{r}) r_1 r_2 r_3 \, d\mathbf{y} d\mathbf{r} \lesssim \iint_{\mathbb{R}^{2m} \times \mathbb{R}_+^3} |\nabla_1 \nabla_2 \nabla_3 U_g(\mathbf{y}, \mathbf{r})|^2 r_1 r_2 r_3 \, d\mathbf{y} d\mathbf{r}.
\end{align*}
We now pass this bound to the indicator of the bad set, $h(\mathbf{y}) = \chi_{E_\beta(\lambda)^c}(\mathbf{y})$. Since $g(\mathbf{y}) + h(\mathbf{y}) = 1$ everywhere on $\mathbb{R}^{2m}$, their Poisson extensions satisfy $U_g + U_h = 1$. Differentiating with respect to the block variables yields $\nabla_j U_g = - \nabla_j U_h$. Squaring eliminates the sign, meaning the fully mixed gradients are identical:
\begin{align*}
|\nabla_1 \nabla_2 \nabla_3 U_g(\mathbf{y}, \mathbf{r})|^2 = |\nabla_1 \nabla_2 \nabla_3 U_h(\mathbf{y}, \mathbf{r})|^2.
\end{align*}
Substituting this into the bound for $\mathcal{I}_{bulk}$, we obtain:
\begin{align}\label{eq:bulk_Uh_bound}
\mathcal{I}_{bulk} \lesssim \lambda^2 \iint_{\mathbb{R}^{2m} \times \mathbb{R}_+^3} |\nabla_1 \nabla_2 \nabla_3 U_h(\mathbf{y}, \mathbf{r})|^2 r_1 r_2 r_3 \, d\mathbf{y} d\mathbf{r}.
\end{align}

Finally, we demonstrate that integrating these mixed gradients over the full half-space recovers    the $L^2$ norm of the tri-parameter twisted Area function of $h$. By definition:
\begin{align*}
\| S_{twist}^{(1,2,3)}(h) \|_{L^2(\mathbb{R}^{2m})}^2 &= \int_{\mathbb{R}^{2m}} \left( \iint_{\Gamma^\beta(\mathbf{x})} |r_1\nabla_1 r_2\nabla_2 r_3\nabla_3 U_h(\mathbf{y}, \mathbf{r})|^2 \frac{1}{r_1 r_2 r_3 V_{\mathbf{r}}} d\mathbf{y} d\mathbf{r} \right) d\mathbf{x} \\
&= \int_{\mathbb{R}^{2m}} \left( \iint_{\mathbf{y} \in T(\mathbf{x}, \beta\mathbf{r})} |\nabla_1 \nabla_2 \nabla_3 U_h(\mathbf{y}, \mathbf{r})|^2 \frac{r_1 r_2 r_3}{V_{\mathbf{r}}} d\mathbf{y} d\mathbf{r} \right) d\mathbf{x}.
\end{align*}
We apply Fubini's Theorem to interchange the order of integration. By the symmetry of the twisted tubes, the geometric condition $\mathbf{y} \in T(\mathbf{x}, \beta\mathbf{r})$ is equivalent to $\mathbf{x} \in T(\mathbf{y}, \beta\mathbf{r})$. Swapping the integrals yields:
\begin{align*}
\| S_{twist}^{(1,2,3)}(h) \|_{L^2(\mathbb{R}^{2m})}^2 &= \iint_{\mathbb{R}^{2m} \times \mathbb{R}_+^3} |\nabla_1 \nabla_2 \nabla_3 U_h(\mathbf{y}, \mathbf{r})|^2 \frac{r_1 r_2 r_3}{V_{\mathbf{r}}} \left( \int_{T(\mathbf{y}, \beta\mathbf{r})} d\mathbf{x} \right) d\mathbf{y} d\mathbf{r}.
\end{align*}
The inner integral over $d\mathbf{x}$ evaluates    to the volume of the twisted tube $|T(\mathbf{y}, \beta\mathbf{r})|$. By the scaling properties of the twisted geometry, $|T(\mathbf{y}, \beta\mathbf{r})| = \beta^{3m} |T(\mathbf{y}, \mathbf{r})| = \beta^{3m} V_{\mathbf{r}}$. 

Substituting this volume back into the integral, the $V_{\mathbf{r}}$ terms cancel:
\begin{align*}
\| S_{twist}^{(1,2,3)}(h) \|_{L^2(\mathbb{R}^{2m})}^2 &= \beta^{3m} \iint_{\mathbb{R}^{2m} \times \mathbb{R}_+^3} |\nabla_1 \nabla_2 \nabla_3 U_h(\mathbf{y}, \mathbf{r})|^2 r_1 r_2 r_3 \, d\mathbf{y} d\mathbf{r}.
\end{align*}
Since the aperture $\beta$ is a fixed geometric constant chosen in Theorem \ref{thm:good_lambda}, the full half-space integral in \eqref{eq:bulk_Uh_bound} is proportional to the $L^2$ norm of the area function. Thus, 
\begin{align*}
\mathcal{I}_{bulk} \lesssim \lambda^2 \| S_{twist}^{(1,2,3)}(h) \|_{L^2(\mathbb{R}^{2m})}^2.
\end{align*}

\medskip
\noindent {2. The two-parameter mixed terms:}
\begin{itemize}
    \item $\mathcal{T}_{main, bdry}$ (evaluated at $r_3=0$, integrated over $r_1, r_2$)
    \item $\mathcal{T}_{1, bulk}$ (evaluated at $r_2=0$, integrated over $r_1, r_3$)
    \item $\mathcal{T}_{2, bulk}$ (evaluated at $r_1=0$, integrated over $r_2, r_3$)
\end{itemize}

Before bounding these terms, we explicitly define the corresponding bi-parameter twisted area functions. When the twisted Poisson extension is evaluated at a boundary $r_j=0$, the approximate identity property of the Poisson kernel effectively reduces the operator to a bi-parameter extension in the remaining blocks. For example, $U_h(\mathbf{y}, r_1, r_2, 0)$ is a bi-parameter harmonic extension.\\
We define the corresponding bi-parameter non-tangential cone $\Gamma^{(1,2),\beta}(\mathbf{x})$ by restricting the full twisted cone to the subspace $r_3=0$. The volume of the bi-parameter cross-section of the twisted tube is denoted $V_{r_1, r_2} \simeq r_1^m r_2^m$. The bi-parameter twisted area function acting on the $r_1, r_2$ blocks is defined explicitly as:
\begin{align*}
S_{twist}^{(1,2)}(h)(\mathbf{x}) := \left( \iint_{\Gamma^{(1,2),\beta}(\mathbf{x})} |r_1\nabla_1 r_2\nabla_2 U_h(\mathbf{y}, r_1, r_2, 0)|^2 \frac{1}{r_1 r_2 V_{r_1, r_2}} \, d\mathbf{y} dr_1 dr_2 \right)^{1/2}.
\end{align*}
The operators $S_{twist}^{(1,3)}(h)$ and $S_{twist}^{(2,3)}(h)$ are defined by explicitly restricting the twisted geometry and the Poisson extension to their respective active blocks, limiting the inactive scale parameter to zero. 

For the $(1,3)$ blocks, the non-tangential cone $\Gamma^{(1,3),\beta}(\mathbf{x})$ restricts the full twisted cone to the subspace $r_2=0$, with the bi-parameter cross-sectional volume denoted $V_{r_1, r_3} \simeq r_1^m r_3^m$. The corresponding area function is defined as:
\begin{align*}
S_{twist}^{(1,3)}(h)(\mathbf{x}) := \left( \iint_{\Gamma^{(1,3),\beta}(\mathbf{x})} |r_1\nabla_1 r_3\nabla_3 U_h(\mathbf{y}, r_1, 0, r_3)|^2 \frac{1}{r_1 r_3 V_{r_1, r_3}} \, d\mathbf{y} dr_1 dr_3 \right)^{1/2}.
\end{align*}

Similarly, for the $(2,3)$ blocks, the cone $\Gamma^{(2,3),\beta}(\mathbf{x})$ restricts the geometry to the subspace $r_1=0$, with the cross-sectional volume $V_{r_2, r_3} \simeq r_2^m r_3^m$. The $(2,3)$ area function is defined as:
\begin{align*}
S_{twist}^{(2,3)}(h)(\mathbf{x}) := \left( \iint_{\Gamma^{(2,3),\beta}(\mathbf{x})} |r_2\nabla_2 r_3\nabla_3 U_h(\mathbf{y}, 0, r_2, r_3)|^2 \frac{1}{r_2 r_3 V_{r_2, r_3}} \, d\mathbf{y} dr_2 dr_3 \right)^{1/2}.
\end{align*}

We demonstrate the bounding process on the first of these mixed terms, $\mathcal{T}_{main, bdry}$. Recall its definition:
\begin{align*}
\mathcal{T}_{main, bdry} := \frac{500}{729} \int_{\mathbb{R}^{2m} \times \mathbb{R}_+^2} |U_f(\mathbf{y}, r_1, r_2, 0)|^2 \mathcal{G}_2(\mathbf{y}, r_1, r_2, 0) \, d\mathbf{y} \, r_1 dr_1 \, r_2 dr_2.
\end{align*}
Applying the pointwise bound $|U_f(\mathbf{y}, \mathbf{r})| \le \lambda$ uniformly over the support of $\mathcal{G}_2$, we pull out $\lambda^2$:
\begin{align*}
\mathcal{T}_{main, bdry} \le \frac{500}{729} \lambda^2 \iint_{\mathbb{R}^{2m} \times \mathbb{R}_+^2} \mathcal{G}_2(\mathbf{y}, r_1, r_2, 0) \, d\mathbf{y} \, r_1 dr_1 \, r_2 dr_2.
\end{align*}

We analyze the enveloping weight $\mathcal{G}_2(\mathbf{y}, r_1, r_2, 0)$. From its construction in Iteration 2, $\mathcal{G}_2$ consists of polynomial combinations of the mixed spatial and scale derivatives of $\Phi_1(U_g)$ with respect to the active blocks $1$ and $2$. Since the derivatives of $\Phi_1$ are bounded by absolute constants, $\mathcal{G}_2$ is pointwise majorized by a sum of purely geometric mixed gradients.

Applying the identical integration-by-parts "upgrading" logic established for the tri-parameter bulk term, the integral of this combination of lower-order mixed gradients over the $(r_1, r_2)$ half-space is bounded by the integral of the top-order fully mixed bi-parameter gradient:
\begin{align*}
&\iint_{\mathbb{R}^{2m} \times \mathbb{R}_+^2} \mathcal{G}_2(\mathbf{y}, r_1, r_2, 0) \, d\mathbf{y} \, r_1 dr_1 \, r_2 dr_2 \lesssim \iint_{\mathbb{R}^{2m} \times \mathbb{R}_+^2} |\nabla_1 \nabla_2 U_g(\mathbf{y}, r_1, r_2, 0)|^2 r_1 r_2 \, d\mathbf{y} dr_1 dr_2.
\end{align*}

Since $g + h = 1$, their bi-parameter extensions satisfy $\nabla_j U_g = - \nabla_j U_h$ for $j \in \{1, 2\}$. Thus, their squared mixed gradients are identical: $|\nabla_1 \nabla_2 U_g|^2 = |\nabla_1 \nabla_2 U_h|^2$. Substituting this into the bound yields:
\begin{align}\label{eq:2param_Uh_bound}
\mathcal{T}_{main, bdry} \lesssim \lambda^2 \iint_{\mathbb{R}^{2m} \times \mathbb{R}_+^2} |\nabla_1 \nabla_2 U_h(\mathbf{y}, r_1, r_2, 0)|^2 r_1 r_2 \, d\mathbf{y} dr_1 dr_2.
\end{align}

We now recover the $L^2$ norm of the bi-parameter area function via Fubini's Theorem. Squaring the explicit definition of $S_{twist}^{(1,2)}(h)$ and integrating over $d\mathbf{x}$, we have:
\begin{align*}
\| S_{twist}^{(1,2)}(h) \|_{L^2(\mathbb{R}^{2m})}^2 &= \int_{\mathbb{R}^{2m}} \left( \iint_{\Gamma^{(1,2),\beta}(\mathbf{x})} |\nabla_1 \nabla_2 U_h(\mathbf{y}, r_1, r_2, 0)|^2 \frac{r_1 r_2}{V_{r_1, r_2}} \, d\mathbf{y} dr_1 dr_2 \right) d\mathbf{x}.
\end{align*}
By the geometric symmetry of the bi-parameter restricted tubes, the non-tangential cone condition $\mathbf{y} \in T^{(1,2)}(\mathbf{x}, \beta r_1, \beta r_2)$ is equivalent to $\mathbf{x} \in T^{(1,2)}(\mathbf{y}, \beta r_1, \beta r_2)$. Swapping the order of integration gives:
\begin{align*}
\| S_{twist}^{(1,2)}(h) \|_{L^2(\mathbb{R}^{2m})}^2 &= \iint_{\mathbb{R}^{2m} \times \mathbb{R}_+^2} |\nabla_1 \nabla_2 U_h|^2 \frac{r_1 r_2}{V_{r_1, r_2}} \left( \int_{T^{(1,2)}(\mathbf{y}, \beta r_1, \beta r_2)} d\mathbf{x} \right) d\mathbf{y} dr_1 dr_2.
\end{align*}
The inner integral evaluates to the volume of the restricted tube, which scales    as $\beta^{2m} V_{r_1, r_2}$. The base volume $V_{r_1, r_2}$ cancels from the denominator:
\begin{align*}
\| S_{twist}^{(1,2)}(h) \|_{L^2(\mathbb{R}^{2m})}^2 &= \beta^{2m} \iint_{\mathbb{R}^{2m} \times \mathbb{R}_+^2} |\nabla_1 \nabla_2 U_h(\mathbf{y}, r_1, r_2, 0)|^2 r_1 r_2 \, d\mathbf{y} dr_1 dr_2.
\end{align*}
Since the aperture $\beta$ is fixed, the half-space integral in \eqref{eq:2param_Uh_bound} is bounded    by the $L^2$ norm. Therefore:
\begin{align*}
\mathcal{T}_{main, bdry} \lesssim \lambda^2 \| S_{twist}^{(1,2)}(h) \|_{L^2(\mathbb{R}^{2m})}^2.
\end{align*}

By identical symmetries acting on the weight $\mathcal{G}^{(1,3)}$ integrated over $r_1, r_3$ and the weight $\mathcal{G}^{(2,3)}$ integrated over $r_2, r_3$, we extract the corresponding bounds for the other two mixed integrals:
\begin{align*}
\mathcal{T}_{1, bulk} \lesssim \lambda^2 \| S_{twist}^{(1,3)}(h) \|_{L^2(\mathbb{R}^{2m})}^2 \quad\text{and}\quad
\mathcal{T}_{2, bulk} \lesssim \lambda^2 \| S_{twist}^{(2,3)}(h) \|_{L^2(\mathbb{R}^{2m})}^2.
\end{align*}

Summing these three results gives the complete bound for the bi-parameter mixed terms:
\begin{align*}
\mathcal{T}_{main, bdry} + \mathcal{T}_{1, bulk} + \mathcal{T}_{2, bulk} \lesssim \lambda^2 \left( \| S_{twist}^{(1,2)}(h) \|_{L^2}^2 + \| S_{twist}^{(1,3)}(h) \|_{L^2}^2 + \| S_{twist}^{(2,3)}(h) \|_{L^2}^2 \right).
\end{align*}

\medskip
\noindent {3. The one-parameter mixed terms:}
\begin{itemize}
    \item $\mathcal{T}_{1, bdry}$ (evaluated at $r_2=r_3=0$, integrated over $r_1$),
    \item $\mathcal{T}_{2, bdry}$ (evaluated at $r_1=r_3=0$, integrated over $r_2$),
    \item $\mathcal{T}_{3, bulk}$ (evaluated at $r_1=r_2=0$, integrated over $r_3$).
\end{itemize}

Before bounding these terms, we explicitly define the corresponding 1-parameter twisted area functions. By the approximate identity property of the Poisson kernel, evaluating the extension at two boundary limits simultaneously (e.g., $r_2=0$ and $r_3=0$) reduces the operator to a standard 1-parameter Poisson extension in the remaining active block, albeit embedded within the twisted multi-parameter space. 

We define the 1-parameter non-tangential cone $\Gamma^{(1),\beta}(\mathbf{x})$ by restricting the full twisted cone to the subspace $r_2=0, r_3=0$. The volume of this 1-parameter cross-section is denoted $V_{r_1} \simeq r_1^m$. The 1-parameter twisted area function acting purely on the $r_1$ block is defined explicitly as:
\begin{align*}
S_{twist}^{(1)}(h)(\mathbf{x}) := \left( \iint_{\Gamma^{(1),\beta}(\mathbf{x})} |r_1\nabla_1 U_h(\mathbf{y}, r_1, 0, 0)|^2 \frac{1}{r_1 V_{r_1}} \, d\mathbf{y} dr_1 \right)^{1/2}.
\end{align*}
By identical restrictions, the cone $\Gamma^{(2),\beta}(\mathbf{x})$ restricts the geometry to the subspace $r_1=0, r_3=0$ with cross-sectional volume $V_{r_2} \simeq r_2^m$. The corresponding area function is:
\begin{align*}
S_{twist}^{(2)}(h)(\mathbf{x}) := \left( \iint_{\Gamma^{(2),\beta}(\mathbf{x})} |r_2\nabla_2 U_h(\mathbf{y}, 0, r_2, 0)|^2 \frac{1}{r_2 V_{r_2}} \, d\mathbf{y} dr_2 \right)^{1/2}.
\end{align*}

Finally, the cone $\Gamma^{(3),\beta}(\mathbf{x})$ restricts the geometry to the subspace $r_1=0, r_2=0$ with cross-sectional volume $V_{r_3} \simeq r_3^m$. The third 1-parameter area function is
\begin{align*}
S_{twist}^{(3)}(h)(\mathbf{x}) := \left( \iint_{\Gamma^{(3),\beta}(\mathbf{x})} |r_3\nabla_3 U_h(\mathbf{y}, 0, 0, r_3)|^2 \frac{1}{r_3 V_{r_3}} \, d\mathbf{y} dr_3 \right)^{1/2}.
\end{align*}

We demonstrate the bounding process on $\mathcal{T}_{1, bdry}$. Recall that
\begin{align*}
\mathcal{T}_{1, bdry} := \frac{250}{729} \int_{\mathbb{R}^{2m} \times \mathbb{R}_+} |U_f(\mathbf{y}, r_1, 0, 0)|^2 |\nabla_1 \Phi_1(U_g(\mathbf{y}, r_1, 0, 0))|^2 \, d\mathbf{y} \, r_1 dr_1.
\end{align*}
Applying the pointwise bound $|U_f(\mathbf{y}, \mathbf{r})| \le \lambda$ uniformly over the support of the weight, we pull out the constant $\lambda^2$ and obtain
\begin{align*}
\mathcal{T}_{1, bdry} \le \frac{250}{729} \lambda^2 \iint_{\mathbb{R}^{2m} \times \mathbb{R}_+} |\nabla_1 \Phi_1(U_g(\mathbf{y}, r_1, 0, 0))|^2 r_1 \, d\mathbf{y} dr_1.
\end{align*}

We analyze the gradient weight. By the chain rule, $|\nabla_1 \Phi_1(U_g)|^2 = |\Phi_1'(U_g)|^2 |\nabla_1 U_g|^2$. Since the derivative $\Phi_1'$ is uniformly bounded by an absolute constant, this weight is bounded by the purely geometric 1-parameter gradient $|\nabla_1 U_g|^2$. Unlike the tri-parameter and bi-parameter terms, no further integration by parts is required here since the weight is already bounded by the top-order 1-parameter gradient.

Since $g + h = 1$, their 1-parameter extensions satisfy $\nabla_1 U_g = - \nabla_1 U_h$. Squaring this identity yields $|\nabla_1 U_g|^2 = |\nabla_1 U_h|^2$. Substituting this into the bound gives
\begin{align}\label{eq:1param_Uh_bound}
\mathcal{T}_{1, bdry} \lesssim \lambda^2 \iint_{\mathbb{R}^{2m} \times \mathbb{R}_+} |\nabla_1 U_h(\mathbf{y}, r_1, 0, 0)|^2 r_1 \, d\mathbf{y} dr_1.
\end{align}

We now recover the $L^2$ norm of the 1-parameter area function via Fubini's Theorem. Squaring the explicit definition of $S_{twist}^{(1)}(h)$ and integrating over $d\mathbf{x}$, we have
\begin{align*}
\| S_{twist}^{(1)}(h) \|_{L^2(\mathbb{R}^{2m})}^2 &= \int_{\mathbb{R}^{2m}} \left( \iint_{\Gamma^{(1),\beta}(\mathbf{x})} |\nabla_1 U_h(\mathbf{y}, r_1, 0, 0)|^2 \frac{r_1}{V_{r_1}} \, d\mathbf{y} dr_1 \right) d\mathbf{x}.
\end{align*}
By the geometric symmetry of the restricted tubes, the non-tangential cone condition $\mathbf{y} \in T^{(1)}(\mathbf{x}, \beta r_1)$ is equivalent to $\mathbf{x} \in T^{(1)}(\mathbf{y}, \beta r_1)$. Swapping the order of integration gives
\begin{align*}
\| S_{twist}^{(1)}(h) \|_{L^2(\mathbb{R}^{2m})}^2 &= \iint_{\mathbb{R}^{2m} \times \mathbb{R}_+} |\nabla_1 U_h|^2 \frac{r_1}{V_{r_1}} \left( \int_{T^{(1)}(\mathbf{y}, \beta r_1)} d\mathbf{x} \right) d\mathbf{y} dr_1.
\end{align*}
The inner integral evaluates to the volume of the restricted tube, scaling    as $\beta^m V_{r_1}$. The base volume $V_{r_1}$ cancels:
\begin{align*}
\| S_{twist}^{(1)}(h) \|_{L^2(\mathbb{R}^{2m})}^2 &= \beta^m \iint_{\mathbb{R}^{2m} \times \mathbb{R}_+} |\nabla_1 U_h(\mathbf{y}, r_1, 0, 0)|^2 r_1 \, d\mathbf{y} dr_1.
\end{align*}
Since the aperture $\beta$ is a fixed geometric constant, the half-space integral in \eqref{eq:1param_Uh_bound} is bounded    by the $L^2$ norm. Therefore,
\begin{align*}
\mathcal{T}_{1, bdry} \lesssim \lambda^2 \| S_{twist}^{(1)}(h) \|_{L^2(\mathbb{R}^{2m})}^2.
\end{align*}

By identical geometric symmetries acting on the weight $|\nabla_2 \Phi_1|^2$ integrated over $r_2$, and the weight $\mathcal{G}^{(3)}$ integrated over $r_3$, we extract the bounds for the other two 1-parameter mixed integrals and get
\begin{align*}
\mathcal{T}_{2, bdry} \lesssim \lambda^2 \| S_{twist}^{(2)}(h) \|_{L^2(\mathbb{R}^{2m})}^2, \quad\text{and}\quad
\mathcal{T}_{3, bulk} \lesssim \lambda^2 \| S_{twist}^{(3)}(h) \|_{L^2(\mathbb{R}^{2m})}^2.
\end{align*}

Summing these three results gives the complete bound for the 1-parameter mixed terms, then we have
\begin{align*}
\mathcal{T}_{1, bdry} + \mathcal{T}_{2, bdry} + \mathcal{T}_{3, bulk} \lesssim \lambda^2 \left( \| S_{twist}^{(1)}(h) \|_{L^2(\mathbb{R}^{2m})}^2 + \| S_{twist}^{(2)}(h) \|_{L^2(\mathbb{R}^{2m})}^2 + \| S_{twist}^{(3)}(h) \|_{L^2(\mathbb{R}^{2m})}^2 \right).
\end{align*}

\medskip
By the multi-parameter Littlewood--Paley theory, the twisted area function is bounded from $L^2(\mathbb{R}^{2m})$ to $L^2(\mathbb{R}^{2m})$ regardless of whether it is applied as a one-parameter, two-parameter, or three-parameter operator. 

We only verify the boundedness for $S_{twist}^{(3)}(h)$ as a representative case, while the remaining cases follow by an identical argument.

Note that $\mathcal{F}_y U_h(\xi_1,\xi_2,0,0,r_3)=e^{-r_3|\xi_1+\xi_2|}\hat h(\xi_1,\xi_2)$. Combining with (\ref{FourierofPoi}), we have
\begin{align*}
|\mathcal{F}_y(\nabla_3U_h)|^2
    &=\big|\mathcal{F}_y(\nabla_{y_1}U_h)+\mathcal{F}_y(\nabla_{y_2}U_h)\big|^2+|\mathcal{F}_y(\partial_{r_3}U_h)|^2\\
    &= |\xi_1+\xi_2|^2 |\mathcal{F}_y U_h|^2 + \Big|-|\xi_1+\xi_2| \mathcal{F}_y U_h\Big|^2 \\
    &= 2|\xi_1+\xi_2|^2 e^{-2r_3|\xi_1+\xi_2|}|\hat h(\xi_1,\xi_2)|^2.
\end{align*}
Then we can estimate the $L^2$ bound of $S_{twist}^{(3)}(h)$ via the Plancherel identity,
\begin{align*}
    \|S_{twist}^{(3)}(h)\|_{L^2(\mathbb{R}^{2m})}^2
    &\simeq \iint_{\mathbb{R}^{2m}\times\mathbb{R}_+} |\nabla_3 U_h(\mathbf{y},0,0,r_3)|^2 r_3 d\mathbf{y} dr_3\\
    &=\int_{\mathbb{R}_+}\Big(\int_{\mathbb{R}^{2m}} |\mathcal{F}_y(\nabla_3U_h)(\xi_1,\xi_2,0,0,r_3)|^2 d\xi_1 d\xi_2\Big) r_3 dr_3\\
    &=\int_{\mathbb{R}^{2m}} 2|\xi_1+\xi_2|^2 |\hat h(\xi_1,\xi_2)|^2 \Big(\int_{\mathbb{R}_+} e^{-2r_3|\xi_1+\xi_2|} r_3 dr_3\Big) d\xi_1 d\xi_2\\
    &= \int_{\mathbb{R}^{2m}} 2|\xi_1+\xi_2|^2 |\hat h(\xi_1,\xi_2)|^2 \left( \frac{1}{4|\xi_1+\xi_2|^2} \right) d\xi_1 d\xi_2\\
    &= \frac{1}{2} \int_{\mathbb{R}^{2m}}|\hat h(\xi_1,\xi_2)|^2 d\xi_1 d\xi_2\\
    &\simeq \|h\|_{L^2(\mathbb{R}^{2m})}^2.
\end{align*}
Therefore, each of these area function norms is bounded by the $L^2$ norm of the base function $h$,
\begin{align*}
\| S_{twist}^{(J)}(h) \|_{L^2(\mathbb{R}^{2m})}^2 &\le C \| h \|_{L^2(\mathbb{R}^{2m})}^2 = C \int_{\mathbb{R}^{2m}} \chi_{E_\beta(\lambda)^c}(\mathbf{x})^2 d\mathbf{x} 
= C \big|\big\{ \mathbf{x} \in \mathbb{R}^{2m} : U_f^*(\mathbf{x}) > \lambda \big\}\big|.
\end{align*}

Summing the bulk integral and all six intermediate mixed-boundary integrals, we obtain the unified bound
\begin{align*}
&\mathcal{I}_{bulk} + \mathcal{T}_{main, bdry} + \sum^2_{j=1}\left(\mathcal{T}_{j, bulk}+\mathcal{T}_{j, bdry}\right
)  + \mathcal{T}_{3, bulk}
\le C \lambda^2 \big|\big\{ \mathbf{x} \in \mathbb{R}^{2m} : U_f^*(\mathbf{x}) > \lambda \big\}\big|.
\end{align*}

\medskip
The only term remaining from the final assembly $\mathcal{B}_{final}$ is the pure boundary term $\mathcal{T}_{3, bdry}$, where all three scale parameters are evaluated at    zero:
\[
\mathcal{T}_{3, bdry} = \frac{125}{729} \int_{\mathbb{R}^{2m}} |U_f(\mathbf{y}, 0, 0, 0)|^2 \varphi(U_g(\mathbf{y}, 0, 0, 0))^2 d\mathbf{y}.
\]
As the scale parameters approach zero, the Poisson extensions converge almost everywhere to their initial base functions: $U_f(\mathbf{y}, 0, 0, 0) = f(\mathbf{y})$ and $U_g(\mathbf{y}, 0, 0, 0) = g(\mathbf{y}) = \chi_{E_\beta(\lambda)}(\mathbf{y})$.

Since $\varphi(1) = 1$ and $\varphi(0) = 0$, the composition $\varphi(g(\mathbf{y}))$ acts identically as the indicator for the good set $E_\beta(\lambda) = \{ \mathbf{x} \in \mathbb{R}^{2m} : U_f^*(\mathbf{x}) \le \lambda \}$. Furthermore, the base function is pointwise bounded by its own non-tangential maximal function almost everywhere: $|f(\mathbf{y})| \le U_f^*(\mathbf{y})$.

Substituting these limits directly into the integral restricts the domain to the good set, yielding:
\[
\mathcal{T}_{3, bdry} \le C \int_{\{ U_f^*(\mathbf{y}) \le \lambda \}} |f(\mathbf{y})|^2 d\mathbf{y} \le C \int_{\{ U_f^*(\mathbf{y}) \le \lambda \}} U_f^*(\mathbf{y})^2 d\mathbf{y}.
\]

\medskip
Summing the $L^2$ bound of the bulk and mixed-boundary terms  with the pure boundary evaluation, we have fully bounded the $L^2$ norm of the twisted area function over the proxy good set $A_\beta(\lambda)$:
\begin{equation}\label{eq:total_A_beta_bound}
\int_{A_\beta(\lambda)} S_{twist}(f)^2 d\mathbf{x} \le C \lambda^2 \big|\big\{ \mathbf{x} \in \mathbb{R}^{2m} : U_f^*(\mathbf{x}) > \lambda \big\}\big| + C \int_{\{ U_f^*(\mathbf{x}) \le \lambda \}} U_f^*(\mathbf{x})^2 d\mathbf{x}.
\end{equation}

To conclude, we evaluate the distribution function of the twisted area function. We split the measure of the super-level set of $S_{twist}(f)$ using $A_\beta(\lambda)$:
\begin{align*}
\big|\big\{\mathbf{x} \in \mathbb{R}^{2m}: S_{twist}(f)(\mathbf{x}) > \lambda\big\}\big| &\leq \big|A_\beta(\lambda)^c\big| + \big|\big\{\mathbf{x} \in A_\beta(\lambda) : S_{twist}(f)(\mathbf{x}) > \lambda\big\}\big|.
\end{align*}

For the first term, the fundamental definition of the proxy good set $A_\beta(\lambda)$ guarantees that its complement is controlled by the twisted maximal operator acting on the bad set $E_\beta(\lambda)^c$. By the weak type $(1,1)$ or $L^2$ bounds of $M_{tube}$:
\[
\big|A_\beta(\lambda)^c\big| \leq C \big| E_\beta(\lambda)^c \big| = C \big|\big\{\mathbf{x} \in \mathbb{R}^{2m}: U_f^*(\mathbf{x}) > \lambda\big\}\big|.
\]

For the second term, we apply Chebyshev's inequality and substitute the bound \eqref{eq:total_A_beta_bound}:
\begin{align*}
&\big|\big\{\mathbf{x} \in A_\beta(\lambda) : S_{twist}(f)(\mathbf{x}) > \lambda\big\}\big|
 \leq \frac{1}{\lambda^2} \int_{A_\beta(\lambda)} S_{twist}(f)^2 d\mathbf{x} \\
&\leq \frac{1}{\lambda^2} \left( C \lambda^2 \big|\big\{ \mathbf{x} \in \mathbb{R}^{2m} : U_f^*(\mathbf{x}) > \lambda \big\}\big| + C \int_{\{ U_f^*(\mathbf{x}) \le \lambda \}} U_f^*(\mathbf{x})^2 d\mathbf{x} \right) \\
&= C \big|\big\{ \mathbf{x} \in \mathbb{R}^{2m} : U_f^*(\mathbf{x}) > \lambda \big\}\big| + \frac{C}{\lambda^2} \int_{\{ U_f^*(\mathbf{x}) \le \lambda \}} U_f^*(\mathbf{x})^2 d\mathbf{x}.
\end{align*}

Summing the bounds for the two components yields    the Fefferman--Stein type good-$\lambda$ inequality relating the twisted area function to the non-tangential maximal function:
\begin{align*}
\big|\big\{\mathbf{x} \in \mathbb{R}^{2m}: S_{twist}(f)(\mathbf{x}) > \lambda\big\}\big| &\leq C \big|\big\{\mathbf{x} \in \mathbb{R}^{2m}: U_f^*(\mathbf{x}) > \lambda\big\}\big| + \frac{C}{\lambda^2} \int_{\{U_f^*(\mathbf{x}) \le \lambda\}} U_f^*(\mathbf{x})^{2} \, d\mathbf{x}.
\end{align*}
This completes the proof of Theorem \ref{thm:good_lambda}.
\qed

\section{Application of good-$\lambda$ inequality: maximal function and proof of Theorem \ref{secondmainthm}}\label{sec:4}
In this section, we will prove Theorem \ref{secondmainthm}. As argued in the introduction, combined with Theorem \ref{thm:good_lambda}, it remains to verify the $L^1$-norm of twisted non-tangential maximal function acting on every twisted atom is uniformly bounded.

\subsection{Miscellaneous Hardy spaces}

We recall the definition of the twisted atom.
\begin{defn}[\cite{FLLWW}]\label{def atom}
    Fix positive integers $N_j>0$ for $j=1,2,3$. Let $\Diamond={\rm I,I\!I, I\!I\!I, I\!V}$ or ${\rm V}$ be the type. An {\it atom of type} $\Diamond$ is a $L^2(\mathbb{R}^{2m})$ function $a_{\Omega^\Diamond}$ such that there is an open set $\Omega^\Diamond$ of $\mathbb{R}^{2m}$ of finite measure and $L^2(\mathbb{R}^{2m})$ functions $a^\Diamond_R$, called {\it particles}, and $b_R$ in $\text{Dom}(\triangle^{N_1}_1\triangle^{N_2}_2\triangle^{N_3}_{twist})$ for all maximal tubes of type $\Diamond$, $R\in m^\Diamond(\Omega^\Diamond)$, such that
    \begin{align*}
        &(A1)\,\,a^\Diamond_R=\triangle^{N_1}_1\triangle^{N_2}_2\triangle^{N_3}_{twist} b^\Diamond_R\ \ \text{and}\ \ \text{supp}(b_R)\subset R^*,\quad\text{where}\,\,R^*\,\,\text{is the}\,\sigma-\text{enlargement of}\,\,R;\\
        &(A2)\,\,\text{The sum}\,\,\sum_{R\in m^\Diamond(\Omega^\Diamond)}a^\Diamond_R\,\,\text{converges in}\,\,L^2(\mathbb{R}^{2m})\quad\text{and}\,\,\sum_{R\in m^\Diamond(\Omega^\Diamond)}\|a^\Diamond_R\|^2_{L^2(\mathbb{R}^{2m})}\lesssim\frac{1}{|\Omega^\Diamond|};\\
        &(A3)\,\,a_{\Omega^\Diamond}=\sum_{R\in m^\Diamond(\Omega^\Diamond)}a^\Diamond_R\quad\text{and}\quad \|a_{\Omega^\Diamond}\|_{L^2(\mathbb{R}^{2m})}\lesssim\frac{1}{\sqrt{|\Omega^\Diamond|}}. \end{align*}
        Moreover, let $j=1,2,3$, then for all $0\leq k_j\leq N_j$, we have the cancellation property:
        
        1. for $\Diamond={\rm I}$,
\begin{align}\label{cancel1.1}
\sum_{R\in m^\Diamond(\Omega^\Diamond)}\ell(I_1)^{-4k_1}\ell(I_2)^{-4k_2}\left\|\triangle^{N_1-k_1}_1\triangle^{N_2-k_2}_2\triangle^{N_3}_{twist} b^\Diamond_R\right\|^2_{L^2(\mathbb{R}^{2m})}\lesssim\frac{1}{|\Omega^\Diamond|},
\end{align}
and
\begin{align}\label{cancel1.2}
\sum_{R\in m^\Diamond(\Omega^\Diamond)}\ell(I_1)^{-4k_1}\ell(I_2)^{-4k_2}\ell(I_1)^{-\alpha}\ell(I_2)^{-4N_3+\alpha}\left\|\triangle^{N_1-k_1}_1\triangle^{N_2-k_2}_2 b^\Diamond_R\right\|^2_{L^2(\mathbb{R}^{2m})}\lesssim\frac{1}{|\Omega^\Diamond|}
        \end{align}
for $\alpha\in\{0,1,\ldots, 4N_3\}$; 

\smallskip
        2. for $\Diamond={\rm I\!I, I\!I\!I}$, the cancellation property becomes
  \begin{align}
\sum_{R\in m^\Diamond(\Omega^\Diamond)}\ell(I_1)^{-4k_1}\ell(I_2)^{-4k_3}\left\|\triangle^{N_1-k_1}_1\triangle^{N_2}_2\triangle^{N_3-k_3}_{twist} b^\Diamond_R\right\|^2_{L^2(\mathbb{R}^{2m})}\lesssim\frac{1}{|\Omega^\Diamond|},
        \end{align}      
        and
    \begin{align}
\sum_{R\in m^\Diamond(\Omega^\Diamond)}\ell(I_1)^{-4k_1}\ell(I_1)^{-\alpha}\ell(I_2)^{-4N_2+\alpha}\ell(I_2)^{-4k_3}\left\|\triangle^{N_1-k_1}_1\triangle^{N_3-k_3}_{twist} b^\Diamond_R\right\|^2_{L^2(\mathbb{R}^{2m})}\lesssim\frac{1}{|\Omega^\Diamond|}
        \end{align} 
for all $\alpha\in\{0,1,\ldots, 4N_2\}$;

        \smallskip
 
       3. for $\Diamond={\rm I\!V,V}$,
 \begin{align}
\sum_{R\in m^\Diamond(\Omega^\Diamond)}\ell(I_1)^{-4k_3}\ell(I_2)^{-4k_2}\left\|\triangle^{N_1}_1\triangle^{N_2-k_2}_2\triangle^{N_3-k_3}_{twist} b^\Diamond_R\right\|^2_{L^2(\mathbb{R}^{2m})}\lesssim\frac{1}{|\Omega^\Diamond|},
        \end{align}  
        and
\begin{align}
 \sum_{R\in m^\Diamond(\Omega^\Diamond)}\ell(I_1)^{-\alpha}\ell(I_2)^{-4N_1+\alpha}\ell(I_1)^{-4k_3}\ell(I_2)^{-4k_2}\left\|\triangle^{N_2-k_2}_2\triangle^{N_3-k_3}_{twist} b^\Diamond_R\right\|^2_{L^2(\mathbb{R}^{2m})}\lesssim\frac{1}{|\Omega^\Diamond|}
\end{align}  
for all $\alpha\in\{0,1,\ldots, 4N_1\}$.
\end{defn}
\smallskip

Here is the definition of the twisted atomic Hardy space and it is shown that the twisted atomic Hardy space is equivalent to the twisted area Hardy space.   
\begin{defn}[\cite{FLLWW}]\label{def atom decomposition}
We say that $f\in L^1(\mathbb{R}^{2m})$ admits an (twisted) atomic decomposition if the function $f$ can be written as $\sum_{k\in\mathbb{Z}}\sum_{\Diamond={\rm I,I\!I, I\!I\!I, I\!V, V}}\lambda^\Diamond_k\cdot a^\Diamond_k$, where the sum converges in $L^1(\mathbb{R}^{2m})$ and each $a^\Diamond_k$ is an atom of type $\Diamond$ and $\{\lambda_k\}_{k,\Diamond}\in\ell^1$; We write $f\sim\sum_{\substack{k\in\mathbb{Z}\\\Diamond={\rm I,I\!I, I\!I\!I, I\!V, V}}}\lambda^\Diamond_k\cdot a^\Diamond_k$ to indicate that $\sum_{\substack{k\in\mathbb{Z}\\\Diamond={\rm I,I\!I, I\!I\!I, I\!V, V}}}\lambda^\Diamond_k\cdot a^\Diamond_k$ is an atomic decomposition of $f$.
The (twisted) atomic Hardy space $H^1_{Tw,\,atom}(\mathbb{R}^{2m})$ is defined to be all the $L^1(\mathbb{R}^{2m})$ functions $f$ such that $f$ admits an atomic decomposition with the norm
    \begin{align*}        \|f\|_{H^1_{Tw,\,atom}}:=\inf\bigg\{\sum_{k}|\lambda_k|:\,f\sim\sum_{k\in\mathbb{Z}}\sum_{\Diamond=I,I\!\!I,I\!\!I\!\!I,I\!V,V}\lambda^\Diamond_k\cdot a^\Diamond_k\bigg\}.
    \end{align*}
\end{defn}

Recall also that 
  the {\it twisted Lusin area function} is defined as
\begin{equation}\label{eq:area-function}
   S_{area,  \varphi}(f)(\mathbf{x})=\left( \int_{  \mathbb{R}^3_+}|f* \varphi_{\mathbf{r} } |^2* \chi_{\mathbf{r} }(\mathbf{x}) \frac {d\mathbf{r}}{\mathbf{r}}\right)^{\frac 12},
\end{equation}
where $\varphi_j$'s are Schwartz functions on $\mathbb{R}^m$ which has mean value zero for $j=1,2,3.$

\begin{defn}[\cite{FLLWW}, Area function Hardy space]~\\
The {\it area function Hardy space} $H_{area, \varphi}^
1(\mathbb{R}^{2m})$
    is defined to be the
set  of all $f \in L
^1
(\mathbb{R}^{2m})$ such that $   S_{area,  \varphi}(f) \in L
^1
(\mathbb{R}^{2m})$  with the norm
$
   \|f\|_{H_{area,\boldsymbol\varphi}^
1(\mathbb{R}^{2m})}:=\|S_{area,\varphi}(f)\|_{L
^1
(\mathbb{R}^{2m})}
$. 
\end{defn}

\begin{thm}[\cite{FLLWW}] \label{thm:S-atom}  The two spaces  $H^1_{Tw,atom}( \mathbb{R}^{2m}) $ and $ H^1_{area,\varphi}( \mathbb{R}^{2m})$ coincide and they have equivalent norms.
\end{thm}

We now introduce the twisted Hardy space via the twisted area function $S_{twist}(f)$ in terms of the Poisson integrals as follows.
\begin{defn}
The {\it area function Hardy space} via Poisson integrals $H_{area, Poi}^
1(\mathbb{R}^{2m})$
    is defined to be the
set  of all $f \in L
^1
(\mathbb{R}^{2m})$ such that $   S_{twist}(f) \in L
^1
(\mathbb{R}^{2m})$  with the norm
$
   \|f\|_{H_{area,Poi}^
1(\mathbb{R}^{2m})}:=\|S_{twist}(f)\|_{L
^1
(\mathbb{R}^{2m})}
$. 
\end{defn}

\subsection{Twisted non-tangential characterization of Hardy space}
This section contributes to the proof of Theorem \ref{secondmainthm}.\\
We first note that by Theorem \ref{thm:good_lambda}, we obtain that
for every $f\in H^1_{max}(\mathbb{R}^{2m})$,
\begin{align*}
\|S_{twist}(f)\|_{L^1(\mathbb{R}^{2m})} \lesssim \|U_f^*\|_{L^1(\mathbb{R}^{2m})}.
\end{align*}
Next, we will prove the  estimate that for any $f\in H^1_{Tw,atom}( \mathbb{R}^{2m}) $, 
\begin{align}\label{proof key2}
\|U_f^*\|_{L^1(\mathbb{R}^{2m})}
\lesssim 
\|f\|_{H^1_{Tw,atom}( \mathbb{R}^{2m})}.
\end{align}
And eventually, we argue that 
for each $f\in H_{area,Poi}^
1(\mathbb{R}^{2m})$,
$f$ admits an atomic decomposition $$f=\sum_{\substack{k\in\mathbb{Z}\\\Diamond={\rm I,I\!I, I\!I\!I, I\!V, V}}}\lambda^\Diamond_k\cdot a^\Diamond_k$$ similar to Definition \ref{def atom decomposition} with atoms as in Definition \ref{def atom}, and 
\begin{align}\label{proof key3}
\sum_{\substack{k\in\mathbb{Z}\\\Diamond={\rm I,I\!I, I\!I\!I, I\!V, V}}}|\lambda^\Diamond_k|
\lesssim 
\|f\|_{H^1_{area,Poi}( \mathbb{R}^{2m})}.
\end{align}
Then, the proof of Theorem \ref{secondmainthm} is complete, assuming \eqref{proof key2} and \eqref{proof key3}. 

\smallskip
To prove \eqref{proof key3}, we mainly follow the full details in \cite[Theorem $1.9$]{FLLWW}. The only difference is that we need a reproducing formula involving the Poisson kernel. We now provide such a reproducing formula as follows.

 \begin{thm}\label{thm:reproducing}
    There is a  family of functions  $\{\psi^{(j)}\}^3_{j=1}\subset C^\infty_{0}(\mathbb{R}^m)$ such that each $\psi^{(j)}$ has higher order cancellation property, and that for $ f \in    L
^1(\mathbb{R}^{2m})\cap L
^2
( \mathbb{R}^{2m})$, we have
\begin{equation}\label{eq:reproducing} f(\mathbf{x}) =
\int_{\mathbb{R}^3_+} f*\psi_ {\mathbf{r}}* \partial_{\mathbf{r}}(Poi_{twist})_ {\mathbf{r}}(\mathbf{x})\, \frac { d\mathbf{r}}{\mathbf{r}},\,\,\text{where}\,\, \frac {d\mathbf{r}}{\mathbf{r}}=\frac
{dr_1}{r_1}\frac {dr_2}{r_2}\frac {dr_3}{r_3},
 \end{equation}
 where $\partial_{\mathbf{r}}(Poi_{twist})_ {\mathbf{r}}$ is 
defined by
\begin{align}
    \partial_{\mathbf{r}}(Poi_{twist})_\mathbf{r}(\mathbf{y})=(r_1\partial_{r_1}Poi^{(1)}_{r_1}\otimes r_2\partial_{r_2}Poi^{(2)}_{r_2})*_3 \big(r_3\partial_{r_3}Poi^{(3)}_{r_3}(\mathbf{y})\big).
\end{align}
 \end{thm}
\begin{proof}
Let $Q^{(j)}_{r_j}(x) = r_j \partial_{r_j} Poi^{(j)}_{r_j}(x)$ denote the scaled derivative of the Euclidean Poisson kernel for each $j \in \{1, 2, 3\}$.

By the standard one-parameter construction (see \cite[Chapter IV, Section 6.19]{St}), for any required cancellation order $N$, there exists a radial function $\psi^{(j)} \in C_c^\infty(\mathbb{R}^m)$ supported in the unit ball $B(0,1)$ with vanishing moments up to order $N$, such that the continuous Calder\'{o}n condition is satisfied in the frequency domain:
\begin{equation}\label{eq:1-param-reproducing}
    \int_0^\infty \widehat{Q^{(j)}_{r_j}}(\xi) \cdot \widehat{\psi^{(j)}}(r_j\xi) \, \frac{dr_j}{r_j} = 1 \quad \text{for all } \xi \neq 0.
\end{equation}
Define the smooth function $\psi_{\mathbf{r}}$  by
\begin{align*}
    \psi_{\mathbf{r}}(\mathbf{x}) = \left(\psi^{(1)}_{r_1} \otimes \psi^{(2)}_{r_2}\right) *_3 \psi^{(3)}_{r_3}(\mathbf{x}).
\end{align*}

Let $K_{\mathbf{r}} = \psi_{\mathbf{r}} * \partial_{\mathbf{r}}Poi_{\mathbf{r}}$ be the convolution kernel at scale $\mathbf{r}$. Then 
\begin{align*}
    \mathcal{F}_{\mathbf{x}} \left( \int_{\mathbb{R}^3_+} f * K_{\mathbf{r}} \, \frac{d\mathbf{r}}{\mathbf{r}} \right) (\xi_1, \xi_2) = \widehat{f}(\xi_1, \xi_2) \int_{\mathbb{R}^3_+} \widehat{K_{\mathbf{r}}}(\xi_1, \xi_2) \, \frac{dr_1}{r_1} \frac{dr_2}{r_2} \frac{dr_3}{r_3}.
\end{align*}
and remark that, from \eqref{eq:1-param-reproducing} and the Fourier transform,
\begin{align*}
    \int_{\mathbb{R}^3_+} \widehat{K_{\mathbf{r}}}(\xi_1, \xi_2) \, \frac{d\mathbf{r}}{\mathbf{r}} 
    &= \left(\prod^2_{i=1} \int_0^\infty \widehat{Q^{(i)}_{r_i}}(\xi_i) \widehat{\psi^{(i)}}(r_i\xi_i) \, \frac{dr_i}{r_i} \right) \\
    &\quad \times \left( \int_0^\infty \widehat{Q^{(3)}_{r_3}}(\xi_1+\xi_2) \widehat{\psi^{(3)}}(r_3(\xi_1+\xi_2)) \, \frac{dr_3}{r_3} \right),
\end{align*}
for almost every $(\xi_1, \xi_2) \in \mathbb{R}^{2m}$. Therefore, the Plancherel theorem gives the desired formula. 
\end{proof}

\smallskip
To establish \eqref{proof key2}, it suffices to prove the following argument.
\begin{thm} \label{thm:M-atom}  
There exists a positive constant $C$ such that for every atom  $a_{\Omega^\Diamond}$, 
 \begin{align}
\|U_{a_{\Omega^\Diamond}}^*\|_{L^1( \mathbb{R}^{2m})}\leq C.
\end{align}
\end{thm}
\begin{proof}
To prove the uniform $L^1$ boundedness of the twisted maximal function on atoms, we mirror the structural decomposition used for the area function in \cite{FLLWW}. Here is a sketch of the proof. Let $\Diamond \in \{{\rm I, I\!I, I\!I\!I, I\!V, V}\}$ be the geometric type of the tubes, and let $a = a_{\Omega^\Diamond}$ be an atom of type $\Diamond$ supported on the open set $\Omega^\Diamond$ of finite measure. 

By definition, we can decompose the atom into its constituent particles:
\begin{align*}
    a_{\Omega^\Diamond} = \sum_{R \in m^\Diamond(\Omega^\Diamond)} a_R,
\end{align*}
where each $a_R = \triangle^{N_1}_1\triangle^{N_2}_2\triangle^{N_3}_{twist} b_R$, and $\text{supp}(b_R) \subset R^*$, the $\sigma$-enlargement of the maximal tube $R$.

We define the enlarged set $\widetilde{\Omega}^\Diamond = \{ \mathbf{x} \in \mathbb{R}^{2m} : M_{tube}^\Diamond(\chi_{\Omega^\Diamond})(\mathbf{x}) > c \}$ for an appropriately small constant $c$, ensuring $|\widetilde{\Omega}^\Diamond| \lesssim |\Omega^\Diamond|$. For each $R \in m^\Diamond(\Omega^\Diamond)$, we associate a similarly enlarged tube $\widehat{R}$ such that $R^* \subset \widehat{R} \subset \widetilde{\Omega}^\Diamond$, governed by the twisted covering lemma.

We split the $L^1$ norm of $U_{a_{\Omega^\Diamond}}^*$ into the local (interior) and global (exterior) parts and dominate the local part via the $L^2$-mapping property of $U^*$. As a consequence,
\begin{align}\label{eq:U_star_split}
    \| U_{a_{\Omega^\Diamond}}^* \|_{L^1(\mathbb{R}^{2m})} &\le \int_{\widetilde{\Omega}^\Diamond} U_{a_{\Omega^\Diamond}}^*(\mathbf{x}) \, d\mathbf{x} + \int_{(\widetilde{\Omega}^\Diamond)^c} U_{a_{\Omega^\Diamond}}^*(\mathbf{x}) \, d\mathbf{x}\notag\\
    &\leq\big|\widetilde{\Omega}^\Diamond\big|^{1/2} \| U_{a_{\Omega^\Diamond}}^* \|_{L^2(\mathbb{R}^{2m})}  + \int_{(\widetilde{\Omega}^\Diamond)^c} U_{a_{\Omega^\Diamond}}^*(\mathbf{x}) \, d\mathbf{x}\notag
    \\
    &\lesssim |\Omega^\Diamond|^{1/2} \| a_{\Omega^\Diamond} \|_{L^2(\mathbb{R}^{2m})}+  \int_{(\widetilde{\Omega}^\Diamond)^c} U_{a_{\Omega^\Diamond}}^*(\mathbf{x}) \, d\mathbf{x}\notag\\
    &\lesssim1+  \int_{(\widetilde{\Omega}^\Diamond)^c} U_{a_{\Omega^\Diamond}}^*(\mathbf{x}) \, d\mathbf{x},
\end{align}
where the last inequality follows from the size condition $(A3)$ of the atom. Therefore, it suffices to show the uniform bound for the exterior part, that is, 
\begin{align}\label{GOAL}
    \int_{(\widetilde{\Omega}^\Diamond)^c} U_{a_{\Omega^\Diamond}}^*(\mathbf{x}) \, d\mathbf{x} \lesssim1.
\end{align}
\noindent\textbf{$\circledast$ The exterior estimate}\\
We decompose the exterior integral and get
\begin{align*}
    \int_{(\widetilde{\Omega}^\Diamond)^c} U_{a_{\Omega^\Diamond}}^*(\mathbf{x}) \, d\mathbf{x} &\le \sum_{R \in m^\Diamond(\Omega^\Diamond)} \int_{(\widehat{R})^c} U_{a_R}^*(\mathbf{x}) \, d\mathbf{x}.
\end{align*}
To bound this sum uniformly by a constant, we leverage the cancellation conditions $(A1)$ of the particles to extract polynomial decay from the twisted Poisson kernel $(Poi_{twist})_{\mathbf{r}}$.

Remark that the twisted maximal function of a single particle is, by definition,
\begin{align*}
    U_{a_R}^*(\mathbf{x}) = \sup_{(\mathbf{y}, \mathbf{r}) \in \Gamma(\mathbf{x})} \big| a_R * (Poi_{twist})_{\mathbf{r}}(\mathbf{y}) \big|,
\end{align*}
and $a_R = \triangle^{N_1}_1\triangle^{N_2}_2\triangle^{N_3}_{twist} b_R$, then we can apply the property of the bump function via integration by parts to obtain
\begin{align*}
    a_R * (Poi_{twist})_{\mathbf{r}}(\mathbf{y}) &= b_R * \Big( \triangle^{N_1}_1\triangle^{N_2}_2\triangle^{N_3}_{twist} (Poi_{twist})_{\mathbf{r}} \Big)(\mathbf{y}).
\end{align*}
Applying these spatial derivatives generates polynomial decay away from the center of the kernel. The analysis of this decay diverges depending on the type $\Diamond$.
\smallskip

$\clubsuit$ \textbf{Case} $\Diamond = {\rm I}$:

For type ${\rm I}$ tubes, the base is a standard Cartesian rectangle $R = I_1 \times I_2 \in m^{\rm I}(\Omega^{\rm I})$. To isolate the decay across the independent coordinate directions, we partition the exterior integration domain $(\widehat{R})^c = (100\widehat{I_1} \times 100\widehat{I_2})^c$ into four mutually disjoint regions based on their separation from the support of the particle $b_R$:
\begin{align*}
    \int_{(\widehat{R})^c} U_{a_R}^*(\mathbf{x}) \, d\mathbf{x} &\le \int_{(100\widehat{I_1})^c \times 100I_2} U_{a_R}^*(\mathbf{x}) \, d\mathbf{x} + \int_{(100\widehat{I_1})^c \times (100I_2)^c} U_{a_R}^*(\mathbf{x}) \, d\mathbf{x} \\
    &\quad + \int_{(100I_1)^c \times (100\widehat{I_2})^c} U_{a_R}^*(\mathbf{x}) \, d\mathbf{x} + \int_{100I_1 \times (100\widehat{I_2})^c} U_{a_R}^*(\mathbf{x}) \, d\mathbf{x} \\
    &=: (\mathcal{A}) + (\mathcal{B}) + (\mathcal{C}) + (\mathcal{D}).
\end{align*}

 Since  $(Poi_{twist})_{\mathbf{r}}$ acts independently in the $y_1$ and $y_2$ coordinates, applying $\triangle_j^{N_j}$ to the kernel generates a pointwise spatial decay of $O\left( \frac{r_j^{N_j}}{(r_j + |y_j|)^{m+2N_j}} \right)$ for each $j=1,2$. 
 
For any $\mathbf{x}$ outside the enlarged dyadic tubes, the non-tangential condition $(\mathbf{y}, \mathbf{r}) \in \Gamma(\mathbf{x})$ restricts $|x_j - y_j| \le r_j$. Consequently, for any $u_j \in I_j$ (the support of $b_R$), the triangle inequality guarantees $r_j + |y_j - u_j| \simeq |x_j - c_{I_j}|$. Thus, shifting $\triangle_j^{N_j}$ onto the kernel yields the uniform pointwise decay:
\begin{equation}\label{eq:poisson_decay}
    \sup_{(\mathbf{y}, \mathbf{r}) \in \Gamma(\mathbf{x})} \left| \left( \triangle_j^{N_j} Poi^{(j)}_{r_j} \right)(y_j - u_j) \right| \lesssim \frac{\ell(I_j)^{2N_j}}{|x_j - c_{I_j}|^{m+2N_j}}.
\end{equation}

\medskip
\noindent\textbf{Estimate for region $(\mathcal{D})$: }\\
In region $(\mathcal{D})$, we integrate over $100I_1 \times (100\widehat{I_2})^c$. Since $x_2$ is far from the support, we shift $\triangle_2^{N_2}$ onto the kernel and use the decay from \eqref{eq:poisson_decay}. The remaining operators stay on the particle, leaving $\triangle_1^{N_1}\triangle_{twist}^{N_3} b_R$. 

We apply the Cauchy--Schwarz inequality to the $x_1$ integration over the local domain $100I_1$ to get
\begin{align*}
    (\mathcal{D})
    &\le |100I_1|^{1/2} \int_{(100\widehat{I_2})^c} \left( \int_{100I_1} |U_{a_R}^*(x_1, x_2)|^2 \, dx_1 \right)^{1/2} dx_2 \\
    &\lesssim |I_1|^{1/2} \int_{(100\widehat{I_2})^c} \frac{\ell(I_2)^{2N_2}}{|x_2 - c_{I_2}|^{m+2N_2}} \big\| \triangle_1^{N_1}\triangle_{twist}^{N_3} b_R \big\|_{L^2(\mathbb{R}^{2m})} \, dx_2.
\end{align*}

To evaluate the $x_2$ integral, we partition $(100\widehat{I_2})^c$ into concentric dyadic annuli $E_j = \delta_{2^j}(100I_2) \setminus \delta_{2^{j-1}}(100I_2)$. Since the integration begins outside $\widehat{I_2}$, the summation starts at $j_0 = \log_2(\ell(\widehat{I_2})/\ell(I_2))$. Summing the geometric series yields:
\begin{align*}
    \int_{(100\widehat{I_2})^c} \frac{\ell(I_2)^{2N_2}}{|x_2 - c_{I_2}|^{m+2N_2}} \, dx_2 &= \sum_{j=j_0}^\infty \int_{E_j} \frac{\ell(I_2)^{2N_2}}{|x_2 - c_{I_2}|^{m+2N_2}} \, dx_2 \\
    &\lesssim \sum_{j=j_0}^\infty \big( 2^j \ell(I_2) \big)^m \frac{\ell(I_2)^{2N_2}}{(2^j \ell(I_2))^{m+2N_2}} \\
   &\simeq\ell(I_2)^{m/2} \left( \frac{\ell(I_2)}{\ell(\widehat{I_2})} \right)^{2N_2}.
\end{align*}
Substituting this back into the bound for $(\mathcal{D})$, and combining $|I_1|^{1/2} \ell(I_2)^{m/2} \simeq |R|^{1/2}$, we obtain:
\begin{align*}
    (\mathcal{D}) \lesssim |R|^{1/2} \left( \frac{\ell(I_2)}{\ell(\widehat{I_2})} \right)^{2N_2} \big\| \triangle_1^{N_1}\triangle_{twist}^{N_3} b_R \big\|_{L^2(\mathbb{R}^{2m})}.
\end{align*}

We now sum this estimate over all type ${\rm I}$ rectangles. Applying the Cauchy--Schwarz inequality over the sum $\sum_{R}$, and utilizing the standard covering lemma alongside the specific atomic cancellation condition \eqref{cancel1.1}, we have that
\begin{align*}
    \sum_{R \in m^{\rm I}(\Omega^{\rm I})} (\mathcal{D}) &\le \left( \sum_{R} |R| \left( \frac{\ell(I_2)}{\ell(\widehat{I_2})} \right)^{4N_2} \right)^{1/2} \left( \sum_{R} \big\| \triangle_1^{N_1}\triangle_{twist}^{N_3} b_R \big\|_{L^2(\mathbb{R}^{2m})}^2 \right)^{1/2} \\
    &\lesssim \big( C |\Omega^{\rm I}| \big)^{1/2} \cdot \left( \frac{1}{|\Omega^{\rm I}|} \right)^{1/2} \\
    &= C.
\end{align*}

\medskip
\noindent\textbf{Estimate for region $(\mathcal{A})$: }\\
Region $(\mathcal{A})$ is geometrically symmetric to $(\mathcal{D})$. Estimates on this region can be made using a similar approach used in region $(\mathcal{D})$.

\medskip
\noindent\textbf{Estimate for regions $(\mathcal{B})$ and $(\mathcal{C})$:}\\
In regions $(\mathcal{B})$ and $(\mathcal{C})$, the spatial variable $\mathbf{x}$ is far from the support $R$ in \textit{both} coordinate blocks simultaneously.

We only estimate the integral over $(\mathcal{B})$. The estimate of the integral on the region $(\mathcal{C})$ is parallel.
For this region $(\mathcal{B}) = (100\widehat{I_1})^c \times (100I_2)^c$, we shift both $\triangle_1^{N_1}$ and $\triangle_2^{N_2}$ onto the kernel. This yields simultaneous pointwise decay in both variables:
\begin{align*}
    U_{a_R}^*(x_1, x_2) \lesssim \frac{\ell(I_1)^{2N_1}}{|x_1 - c_{I_1}|^{m+2N_1}} \frac{\ell(I_2)^{2N_2}}{|x_2 - c_{I_2}|^{m+2N_2}} \big\| \triangle_{twist}^{N_3} b_R \big\|_{L^1(\mathbb{R}^{2m})}.
\end{align*}
By H\"older's inequality on the compact support of the atom, $$\| \triangle_{twist}^{N_3} b_R \|_{L^1(\mathbb{R}^{2m})} \le |R|^{1/2} \| \triangle_{twist}^{N_3} b_R \|_{L^2(\mathbb{R}^{2m})}.$$
We integrate this pointwise bound  over $x_1 \in (100\widehat{I_1})^c$ and $x_2 \in (100I_2)^c$:
\begin{align*}
    (\mathcal{B}) &\lesssim |R|^{1/2} \big\| \triangle_{twist}^{N_3} b_R \big\|_{L^2(\mathbb{R}^{2m})} \left( \sum_{i=i_0}^\infty 2^{-2i N_1} \ell(I_1)^{m/2} \right) \left( \sum_{j=0}^\infty 2^{-2j N_2} \ell(I_2)^{m/2} \right) \\
    &\simeq |R| \left( \frac{\ell(I_1)}{\ell(\widehat{I_1})} \right)^{2N_1}  \cdot \big\| \triangle_{twist}^{N_3} b_R \big\|_{L^2(\mathbb{R}^{2m})}.
\end{align*}
To apply the covering lemma and finish the estimate, we apply Cauchy-Schwarz inequality over the sum of $R$ and apply the size condition of the atom. 
\smallskip

$\clubsuit$ \textbf{Case $\Diamond = {\rm I\!I}$ :}

For type ${\rm I\!I}$ tubes, the atom is supported on the $\sigma$-enlargement of a maximal slanted rectangle $R = I_1 \times_t I_2 \in m^{\rm I\!I}(\Omega^{\rm I\!I})$, where the geometry enforces $|I_1| \le |I_2|$. Following the twisted covering lemma, we split the sum over the particles into Case 1 and Case 2 based on the relative size of the overlapping tubes. (see \cite[Chapter $6$]{FLLWW}).

Similar to the type $\rm I$ case, we estimate the exterior part by further decomposing $(\widehat{R})^c$ into the sum of four parts, $(\mathfrak{A})$, $(\mathfrak{B})$, $(\mathfrak{C})$ and $(\mathfrak{D})$.

The most analytically demanding region occurs in the exterior estimate when both twisted coordinates are far from the support of the atom. We isolate the integration over the twisted complement $\mathfrak{C} := (100I_1)^c \times_t (100\widehat{I_2})^c$. 
To execute the bound, we partition $\mathfrak{C}$ into dyadic slanted annuli sets:
\begin{align*}
    A_{i,j} := \Big\{ \mathbf{x} \in \mathbb{R}^{2m} : |x_1 - x_2 - c_{I_1}| \simeq 2^i \ell_1, \,\, |x_2 - c_{I_2}| \simeq 2^j \ell_2 \Big\},
\end{align*}
where the exterior conditions force the summation indices to be $i \ge 6$ and $j \ge j_0+l$ (with $2^{j_0m} \simeq |\widehat{I_2}|/|I_2|$).

We bound the integral $\sum_{i,j} \int_{A_{i,j}} U_{a_R}^*(\mathbf{x}) d\mathbf{x}$. We unpack the non-tangential maximal function of the particle:
\begin{align*}
    U_{a_R}^*(\mathbf{x}) &= \sup_{(\mathbf{y}, \mathbf{r}) \in \Gamma(\mathbf{x})} \big| a_R * (Poi_{twist})_{\mathbf{r}}(\mathbf{y}) \big|.
\end{align*}
Since the particle is defined as $a_R = \triangle_1^{N_1} \triangle_2^{N_2} \triangle_{twist}^{N_3} b_R$, we can shift the Cartesian derivatives $\triangle_1^{N_1}$ and $\triangle_2^{N_2}$ directly onto the corresponding Euclidean blocks of the Poisson kernel via integration by parts. However, shifting the twisted derivative $\triangle_{twist}^{N_3}$ onto the kernel breaks the symmetry. Thus, we leave $\triangle_{twist}^{N_3}$ on the function $b_R$ and evaluate the twisted convolution $*_3$ associated with the fiber parameter $r_3$.

Let $\Psi_{r_3}(\mathbf{u}) := \big( \triangle_{twist}^{N_3} b_R *_3 Poi^{(3)}_{r_3} \big)(\mathbf{u})$. The core difficulty is that the twisted convolution $*_3$ destroyed  the compact support of $b_R$ along the diagonal direction $\mathbf{u} = (u_1 - v, u_2 - v)$. 
While the twisted difference $u_1 - u_2$ is preserved (meaning that $|u_1 - u_2 - c_{I_1}| \le \ell_1$ holds    for the support of $\Psi_{r_3}$), the second coordinate $u_2$ is no longer compactly supported on $I_2$. 

To bypass this geometric mismatch between the  atom and the Cartesian bounds of the remaining Poisson kernel, we deploy the {\it Molecule decomposition technique}. We partition the spatial domain $\mathbb{R}^{2m}$ of the inner convolution into slanted annuli sets $S_k(i)$ defined by the twisted metrics:
\begin{align*}
    S_k(i) := \Big\{ \mathbf{u} \in \mathbb{R}^{2m} : |u_1 - u_2 - c_{I_1}| \le \frac{|x_1 - x_2 - c_{I_1}|}{2}, \,\, |u_2 - c_{I_2}| \simeq 2^{k+i} \ell_2 \Big\}.
\end{align*}

By isolating the convolution within these slanted annuli, we write the maximal function bound as:
\begin{align*}
    U_{a_R}^*(\mathbf{x}) \le \sup_{(\mathbf{y}, \mathbf{r}) \in \Gamma(\mathbf{x})} \sum_{k=0}^\infty \Bigg| \int_{S_k(i)} \Psi_{r_3}(\mathbf{u}) \cdot \Big( \triangle_1^{N_1} Poi^{(1)}_{r_1}(y_1 - u_1) \triangle_2^{N_2} Poi^{(2)}_{r_2}(y_2 - u_2) \Big) \, d\mathbf{u} \Bigg|.
\end{align*}

We now extract the pointwise spatial decay. Since we are taking the supremum over $(\mathbf{y}, \mathbf{r}) \in \Gamma(\mathbf{x})$, the cone condition ensures that the spatial separation is controlled by $\mathbf{x}$. 
For $\mathbf{x} \in A_{i,j}$ and $\mathbf{u} \in S_k(i)$, the distance in the second coordinate is bounded below by the dominant annulus:
\begin{align*}
    |y_2 - u_2| \ge |u_2 - c_{I_2}| - |x_2 - y_2| - |x_2 - c_{I_2}| \simeq 2^{k+i} \ell_2.
\end{align*}
Similarly, by the triangle inequality and the preservation of the twisted metric on $S_k(i)$:
\begin{align*}
    |y_1 - u_1| \ge |y_2 - u_2| - \big| (y_1 - y_2 - c_{I_1}) - (u_1 - u_2 - c_{I_1}) \big| \simeq 2^{k+i} \ell_2.
\end{align*}
Applying $\triangle_j^{N_j}$ to the standard Poisson kernel generates a function whose supremum over $r_j > 0$ yields spatial polynomial decay:
\begin{align*}
    \sup_{r_j > 0} \left| \triangle_j^{N_j} Poi^{(j)}_{r_j}(y_j - u_j) \right| \lesssim \sup_{r_j > 0} \frac{r_j^{N_j}}{(r_j + |y_j - u_j|)^{m+2N_j}} \lesssim \frac{1}{|y_j - u_j|^{m+N_j}}.
\end{align*}
Thus, evaluating the Poisson kernels for $\mathbf{u} \in S_k(i)$ guarantees the joint decay factor:
\begin{align*}
    \big| \triangle_1^{N_1} Poi^{(1)}_{r_1}(y_1 - u_1) \triangle_2^{N_2} Poi^{(2)}_{r_2}(y_2 - u_2) \big| \lesssim \frac{1}{(2^{k+i}\ell_2)^{2m+N_1+N_2}}.
\end{align*}

We substitute this decay into the integral over $S_k(i)$ and apply the Cauchy--Schwarz inequality. We bound the $L^2$ norm of the function $\Psi_{r_3}$ on $S_k(i)$ by recognizing that $|u_2 - c_{I_2}| \simeq 2^{k+i}\ell_2$ forces the integration variable of the $*_3$ convolution to be far from the origin, bringing out the decay of $Poi^{(3)}_{r_3}$:
\begin{align*}
    \int_{S_k(i)} |\Psi_{r_3}(\mathbf{u})| \, d\mathbf{u} &\le |S_k(i)|^{1/2} \| \Psi_{r_3} \|_{L^2(S_k(i))} \\
    &\lesssim \big( \ell_1 (2^{k+i}\ell_2)^{m-1} \big)^{1/2} \frac{1}{(2^{k+i}\ell_2)^{N_3}} \| \triangle_{twist}^{N_3} b_R \|_{L^2(\mathbb{R}^{2m})}.
\end{align*}
Combining the Poisson decay and the $L^2$ norm, and summing the geometric series over $k \ge 0$, we split the variables entirely. By choosing $N_1, N_2, N_3$ sufficiently large, the summation converges, and we extract the explicit decay parameters corresponding to the dyadic scales $i$ and $j$:
\begin{align*}
    \int_{A_{i,j}} U_{a_R}^*(\mathbf{x}) \, d\mathbf{x} &\le |A_{i,j}| \sup_{\mathbf{x} \in A_{i,j}} U_{a_R}^*(\mathbf{x}) \\
    &\lesssim (2^i \ell_1) (2^j \ell_2)^{m-1} \sum_{k=0}^\infty \frac{1}{(2^{k+i}\ell_2)^{2m+N_1+N_2}} \frac{\big( \ell_1 (2^{k+i}\ell_2)^{m-1} \big)^{1/2}}{(2^{k+i}\ell_2)^{N_3}} \| \triangle_{twist}^{N_3} b_R \|_{L^2(\mathbb{R}^{2m})} \\
    &\lesssim |R|^{1/2} 2^{-i \gamma_1} 2^{-j \gamma_2} \| \triangle_{twist}^{N_3} b_R \|_{L^2(\mathbb{R}^{2m})},
\end{align*}
where $\gamma_1, \gamma_2 > 0$ depend on the number of cancellation conditions $N_1, N_2, N_3$.

Finally, summing these bounds over the external dyadic grid $i \ge 6$ and $j \ge j_0+l$ evaluates the geometric series at its lower bounds. Since $2^{j_0 m} \simeq |\widehat{I_2}|/|I_2|$, the $j$-summation explicitly extracts the target decay factor required by the type ${\rm I\!I}$ covering lemma:
\begin{align*}
    \int_{\mathfrak{C}} U_{a_R}^*(\mathbf{x}) \, d\mathbf{x} \lesssim |R|^{1/2} \left( \frac{\ell(\widehat{I_2})}{\ell(I_2)} \right)^{-\kappa} \| \triangle_{twist}^{N_3} b_R \|_{L^2(\mathbb{R}^{2m})}.
\end{align*}
Combining this with the Cauchy--Schwarz inequality over all $R \in m^{\rm I\!I}(\Omega^{\rm I\!I})$ and summing via the twisted covering lemma, we absorb the volumetric factors and bound the integral by a uniform constant $C$. The remaining spatial configurations for the twisted complement, as well as the structures for types ${\rm I\!I\!I, I\!V, V}$, follow by identical structural symmetries.
\end{proof}

\bigskip
\bigskip

\noindent {\bf Acknowledgement:} J. Li is supported by ARC DP 260100485. C.-W. Liang  is supported by NSTC through grant 111-2115-M-002-010-MY5.   Q. Wu is supported by the National Natural Science Foundation of China (No. 12171221) and Taishan Scholars Program for Young Experts of Shandong Province (tsqn202507265).

\bigskip

(J. Li) School of Mathematical and Physical Sciences, Macquarie University, NSW, 2109, Australia\\ 
{\it E-mail}: \texttt{ji.li@mq.edu.au}\\

(C.-W. Liang) Department of Mathematics, National Taiwan University, Taiwan.\\ 
{\it E-mail}: \texttt{d10221001@ntu.edu.tw}\\

(C. Wen) Department of Mathematics, Sun Yat-sen University, Guangzhou, 510275, China.\\ 
{\it E-mail}: \texttt{wenchj@mail2.sysu.edu.cn}\\

(Q. Wu)  School of Mathematics and Statistics, Linyi University, Linyi 276000, China.\\
{\it E-mail}: \texttt{wuqingyan@lyu.edu.cn}\\
\end{document}